\documentclass[12pt]{amsart}
\usepackage[colorlinks=true, pdfstartview=FitV, linkcolor=blue, citecolor=blue, urlcolor=blue, breaklinks=true]{hyperref}
\usepackage{amsmath,amsfonts,amssymb,amsthm,amscd,array,comment,euscript,mathtools,etoolbox}
\usepackage{paralist}
\usepackage[usenames]{color}
\usepackage[all]{xy}

\newcommand{\arxiv}[1]{\href{http://arxiv.org/abs/#1}{\tt arXiv:\nolinkurl{#1}}}

\newtoggle{comments}
\newtoggle{details}
\newtoggle{detailsnote}

\togglefalse{comments}     % For no comments

\togglefalse{details}      % For no details

\toggletrue{detailsnote}   % For including note about details toggle
\newcommand\C{\mathbb{C}}
\newcommand\Z{\mathbb{Z}}
\newcommand\Q{\mathbb{Q}}

\newcommand\N{\mathbb{N}}
\newcommand\bbA{\mathbb{A}}
\newcommand\kk{{\Bbbk}}

\newcommand\rat{\mathrm{rat}}
\newcommand\ev{\mathrm{ev}}
\newcommand\hev{\mathrm{hev}}

\newcommand\g{\mathfrak{g}}
\newcommand\h{\mathfrak{h}}
\newcommand\n{\mathfrak{n}}
\newcommand\fa{\mathfrak{a}}

\newcommand\sm{\mathsf{m}}
\newcommand\sM{\mathsf{M}}

\newcommand\mm{\mathbb{M}}

\newcommand\cF{\mathcal{F}}

\newcommand\cB{\mathcal{B}}

\newcommand\cI{\mathcal{I}}
\newcommand\cE{\mathcal{E}}

\newcommand\cU{\mathcal{U}}

\newcommand\bx{\mathbf{x}}

\newcommand\bA{\mathbf{A}}
\newcommand\bW{\mathbf{W}}
\newcommand\bV{\mathbf{V}}
\newcommand\bR{\mathbf{R}}
\newcommand\bi{\mathbf{i}}

\newcommand\bT{\mathbf{T}}

\newcommand{\hei}{\operatorname{ht}}
\newcommand{\ts}{\textstyle}
\newcommand*{\defeq}{\mathrel{\vcenter{\baselineskip0.5ex \lineskiplimit0pt
                     \hbox{\scriptsize.}\hbox{\scriptsize.}}}%
                     =}
\newcommand{\llift}{{\lambda'}} % Notation for element of weight lattice of \g (as oppposed to \g^\Gamma)

\DeclareMathOperator{\Hom}{Hom}
\DeclareMathOperator{\Ext}{Ext}
\DeclareMathOperator{\rank}{rank}

\DeclareMathOperator{\Span}{Span}
\DeclareMathOperator{\Spec}{Spec}
\DeclareMathOperator{\maxSpec}{maxSpec}

\DeclareMathOperator{\Id}{Id}
\DeclareMathOperator{\Ann}{Ann}
\DeclareMathOperator{\Supp}{Supp} % Support
\DeclareMathOperator{\Ob}{Ob}
\DeclareMathOperator{\Mor}{Mor}
\DeclareMathOperator{\ad}{ad}
\DeclareMathOperator{\sym}{sym}
\DeclareMathOperator{\wt}{wt}

\DeclareMathOperator{\rad}{rad}

\DeclareMathOperator{\Irr}{Irr}

\theoremstyle{plain}
\newtheorem{theo}{Theorem}[section]
\newtheorem*{theo*}{Theorem}
\newtheorem{prop}[theo]{Proposition}
\newtheorem{lem}[theo]{Lemma}
\newtheorem{cor}[theo]{Corollary}

\theoremstyle{definition}
\newtheorem{defin}[theo]{Definition}
\newtheorem*{rem*}{Remark}
\newtheorem{rem}[theo]{Remark}

\newtheorem{example}[theo]{Example}
\newtheorem{assumption}[theo]{Assumption}

\numberwithin{equation}{section}

\allowdisplaybreaks

\iftoggle{comments}{
\newcommand{\comments}[1]{ \begin{center} \parbox{5 in}{{\bf {\footnotesize Comments:  }}{\footnotesize \textit{#1}}} \end{center}}
}{
\newcommand{\comments}[1]{}
}

\iftoggle{details}{
\newcommand{\details}[1]{\smallskip \color{blue} \begin{footnotesize} \textbf{Details:} #1 \end{footnotesize} \color{black}}
}{
\newcommand{\details}[1]{}
}

\begin{document}
%
%%%%%%%%%%%%%%%%%%%%%%%%%%%%%%%%%%%

\title{Global Weyl modules for equivariant map algebras}

\author{Ghislain Fourier}
\address{G.~Fourier: Mathematisches Institut der Universit\"at zu K\"oln}
\email{gfourier@math.uni-koeln.de}
\author{Nathan Manning}
\address{N.~Manning: Department of Mathematics and Statistics, University of Ottawa}
\email{nmanning@uottawa.ca}
\author{Alistair Savage}
\address{A.~Savage: Department of Mathematics and Statistics, University of Ottawa}
\email{alistair.savage@uottawa.ca}
\thanks{The first author was supported by Deutsche Forschungsgemeinschaft Priority Program 1388 (Representation Theory).  The research of the third author was supported by a Discovery Grant from the Natural Sciences and Engineering Research Council of Canada.  The second author was partially supported by the Discovery Grant of the third author.}

\subjclass[2010]{17B65, 17B10}

%\date{\today}

\keywords{Equivariant map algebra, loop algebra, Weyl module, representation, module}

\begin{abstract}
  Equivariant map algebras are Lie algebras of algebraic maps from a scheme (or algebraic variety) to a \emph{target} finite-dimensional Lie algebra (in the case of the current paper, we assume the latter is a simple Lie algebra) that are equivariant with respect to the action of a finite group.  In the first part of this paper, we define global Weyl modules for equivariant map algebras satisfying a mild assumption.  We then identify a commutative algebra $\bA^\lambda_\Gamma$ that acts naturally on the global Weyl modules, which leads to a Weyl functor from the category of $\bA^\lambda_\Gamma$-modules to the category of modules for the equivariant map algebra in question.  These definitions extend the ones previously given for generalized current algebras (i.e.\ untwisted map algebras) and twisted loop algebras.

  In the second part of the paper, we restrict our attention to equivariant map algebras where the group involved is abelian, acts on the target Lie algebra by diagram automorphisms, and freely on (the set of rational points of) the scheme.  Under these additional assumptions, we prove that $\bA^\lambda_\Gamma$ is finitely generated and the global Weyl module is a finitely generated $\bA^\lambda_\Gamma$-module.  We also define \emph{local} Weyl modules via the Weyl functor and prove that these coincide with the local Weyl modules defined directly in \cite{FKKS12}.  Finally, we show that $\bA^\lambda_\Gamma$ is the algebra of coinvariants of the analogous algebra in the untwisted case.
\end{abstract}

\maketitle \thispagestyle{empty}

\tableofcontents

%%%%%%%%%%%%%%%%%%%%%%%%%%%%%%%%%%%%%%%%%%%%%%%%%%%%%%%%%%%%%%%%%%%
%
\section{Introduction}
%
%%%%%%%%%%%%%%%%%%%%%%%%%%%%%%%%%%%%%%%%%%%%%%%%%%%%%%%%%%%%%%%%%%%

Weyl modules for the loop algebra $\g \otimes \mathbb{C}[t^{\pm 1}]$ of a finite-dimensional simple complex Lie algebra $\g$ were introduced by Chari and Pressley more than a decade ago (see \cite{CP01}).  Since then, the study of their properties (for instance, their homological behavior, dimension and character) has been a fruitful and successful process.

The category of finite-dimensional modules for $\g \otimes  \mathbb{C}[t^{\pm 1}]$ is not semisimple. In analogy with the modular representation theory of simple finite-dimensional Lie algebras, for every simple module, there exists a \emph{(local) Weyl module} satisfying certain universal properties.  This Weyl module is finite-dimensional and its character and dimension have been studied and computed in a series of papers (see \cite{CP01,CL06,FL07,Nao12}).  Local Weyl modules have been identified with certain Demazure modules of affine Kac--Moody algebras and their characters are also known to be characters of the $q \to 1$ limit of simple modules of the quantum affine algebra (see \cite{FL07}).

In \cite{CP01}, the class of \emph{global Weyl modules} was defined.  These modules are projective objects in the category of those $\g \otimes \mathbb{C}[t^{\pm 1}]$-modules whose weights are bounded by some fixed dominant integral $\g$-weight. Their $\g$-weight spaces are right modules over polynomial rings in finitely many variables.  It was conjectured in \cite{CP01}, and it can be deduced from results in the aforementioned series of papers, that the global Weyl module is a free right module of finite rank for this polynomial ring (see Theorem~\ref{theo:global-Weyl-module-projective}).

It turns out that the global Weyl module might be the most interesting object to study in the category of bounded $\g \otimes \mathbb{C}[t^{\pm 1}]$-modules.  For instance, it is subject to an analog of the Bernstein--Gelfand--Gelfand reciprocity for simple Lie algebras (see \cite{BBCKL12,BCM12}).  Furthermore, its character is known to be the $q$-Whittaker function, a solution to the $q$-Toda integrable system (see \cite{BF12}).

There are several approaches to generalizing the above objects. In \cite{FL04}, local and global Weyl modules were defined in the setting where $\mathbb{C}[t^{\pm 1}]$ is replaced by the coordinate ring of a complex affine variety.  A more general approach was taken in \cite{CFK10}.  There the modules for a generalized current algebra $\g \otimes A$, where $A$ is a commutative, associative complex unital algebra were studied.  The global Weyl module is again a projective object in a suitable category (see Corollary~\ref{cor:global-Weyl-module-free}) and its weight spaces are right modules for a certain commutative algebra. The \emph{Weyl functor} was introduced and local Weyl modules were studied, together with their homological properties. The algebra of the highest weight space was analyzed and in important cases identified with a tensor product of symmetric powers of $A$.

A different approach was taken in \cite{CFS08} and \cite{FMS11}, where $\g \otimes \mathbb{C}[t^{\pm 1}]$ was replaced by the twisted loop algebra $(\g \otimes \mathbb{C}[t^{\pm 1}])^\Gamma$.  This is the fixed point algebra of $\g \otimes \mathbb{C}[t^{\pm 1}]$ under the action of a group $\Gamma$ of automorphisms of $\g$, generated by a Dynkin diagram automorphism.  This group acts on $\mathbb{C}[t^{\pm 1}]$ by scaling $t$ by roots of unity.  It turns out that every local Weyl module of the twisted loop algebra is obtained by restriction from a local Weyl module of $\g \otimes \mathbb{C}[t^{\pm 1}]$. The global Weyl module was defined in \cite{FMS11} and it was shown that it is again a free right module of finite rank for a certain commutative algebra and it can be embedded in a direct sum of global Weyl modules for $\g \otimes \mathbb{C}[t^{\pm 1}]$. In \cite[Remark 1.10]{BF12} it is conjectured that its character solves the $q$-Toda integrable system in the nonsimply laced case.

In \cite{FKKS12}, the definition of local Weyl modules was generalized to the setting of \emph{equivariant map algebras}. Let $X =\Spec A$ be an affine scheme and $\Gamma$ be a finite group acting on $\g$ and $X$ by automorphisms. The equivariant map algebra is the Lie algebra of equivariant algebraic maps from $X$ to $\g$ and is denoted $(\g \otimes A)^\Gamma$. Local Weyl modules were defined for these algebras under the assumptions that $X$ is of finite type, $\Gamma$ is an abelian group, and the action on $X$ is free.  A key ingredient in this study was the definition of certain \emph{twisting} and \emph{untwisting} functors that relate the representation theory of $\g \otimes A$ and $(\g \otimes A)^\Gamma$.  It was also shown that the homological properties of local Weyl modules can be generalized to the setting of equivariant map algebras.

In the current paper, we define \emph{global} Weyl modules for equivariant map algebras.  The paper can be divided into two parts.  The first part comprises Sections~\ref{sec:EMAs} to~\ref{sec:Weyl-functor}.  After some preliminaries on equivariant map algebras in~Section~\ref{sec:EMAs}, we define in Section~\ref{sec:Weyl-modules} global Weyl modules for equivariant map algebras satisfying a mild assumption (see Assumption~\ref{assumption}).  In particular, this assumption is always satisfied if $\Gamma$ is cyclic or acts on $\g$ by diagram automorphisms.  We also give a presentation of the global Weyl modules in terms of generators and relations (Proposition~\ref{weyl:gen-rel}).  In Section~\ref{sec:Weyl-functor}, we extend the notion of Weyl functors to the twisted/equivariant setting.  In particular, we define a commutative algebra $\bA^\lambda_\Gamma$ which acts naturally on the global Weyl module with highest $\g^\Gamma$-weight $\lambda$.  The Weyl functor is then a functor from the category of $\bA^\lambda_\Gamma$-modules to the category of $(\g \otimes A)^\Gamma$-modules.  We show that these functors (and the global Weyl modules) possess twisted versions of properties satisfied in the untwisted setting.

The second part of the current paper (Sections~\ref{sec:properties-global}--\ref{sec:A-lambda-Gamma}) concerns equivariant map algebras for which $\Gamma$ is abelian, acts on $\g$ by Dynkin diagram automorphisms, and on $\maxSpec A$ freely.  (Note that $\Gamma$ is abelian and acts freely on $\maxSpec A$ in the case of (twisted) loop and multiloop algebras.)  Under these additional assumptions, our main results are the following.
\begin{enumerate}
  \item (Theorem~\ref{theo:A-fg}) The algebras $\bA^\lambda_\Gamma$ are finitely generated.

  \item (Theorem~\ref{theo:finite-generated}) The global Weyl module (of highest weight $\lambda$) is finitely generated as an $\bA_\Gamma^\lambda$-module.
\end{enumerate}
After recalling the twisting functors introduced in \cite{FKKS12} and proving some additional properties of these functors in Section~\ref{sec:twisting}, we turn our attention to local Weyl modules in Section~\ref{sec:local-Weyl-modules}.  There we define the local Weyl modules as the images of one-dimensional irreducible $\bA^\lambda_\Gamma$-modules under the Weyl functor (as in the untwisted setting).  We show (Proposition~\ref{prop:equivalence-of-lWm-defs}) that the modules so defined coincide with those defined directly (i.e.\ without the Weyl functor) in \cite{FKKS12}.  Finally, in Section~\ref{sec:A-lambda-Gamma}, we examine the algebra $\bA^\lambda_\Gamma$.  In particular, we show (Theorem~\ref{theo:Alambda-isom}) that this algebra is isomorphic to a tensor product of symmetric algebras of fixed point subalgebras of $A$.  Under an additional assumption, we also identify  (Theorem~\ref{theo:AlamGam-coinvariants}) $\bA^\lambda_\Gamma$ with the algebra of $\Gamma$-coinvariants of the algebra $\bA^{\llift}$ corresponding to the case where $\Gamma$ is trivial (defined in \cite{CFK10}, but denoted by $\bA_{\llift}$ there), where $\llift$ is a $\g$-weight corresponding to the $\g^\Gamma$-weight $\lambda$.

\medskip

\paragraph{\textbf{Notation}} The set of nonnegative (respectively, positive) integers is denoted by $\N$ (respectively, $\N_+$).  Throughout $\kk$ will denote an algebraically closed field of characteristic zero.  All algebras are over $\kk$ unless otherwise indicated and all associative algebras are assumed to be unital.  Whenever a group $\Gamma$ acts on a $\kk$-vector space $Y$, we denote the subspace of fixed points by $Y^\Gamma$.  For a ring $B$, $B$-mod will denote the category of left $B$-modules.  The notation $S^nB$, $n \in \N$, will denote the subring $(B^{\otimes n})^{S_n}$ of $B^{\otimes n}$ consisting of elements fixed under the natural action of the symmetric group $S_n$ on $B^{\otimes n}$.  Since we work over a field of characteristic zero, this is isomorphic to the quotient of $B^{\otimes n}$ by the ideal generated by the elements $u - \tau(u)$, $u \in B^{\otimes n}$, $\tau \in S_n$.

For a Lie algebra $L$, we denote its universal enveloping algebra by $\cU(L)$.  We have the standard filtration $\cU(L) = \sum_{n \in \N} \cU(L)_n$.   When we refer to the nodes of the Dynkin diagram of a simple Lie algebra of rank $n$ by elements of the set $\{1,2,\dotsc,n\}$, we are referring to the standard labeling that can be found, for instance, in \cite[\S11.4]{Hum72}.  When we refer to the Dynkin diagram of a reductive Lie algebra, we mean the Dynkin diagram of its semisimple part (i.e.\ its derived algebra).  Since we will need to refer to weights of a simple Lie algebra $\g$ and also its subalgebra $\g^\Gamma$ fixed by the action of a group $\Gamma$, we will typically denote weights of $\g$ by $\llift$ and weights of $\g^\Gamma$ by $\lambda$ to avoid confusion.

We will denote by $A$ a finitely generated (hence Noetherian) commutative associative algebra over $\kk$.  We let $X_\rat$ denote the set of $k$-rational points of $X \defeq \Spec A$.  Since $A$ is finitely generated, we have $X_\rat = \maxSpec A$.  For a point $x \in X_\rat$, we will denote the corresponding maximal ideal of $A$ by $\sm_x$.  In some instances, we will identify a point $x$ with its maximal ideal $\sm_x$.

For the reader's convenience, we give here an index of important notation used in the paper.

\medskip

\begin{center}
  \begin{tiny}
    \begin{tabular}{rl!{\hspace{2cm}}rl}
      \multicolumn{4}{c}{\textbf{Index of Notation}} \\
      $\g^\alpha,\ \g^0,\ \g^\pm$ & Page~\pageref{eq:g-plus-minus} & $\Gamma_i$, $e_i$, $f_i$, $h_i$ & Definition~\ref{def:overline} \\
      $\Lambda_\Gamma,\ \Lambda_\Gamma^+,\ Q_\Gamma,\ Q_\Gamma^+$ & Page~\pageref{def:twisted-lattices} (see also~\eqref{def:Lambda-Gamma}) & $\overline{e_i \otimes a}$, $\overline{f_i \otimes a}$, $\overline{h_i \otimes a}$ & Definition~\ref{def:overline} \\
      $\mathcal U_\Gamma,\ \mathcal U_\Gamma^0,\ \mathcal U_\Gamma^\pm$ & \eqref{eq:u-gamma-def} & $\bV^\Gamma_\lambda$ & Definition~\ref{def:V-Gamma-lambda} \\
      $\Xi$ & Page~\pageref{Xi-def}& $M(\psi),\ M^\Gamma(\psi)$ &  Definition~\ref{def:M-psi} \\
      $\cE$, $\cE_\llift$,  $\cE^\Gamma$, $\cE^\Gamma_\lambda$, $Y_\Gamma$ & Definition~\ref{def:wt-ht} & $\Supp J$ & Page~\pageref{def:supp-ideal} \\
      $\psi_\bx$, $\wt$, $\wt_\Gamma$, $\hei$, $\hei_\Gamma$  & Definition~\ref{def:wt-ht} & $\Ann_A^\Gamma V,\ \Supp_A^\Gamma V$ & Definition~\ref{def:support-module} \\
      $V^\Gamma(\psi)$, $V(\psi)$& Definition~\ref{def:EMA-irreducibles} & $X_*$ & Page~\pageref{def:X-star} \\
      $\cI^\Gamma$, $\cI^\Gamma_{\le \tau}$, $\cI^\Gamma_{\le \lambda}$ &Definition~\ref{def:I-Gamma} & $\cF_{\mathbf x},\ \cF_{\mathbf x}^\Gamma$ & Definition~\ref{def:F-x} \\
      $P^\Gamma(V)$ & \eqref{eq:P-Gamma-V} & $\mathbf T,\ \mathbf T_{\mathbf x}$ & Definition~\ref{def:twisting-functor} \\
      $W^\Gamma(V)$ & Definition~\ref{def:twisted-global-Weyl-module} & $e_\bi$, $f_\bi$, $h_\bi$ & \eqref{eq:ebi-fbi-hbi}\\
      $V^\Gamma(\lambda)$, $v^\Gamma_\lambda$ & Page~\pageref{notation:g-Gamma-irred-modules} & $\kappa_{\bi}$ &Page~\pageref{def:kappa-bi} \\
      $W^\Gamma(\lambda)$, $w^\Gamma_\lambda$ & Lemma~\ref{global-weyl-triang} & $W(\psi),\ W^\Gamma(\psi)$& Definition~\ref{def:local-weyl} \\
      $\mathbf A^\lambda_\Gamma$, $\Ann_{\mathcal U^0_\Gamma}(w_\lambda)$ & Definition~\ref{def:A-Gamma-lambda} & $J(\psi)$ & \eqref{def:J-psi} \\
      $\mathbf W^\Gamma_\lambda$ & Definition~\ref{def:twisted-Weyl-functor} & $\bbA^\lambda_\Gamma$ & \eqref{def:bbA} \\
      $\mathbf R^\Gamma_\lambda$ & Definition~\ref{def:R-Gamma-lambda} & $\tilde{\tau}_\lambda$ & \eqref{def:tau-tilde} \\
      $R^+,\ R^+_\Gamma,\ \Pi,\ \Pi_\Gamma$ & Page~\pageref{def:roots-simple-roots} & $\tau_\lambda$ & Lemma~\ref{lem:tau-surjective} \\
    \end{tabular}
  \end{tiny}
\end{center}

\iftoggle{detailsnote}{
\medskip

\paragraph{\textbf{Note on the arXiv version}} For the interested reader, this arXiv version of this paper includes hidden details of some straightforward computations and arguments that are omitted in the pdf file.  These details can be displayed by switching the \texttt{details} toggle to true in the tex file and recompiling.
}{}

\medskip

\paragraph{\textbf{Acknowledgements}} The authors would like to thank Vyjayanthi Chari, Daniel Daigle, Michael Lau and Erhard Neher for helpful conversations.  In particular, they would like to thank Vyjayanthi Chari for explaining the details of the argument found in the proof of Theorem~\ref{theo:global-Weyl-module-projective}\eqref{theo-item:loop-global-Weyl-module-free}, Daniel Daigle for providing an outline of the arguments found in Appendix~\ref{sec:appendix} and some of the commutative algebra arguments in Section~\ref{sec:A-lambda-Gamma}, Michael Lau for pointing us towards Lemma~\ref{lem:g0-abelian} and Erhard Neher for drawing our attention to Example~\ref{eg:ABFP}.

%%%%%%%%%%%%%%%%%%%%%%%%%%%%%%%%%%%%%%%%%%%%%%%%%%%%%%%%%%%%%%%%%%%
%
\section{Equivariant map algebras} \label{sec:EMAs}
%
%%%%%%%%%%%%%%%%%%%%%%%%%%%%%%%%%%%%%%%%%%%%%%%%%%%%%%%%%%%%%%%%%%%

In this section we recall some basic facts about equivariant map algebras.  We refer the reader to \cite{NSS12,NS12} for further details.

Let $\g$ be a finite-dimensional simple Lie algebra over $\kk$, and let $\Gamma$ be a finite group  acting on $\g$ by automorphisms and on the finitely generated commutative associative algebra $A$ by algebra automorphisms (hence on $\Spec A$ by scheme automorphisms).  Thus $\Gamma$ acts diagonally on $\g \otimes A$. As a Lie subalgebra, the set of fixed points $\g^\Gamma$ is reductive in $\g$ (see \cite[Ch.~VII, \S1, no.~5]{Bou75}). That is, $\g^\Gamma$ is a reductive Lie algebra, which acts semisimply on $\g$ by the restriction of the adjoint action of $\g$.

Let $I$ and $I_\Gamma$ be the set of nodes of the Dynkin diagrams of $\g$ of $\g^\Gamma$ respectively.  Fix a triangular decomposition $\g^\Gamma = \n_\Gamma^- \oplus \h_\Gamma \oplus \n_\Gamma^+$ of $\g^\Gamma$.  Denote by $Q_\Gamma^+$ the positive root lattice associated to this triangular decomposition and $Q_\Gamma^- = -Q_\Gamma^+$ the negative root lattice. Similarly, let $\Lambda_\Gamma$ (respectively, $\Lambda_\Gamma^+$) be the corresponding weight lattice (respectively, set of dominant integral weights).  \label{def:twisted-lattices} Relative to $\h_\Gamma$, choose a set of Chevalley generators $\{e^\Gamma_i, f^\Gamma_i, h^\Gamma_i\}_{i \in I_\Gamma}$.

We have a decomposition
\begin{equation*}\label{eq:g-plus-minus} \ts
  \g=\bigoplus_{\alpha\in\h_\Gamma^*} \g^\alpha,\quad \g^\alpha=\left\{x\in\g\ |\ [h,x]=\alpha(h) x,\ h\in\h_\Gamma\right\},
\end{equation*}
with only finitely many $\g^\alpha$ nonzero.  We use superscripts here to avoid confusion with the weight spaces of $\g$ considered as a $\g$-module.  Let
\begin{equation}\label{eq:g-scalene} \ts
  \g^-=\bigoplus_{\alpha\in Q_\Gamma^- \setminus \{0\}} \g^\alpha,\quad \g^+= \bigoplus_{\alpha \notin Q_\Gamma^-} \g^\alpha.
\end{equation}
Then clearly $\g=\g^-\oplus\g^0\oplus\g^+$, $\n^\pm_\Gamma \subseteq \g^\pm$, and $\g^0$ and $\g^-$ are Lie subalgebras of $\g$.  Note that $\g^0$ is simply the centralizer $C_\g(\h_\Gamma)$ of $\h_\Gamma$ in $\g$. Moreover, since $\Gamma \g^\alpha = \g^\alpha$ for each $\alpha \in \h_\Gamma^*$, we see that $\Gamma \g^\pm = \g^\pm$ and $\Gamma\g^0 = \g^0$. Finally, it is not difficult to see that $\g^0$ is a self-normalizing subalgebra of $\g$.

It may be that $\g^\Gamma=0$, in which case $\h_\Gamma=0$ and so $\g^0=\g$ is simple.  However, we have the following result.

\begin{lem} \label{lem:g0-abelian} We have that $\g^0$ is abelian and hence $\g^\Gamma$ is nonzero if either of the following conditions hold:
  \begin{enumerate}
    \item \label{lem-item:Gamma-cylic-g0-abelian} The group $\Gamma$ is cyclic.

    \item \label{lem-item:Gamma-diag-g0-abelian} The group $\Gamma$ acts on $\g$ by diagram automorphisms with respect to a Cartan subalgebra $\h$ of $\g$.
  \end{enumerate}
\end{lem}

\begin{proof}
  That~\eqref{lem-item:Gamma-cylic-g0-abelian} implies the conclusion follows from~\cite[Lem.~8.1(b)]{Kac90}. If~\eqref{lem-item:Gamma-diag-g0-abelian} holds, then the action of $\Gamma$ factors through a cyclic group or the symmetric group $S_3$ on three letters (in type $D_4$).  Since the fixed point subalgebra in type $D_4$ is the same for the full $S_3$-action as it is for the action of the subgroup $\Z_3$, we may in fact assume that $\Gamma$ is cyclic.  Thus the result follows from the fact that~\eqref{lem-item:Gamma-cylic-g0-abelian} implies the conclusion.
\end{proof}

\begin{assumption} \label{assumption}
  For the remainder of this paper we will assume that the subalgebra $\g^0$ is abelian.  By Lemma~\ref{lem:g0-abelian}, this is true if $\Gamma$ is cyclic or acts on $\g$ by diagram automorphisms.
\end{assumption}

\begin{example} \label{eg:ABFP}
  An example showing that it is possible to have $\g^\Gamma$ be nonzero and $\g^0$ be nonabelian is given in \cite[Example~4.3.1]{ABFP09}.  Let $\g$ be the Lie algebra $\mathfrak{s}$ defined there.  Then $\g$ is simple of type $B_3$.  Let $\Gamma \cong \Z_2 \times \Z_2 \times \Z_2$ be the group generated by the order two automorphisms $\sigma_1,\sigma_2,\sigma_3$ described in that reference.  Then $\g^\Gamma \cong \mathfrak{sl}_2$ and $\g^0$ contains the subalgebra consisting of block diagonal $7 \times 7$ matrices with upper left block a $3 \times 3$ zero matrix and arbitrary skew-symmetric lower right $4 \times 4$ block.  Thus $\g^0$ is not abelian.
\end{example}

\begin{defin}[Equivariant map algebra]
  The \emph{map (Lie) algebra} (or \emph{generalized current algebra}) associated to $\g$ and $A$ is the tensor product $\g \otimes A$, with Lie bracket given by extending
  \[
    [u \otimes f, v \otimes g] = [u,v] \otimes fg,\quad u,v \in \g,\ f,g \in A,
  \]
  by linearity.  Thus $\g \otimes A$ is the Lie algebra of algebraic maps from $X = \Spec A$ to $\g$ (identified with affine space) equipped with pointwise multiplication.  The associated \emph{equivariant map (Lie) algebra} is the Lie algebra of fixed points $(\g \otimes A)^\Gamma \subseteq \g \otimes A$, where we consider the diagonal action of $\Gamma$ on $\g \otimes A$.  Thus $(\g \otimes A)^\Gamma$ is the subalgebra of $\g \otimes A$ consisting of those maps that are equivariant with respect to the action of $\Gamma$.
\end{defin}

Since $\Gamma$ respects the decomposition $\g = \g^- \oplus \g^0 \oplus \g^+$, we have a decomposition
\begin{equation} \label{eq:EMA-triangular-decomp}
  (\g \otimes A)^\Gamma = (\g^- \otimes A)^\Gamma \oplus (\g^0 \otimes A)^\Gamma \oplus (\g^+ \otimes A)^\Gamma.
\end{equation}
We let
  \begin{equation}\label{eq:u-gamma-def}
    \cU_\Gamma \defeq \cU((\g \otimes A)^\Gamma),\quad \cU^0_\Gamma \defeq \cU((\g^0\otimes A)^\Gamma),\quad \cU^\pm_\Gamma \defeq \cU((\g^\pm\otimes A)^\Gamma).
  \end{equation}

Let $\Xi$\label{Xi-def} be the character group of $\Gamma$. This is an abelian group, whose group operation we will write additively. Hence, $0$ is the character of the trivial one-dimensional representation, and if an irreducible representation affords the character $\xi$, then $-\xi$ is the character of the dual representation.

If $\Gamma$ is abelian and acts on an algebra $B$ by automorphisms, it is well known that $B=\bigoplus_{\xi \in \Xi} B_\xi$ is a $\Xi$-grading, where $B_\xi$ is the isotypic component of type $\xi$. It follows that $(\g \otimes A)^\Gamma$ can be written as
\begin{equation} \label{eq:EMSA-grading}
  (\g \otimes A)^\Gamma = \ts \bigoplus_{\xi \in \Xi} \, \g_\xi \otimes A_{-\xi},
\end{equation}
since $\g = \bigoplus_\xi \g_\xi$ and $A=\bigoplus_\xi A_\xi$ are $\Xi$-graded and $(\g_\xi \otimes A_{\xi'})^\Gamma = 0$ if $\xi'\ne -\xi$. The decomposition \eqref{eq:EMSA-grading} is an algebra $\Xi$-grading.

For the remainder of this section we assume that $\Gamma$ is abelian, acts freely on $X_\rat$, and acts by diagram automorphisms on $\g$.  Then the set of nodes $I_\Gamma$ of the Dynkin diagram of $\g^\Gamma$ can be naturally identified with the set of $\Gamma$-orbits in $I$.  We will often equate the two in what follows.

\begin{defin}[$\cE$, $\cE^\Gamma$, $Y_\Gamma$, $\psi_\bx$, $\wt$, $\wt_\Gamma$, $\hei$, $\hei_\Gamma$, $\cE_\llift$, $\cE^\Gamma_\lambda$] \label{def:wt-ht}
  Let $\cE$ denote the set of finitely supported functions $\psi \colon X_\rat \to \Lambda^+$.  Here the support of $\psi \in \cE$ is
  \[
    \Supp \psi = \{x \in X_\rat\ |\ \psi(x) \ne 0\}.
  \]
  Since $\Gamma$ acts on $\g$ by diagram automorphisms, it acts naturally on $\Lambda^+$.  We let $\cE^\Gamma$ denote the subset of $\cE$ consisting of those functions that are $\Gamma$-equivariant.

  For a $\Gamma$-invariant subset $Y$ of $X_\rat$, let $Y_\Gamma$ denote the set of subsets of $Y$ containing exactly one point from each $\Gamma$-orbit in $Y$.  For $\psi \in \cE^\Gamma$ and $\bx \in (\Supp \psi)_\Gamma$, define
  \[
    \psi_\bx \colon X_\rat \to \Lambda^+,\quad \psi_\bx(x) =
    \begin{cases}
      \psi(x) & \text{if } x \in \bx, \\
      0 & \text{if } x \not \in \bx.
    \end{cases}
  \]

  For $\psi \in \cE$, we define
  \[ \ts
    \wt \psi \defeq \sum_{x \in \Supp \psi} \psi(x) \in \Lambda^+.
  \]
  If $\psi \in \cE^\Gamma$, we define
  \[
    \wt_\Gamma \psi \defeq (\wt \psi_\bx)|_{\h^\Gamma} \text{ for } \bx \in (\Supp \psi)_\Gamma
  \]
  (this definition is independent of the choice of $\bx$).

  For $\llift \in \Lambda$, write $\llift = \sum_{i \in I} k_i \alpha_i$, $k_i \in \Q$, as a linear combination of simple roots, and define
  \[ \ts
    \hei \llift \defeq \sum_{i \in I} k_i.
  \]
  Similarly, for $\lambda \in \Lambda_\Gamma$, write $\lambda = \sum_{\bi \in I_\Gamma} k_\bi \alpha_\bi$, $k_i \in \Q$, as a linear combination of simple roots, and define
  \[ \ts
    \hei_\Gamma \lambda \defeq \sum_{\bi \in I_\Gamma} k_\bi.
  \]
  For $\psi \in \cE$ (respectively, $\psi \in \cE^\Gamma$), we define
  \[
    \hei \psi \defeq \hei (\wt \psi),\quad (\text{respectively,} \hei_\Gamma \psi \defeq \hei_\Gamma (\wt_\Gamma \psi)).
  \]

  For $\llift \in \Lambda^+$ (respectively, $\lambda \in \Lambda^+_\Gamma$), define $\cE_\llift \defeq \{\psi \in \cE\ |\ \wt \psi = \llift\}$ (respectively, $\cE^\Gamma_\lambda \defeq \{\psi \in \cE^\Gamma\ |\ \wt_\Gamma \psi = \lambda\}$).
\end{defin}

Since $\Gamma$ acts freely on $X_\rat$, the isotropy of any point of $X_\rat$ is trivial.  Having fixed a triangular decomposition of $\g$, the irreducible finite-dimensional representations of $\g$ are enumerated by the set $\Lambda^+$ of dominant integral weights (by associating to a representation its highest weight).  Thus, by \cite[Th.~5.5]{NSS12}, the irreducible finite-dimensional $(\g \otimes A)^\Gamma$-modules are enumerated by the set $\cE^\Gamma$.

\begin{defin}[Modules $V^\Gamma(\psi)$, $V(\psi)$] \label{def:EMA-irreducibles}
  For $\psi \in \cE^\Gamma$, we let $V^\Gamma(\psi)$ denote the corresponding irreducible finite-dimensional $(\g \otimes A)^\Gamma$-module.  Similarly, for $\psi \in \cE$, we let $V(\psi)$ denote the corresponding irreducible finite-dimensional $(\g \otimes A)$-module.
\end{defin}

%%%%%%%%%%%%%%%%%%%%%%%%%%%%%%%%%%%%%%%%%%%%%%%%%%%%%%%%%%%%%%%%%%%
%
\section{Global Weyl modules} \label{sec:Weyl-modules}
%
%%%%%%%%%%%%%%%%%%%%%%%%%%%%%%%%%%%%%%%%%%%%%%%%%%%%%%%%%%%%%%%%%%%

In this section, we introduce our main object of study, the global Weyl module.  Let $\fa$ be a reductive Lie algebra.  Then, given a triangular decomposition $\fa = \n^- \oplus \h \oplus \n^+$, the irreducible finite-dimensional modules of $\fa$ are naturally enumerated by the dominant integral weights of $\fa$.  For a dominant integral weight $\lambda$ of $\fa$, let $V(\lambda)$ denote the irreducible $\fa$-module of highest weight $\lambda$ and let $v_\lambda \in V(\lambda)$ be a highest weight vector.

\begin{defin}[Partial order on $\Irr \fa$]
  Suppose that $V_1$ and $V_2$ are finite-dimensional irreducible $\fa$-modules.  If we fix a triangular decomposition $\fa = \n^- \oplus \h \oplus \n^+$, then $V_i$ has a highest weight $\lambda_i$, $i=1,2$.  We say that $V_1 \le V_2$ if $\lambda_2 - \lambda_1$ lies in the positive root lattice of $\fa$.  This partial order is independent of the particular choice of triangular decomposition. \details{This follows from \cite[Ch.~VIII, \S5, no.~3, Prop.~5]{Bou75}.} It induces a partial order on the set $\Irr \fa$ of isomorphism classes of finite-dimensional irreducible $\fa$-modules.
\end{defin}

Let $V$ be a direct sum of irreducible finite-dimensional $\fa$-modules.
Thus we have a decomposition
\begin{equation} \ts
  V = \bigoplus_{\sigma \in \Irr \fa} V_\sigma,
\end{equation}
where $V_\sigma$ is the $\sigma$ isotypic component of $V$ for $\sigma \in \Irr \fa$.  For $\tau \in \Irr \fa$, we define
\begin{equation} \ts
  V_{\not \le \tau} \defeq \bigoplus_{\sigma \in \Irr \fa,\, \sigma \not \le \tau} V_\sigma.
\end{equation}

We will identify $\g^\Gamma$ with the subalgebra $\g^\Gamma \otimes \kk = (\g \otimes \kk)^\Gamma \subseteq (\g \otimes A)^\Gamma$.  For a $(\g \otimes A)^\Gamma$-module $V$ and $\lambda \in \h_\Gamma^*$ (for some Cartan subalgebra $\h_\Gamma$ of $\g^\Gamma$), we denote by $V_\lambda$ the $\lambda$ weight space of $V$, where $V$ is considered as a $\g^\Gamma$-module by restriction.

\begin{defin}[Categories $\cI^\Gamma$, $\cI^\Gamma_{\le \tau}$ and $\cI^\Gamma_{\le \lambda}$]\label{def:I-Gamma}
  Let $\cI^\Gamma$ denote the full subcategory of the category of $(\g \otimes A)^\Gamma$-modules whose objects are the modules whose restriction to $\g^\Gamma$ are direct sums of irreducible finite-dimensional $\g^\Gamma$-modules.  For $\tau \in \Irr \g^\Gamma$, let $\cI^\Gamma_{\le\tau}$ denote the full subcategory of $\cI^\Gamma$ whose objects consist of those modules whose $\sigma$ isotypic components are zero for $\sigma \in \Irr \g^\Gamma$, $\sigma \not \le \tau$.  That is, the objects of $\cI^\Gamma_{\le\tau}$ are $(\g \otimes A)^\Gamma$-modules $V$ whose decomposition into isotypic components is of the form $V = \bigoplus_{\sigma \le \tau} V_\sigma$.  If we have fixed a triangular decomposition of $\g^\Gamma$, then $\Irr \g^\Gamma$ can be identified with the set $\Lambda^+_\Gamma$ of dominant integral weights for $\g^\Gamma$.  In this case, we will sometimes write $\cI^\Gamma_{\le \lambda}$ instead of $\cI^\Gamma_{\le \tau}$, where $\tau$ is the isomorphism class of the $\g^\Gamma$-module of highest weight $\lambda \in \Lambda^+_\Gamma$.  If $\Gamma$ is the trivial group, we often omit the superscripts $\Gamma$.
\end{defin}

\begin{lem}\label{lem:tensor-algebra}
If $V$ is a direct sum of irreducible finite-dimensional $\fa$-modules, then so is the tensor algebra $T(V)\defeq\bigoplus_{n=0}^\infty V^{\otimes n}$.
\end{lem}

\begin{proof}
  This follows from the fact that the action of $\fa$ preserves each summand $V^{\otimes n}$, which clearly has the given property.
\end{proof}

\begin{lem} \label{lem:induced-completely-reducible}
  If $V$ is a direct sum of irreducible finite-dimensional $\g^\Gamma$-modules, then the induced module
  \begin{equation} \label{eq:P-Gamma-V}
    P^\Gamma(V)\defeq \cU_\Gamma \otimes_{\cU(\g^\Gamma)} V
  \end{equation}
 is a projective object in the category $\cI^\Gamma$.
\end{lem}

\begin{proof}
  Consider the action of $\g^\Gamma$ on $\g\otimes A$ given by (restriction of) the adjoint action on the first factor.  Since $\g$ is a completely reducible $\g^\Gamma$-module, it follows that $\g \otimes A$ is a direct sum of irreducible finite-dimensional $\g^\Gamma$-modules. It is easily checked that $\g^\Gamma$ preserves the subalgebra $(\g\otimes A)^\Gamma$, which therefore also has this property.  Then, by Lemma~\ref{lem:tensor-algebra}, we see that $T((\g\otimes A)^\Gamma)$, and hence $\cU_\Gamma$, are direct sums of irreducible finite-dimensional $\g^\Gamma$-modules.  Since the tensor product is distributive over direct sums, $\cU_\Gamma \otimes_\kk V$ is a direct sum of irreducible finite-dimensional $\g^\Gamma$-modules, hence so is its quotient $P^\Gamma(V)$.  Thus $P^\Gamma(V) \in \Ob \cI^\Gamma$.  The fact that $P^\Gamma(V)$ is projective in this category is a special case of a standard result proved in \cite[Lem.~2]{Hoc56}.
\end{proof}

\begin{defin}[Twisted global Weyl module $W^\Gamma(V)$] \label{def:twisted-global-Weyl-module}
  Let $V$ be an irreducible finite-dimensional $\g^\Gamma$-module.  The corresponding \emph{(twisted) global Weyl module} is the $(\g \otimes A)^\Gamma$-module
  \[
    W^\Gamma(V) \defeq P^\Gamma(V) / \big( \cU_\Gamma (P^\Gamma(V)_{\not \le [V]}) \big),
  \]
  where $[V] \in \Irr \g^\Gamma$ is the isomorphism class of $V$.  Up to isomorphism, $W^\Gamma(V)$ depends only on the isomorphism class of $V$.  If $\Gamma$ is trivial, we will often drop the superscript $\Gamma$.  It follows immediately from Lemma~\ref{lem:induced-completely-reducible} that, for all $\tau \in \Irr \g^\Gamma$ and $V \in \tau$, we have $W^\Gamma(V) \in \Ob \cI^\Gamma_{\le\tau}$.
\end{defin}

Given a triangular decomposition $\g^\Gamma = \n_\Gamma^- \oplus \h \oplus \n_\Gamma^+$ of $\g^\Gamma$, we use the notation $V^\Gamma(\lambda)$ to denote the irreducible $\g^\Gamma$-module of highest weight $\lambda \in \Lambda_\Gamma^+$ and $v^\Gamma_\lambda$ to denote a highest weight vector in this module. \label{notation:g-Gamma-irred-modules}  If $\Gamma$ is trivial, we omit the $\Gamma$ superscripts.

\begin{lem}\label{global-weyl-triang}
  For any dominant integral weight $\lambda \in \Lambda_\Gamma^+$, we have
  \[
    W^\Gamma(\lambda) \defeq W^\Gamma(V^\Gamma(\lambda)) = \big( \cU_\Gamma \otimes_{\cU(\g^\Gamma)} V^\Gamma(\lambda) \big)/ \big( \cU_\Gamma (\g^+ \otimes A)^\Gamma \otimes v^\Gamma_\lambda \big).
  \]
  We let $w^\Gamma_\lambda$ denote the image of $1 \otimes v_\lambda^\Gamma$ in the above quotient.  Then $W^\Gamma(\lambda)$ is generated by $w^\Gamma_\lambda$.
\end{lem}

\begin{proof}
  Let $V = V^\Gamma(\lambda)$ and $v=v^\Gamma_\lambda$.  We need to show that
  \begin{equation} \label{eq:twisted-global-Weyl-equivalence}
    \cU_\Gamma(\g^+ \otimes A)^\Gamma \otimes v = \cU_\Gamma ( P^\Gamma(V))_{\not \le [V]}.
  \end{equation}
  It is clear from the definition of $\g^+$ that we have a weight decomposition
  \[ \ts
    (\g^+ \otimes A)^\Gamma \otimes v = \bigoplus_{\mu \in \Lambda_\Gamma,\, \mu \not \le \lambda} \left( (\g^+ \otimes A)^\Gamma \otimes v \right)_\mu.
  \]
  Thus the left-hand side of~\eqref{eq:twisted-global-Weyl-equivalence} is contained in the right-hand side.

  It remains to prove the reverse inclusion.  For this, it suffices to show that
  \[
    P^\Gamma(V)_\tau \subseteq  \cU_\Gamma (\g^+ \otimes A)^\Gamma \otimes v
  \]
  for all $\tau \not \le [V]$.  Now,
  \[
    P^\Gamma(V) = \cU_\Gamma^- \cU_\Gamma^0 \cU_\Gamma^+ \otimes v = \big( \cU_\Gamma^- \cU_\Gamma^0 \otimes v \big) + \big( \cU_\Gamma (\g^+ \otimes A)^\Gamma \otimes v \big)
  \]
  and the first summand is contained in $\bigoplus_{\mu \le \lambda}  P^\Gamma(V)_\mu$.  The second summand thus contains all submodules of $P^\Gamma(V)$ generated by vectors of weight $\mu$ for $\mu \not \le \lambda$, hence it contains $P^\Gamma(V)_\tau$.
\end{proof}

\begin{rem}
  In the case where $\Gamma$ is trivial, it follows from Lemma~\ref{global-weyl-triang} that Definition~\ref{def:twisted-global-Weyl-module} agrees with the usual definition of the untwisted global Weyl module (see \cite[\S3.3]{CFK10}).  Furthermore, Definition~\ref{def:twisted-global-Weyl-module} reduces to the definition of the twisted global Weyl module in \cite[Def.~3.3]{FMS11} when $A=\C[t^{\pm 1}]$ and $\Gamma$ acts on $\g$ by diagram automorphisms.  Outside of these cases, the general definition of the global Weyl module does not seem to have appeared previously in the literature.
\end{rem}

We conclude this section by giving a presentation of the global Weyl module in terms of generators and relations.

\begin{prop} \label{weyl:gen-rel}
  The global Weyl module $W^\Gamma(\lambda)$ is generated by the vector $w_\lambda^\Gamma$ with defining relations
  \begin{equation} \label{eq:global-Weyl-relations}
    (\g^+\otimes A)^\Gamma w_\lambda^\Gamma=0,\quad (h\otimes 1)w_\lambda^\Gamma=\lambda(h)w_\lambda^\Gamma,\quad (f^\Gamma_i \otimes 1)^{\lambda(h_i)+1}w_\lambda^\Gamma=0,\ \ h\in\h_\Gamma, \ \  i\in I_\Gamma.
  \end{equation}
\end{prop}

\begin{proof}
  That $w_\lambda^\Gamma$ satisfies the latter two relations follows immediately from the fact that $v_\lambda^\Gamma$ does, and for the first relation we only need to observe that the weights of $W^\Gamma(\lambda)$ lie in $\lambda-Q_\Gamma^+$. To see that these are all the relations, let $W$ be the $(\g\otimes A)^\Gamma$-module generated by a vector $w$ with the given relations, so that there is a surjective homomorphism of $(\g\otimes A)^\Gamma$-modules $\pi_1 \colon W \to W^\Gamma(\lambda)$, extending the assignment $w\mapsto w_\lambda^\Gamma$.  By the relations in~\eqref{eq:global-Weyl-relations}, the vector $w \in W$ generates a $\g^\Gamma$-submodule of $W$ isomorphic to a quotient of $V^\Gamma(\lambda)$.  Thus we have a surjective homomorphism
  \[
    \pi_2 \colon P^\Gamma(V) \to W,\quad u_1 \otimes_{\cU(\g^\Gamma)} u_2 v^\Gamma_\lambda \mapsto u_1u_2w,\quad u_1 \in \cU_\Gamma,\ u_2 \in \cU(\g^\Gamma).
  \]
  Since the $\g^\Gamma$-weights of $W$ are bounded above by $\lambda$, it follows that $P^\Gamma(V)_{\nleq[V^\Gamma(\lambda)]}\subseteq\ker(\pi_2)$.  Thus $\pi_2$ induces a map $W^\Gamma(\lambda) \to W$ inverse to $\pi_1$.
\end{proof}

%%%%%%%%%%%%%%%%%%%%%%%%%%%%%%%%%%%%%%%%%%%%%%%%%%%%%%%%%%%%%%%%%%%
%
\section{The Weyl functor} \label{sec:Weyl-functor}
%
%%%%%%%%%%%%%%%%%%%%%%%%%%%%%%%%%%%%%%%%%%%%%%%%%%%%%%%%%%%%%%%%%%%

In this section we extend the definition of the Weyl functor defined in \cite{CFK10} (in the untwisted setting) to the twisted setting of equivariant map algebras.  Throughout this section we fix a triangular decomposition $\g^\Gamma = \n_\Gamma^- \oplus \h_\Gamma \oplus \n_\Gamma^+$.

The following lemma is a modification of the construction in \cite[\S3.4]{CFK10}.  Recall the definition of $\cU^0_\Gamma$ given in~\eqref{eq:u-gamma-def}.

\begin{lem}
  The assignment
  \[
    (uw_\lambda^\Gamma)a \defeq uaw_\lambda^\Gamma, \quad \text{for all } a \in \cU_\Gamma^0,\ u \in \cU_\Gamma,
  \]
  defines a right action of $\cU_\Gamma^0$ on $W^\Gamma(\lambda)$.
\end{lem}

\begin{proof}
  To prove that this action is well-defined, we must prove the implication
  \[
    uw^\Gamma_\lambda = u'w^\Gamma_\lambda \implies uaw^\Gamma_\lambda = u'aw^\Gamma_\lambda
  \]
  for all $u,u' \in \cU_\Gamma$ and $a \in \cU^0_\Gamma$.  This is equivalent to proving that
  \[
    (u-u')w^\Gamma_\lambda = 0 \implies (u-u')aw^\Gamma_\lambda =0
  \]
  for all $u,u' \in \cU_\Gamma$ and $a \in \cU^0_\Gamma$.  For this, it suffices to show that the vectors $aw_\lambda^\Gamma, a\in \cU_\Gamma^0$, satisfy the defining relations of $w_\lambda^\Gamma$ given in Proposition~\ref{weyl:gen-rel}. It follows from the definition of $\g^0$ that $aw_\lambda^\Gamma \in W^\Gamma(\lambda)_\lambda$.  Hence each such vector must lie in the $[V^\Gamma(\lambda)]$-isotypic component of $W^\Gamma(\lambda)$, and so the relation $(f_i^\Gamma \otimes 1)^{\lambda(h_i^\Gamma)+1 }aw_\lambda^\Gamma=0$ must hold.  Finally, $(\g^+\otimes A)^\Gamma aw_\lambda^\Gamma=0$ since the weights of $W^\Gamma(\lambda)$ are bounded above by $\lambda$.
\end{proof}

\begin{defin}[Algebra $\bA_\Gamma^\lambda$] \label{def:A-Gamma-lambda}
  The set $\Ann_{\cU^0_\Gamma}(w_\lambda^\Gamma)=\left\{u\in \cU^0_\Gamma\ |\ uw_\lambda^\Gamma=0\right\}$ is an ideal in the (commutative) algebra $\cU^0_\Gamma$, and we define
  \begin{equation}
    \bA_\Gamma^\lambda \defeq \cU^0_\Gamma / \Ann_{\cU^0_\Gamma}(w_\lambda^\Gamma).
  \end{equation}
\end{defin}

It follows from Definition~\ref{def:A-Gamma-lambda} that $W^\Gamma(\lambda)$ is a $(\cU_\Gamma, \bA_\Gamma^\lambda)$-bimodule.

\begin{lem} \label{lem:extend-inner-aut}
  Every inner automorphism of $\g^\Gamma$ can be extended to an inner automorphism of $\g$ commuting with the action of $\Gamma$.
\end{lem}

\begin{proof}
  Let $\sigma$ be an inner automorphism of $\g^\Gamma$.  Then $\sigma$ acts trivially on the center of $\g^\Gamma$ and so it suffices to extend the restriction of $\sigma$ to the semisimple part $[\g^\Gamma,\g^\Gamma]$ of $\g^\Gamma$.  This restriction can be written in the form $e^{\ad x_1} e^{\ad x_2} \dotsm e^{\ad x_n}$, where $x_i \in ([\g^\Gamma,\g^\Gamma])_{\beta_i}$, $i=1,\dotsc,n$, and each $\beta_i$ is a (nonzero) root of $[\g^\Gamma,\g^\Gamma]$ (see, for example, \cite[Th.~27.5(e)]{Hum75}).  Since $\g$ is a sum of finite-dimensional irreducible $\g^\Gamma$-modules, the action of $\ad x_i$ on $\g$ is nilpotent (since it increases weights by $\beta_i$) for all $i=1,\dotsc,n$.  Thus $\sigma$ extends to an inner automorphism of $\g$. Furthermore, since $x_i \in \g^\Gamma$ for each $i=1,\dotsc,n$, this automorphism commutes with the action of $\Gamma$ on $\g$.
\end{proof}

\begin{prop}
  Up to isomorphism, $\bA_\Gamma^\lambda$ depends only the isomorphism class of $V^\Gamma(\lambda)$ and not on the choice of triangular decomposition of $\g^\Gamma$.
\end{prop}

\begin{proof}
  Let $D_i = (\n_{D_i}^\pm, \h_{D_i})$, $i=1,2$, be two triangular decompositions of $\g^\Gamma$ and fix a finite-dimensional irreducible $\g^\Gamma$-module $V$.  For $i=1,2$, let $\lambda_i$ be the highest weights of $V$ with respect to the triangular decomposition $D_i$ and let $v_i$ be a highest weight vector.  By \cite[Ch.~VIII, \S5, no.~3, Prop.~5]{Bou75}, there exists an inner automorphism $\sigma$ of $\g^\Gamma$ such that $\sigma(\h_{D_1})=\h_{D_2}$ and $\sigma$ carries the positive root spaces to the positive root spaces.  By Lemma~\ref{lem:extend-inner-aut}, we can extend $\sigma$ to an inner automorphism of $\g$, also denoted $\sigma$, which commutes with the action of $\Gamma$.  Thus, $\sigma$ induces an automorphism of $(\g \otimes A)^\Gamma$.  For a $\g^\Gamma$-module (or $(\g \otimes A)^\Gamma$-module) $W$, let $W^\sigma$ denote the module obtained from $W$ by twisting the action by $\sigma$.  That is, $W^\sigma$ is equal to $W$ as a vector space, with action given by
  \[
    (x,w) \mapsto \sigma^{-1}(x)w,\quad x \in \g^\Gamma\ \text{(respectively, $x \in (\g \otimes A)^\Gamma$)},\ w \in W^\sigma.
  \]
  Then, in $V^\sigma$, $v_1$ is a highest weight vector with respect to $D_2$ of highest weight $\lambda_1 \circ \sigma^{-1}$.  Since $\sigma$ is inner, $V^\sigma \cong V$ (see, for instance, \cite[Ch.~VIII, \S7, no.~2, Rem.~1]{Bou75}) and so $\lambda_1 \circ \sigma^{-1} = \lambda_2$.  Thus we have an isomorphism $V^\sigma \cong V$ determined by $v_1 \mapsto v_2$.  This isomorphism maps $uv_1$ to $\sigma(u)v_2$ for all $u \in \cU(\g^\Gamma)$.

  We also have an isomorphism $(\cU_\Gamma)^\sigma \cong \cU_\Gamma$ of $(\cU_\Gamma,\cU(\g^\Gamma))$-bimodules, where $(\cU_\Gamma)^\sigma$ has the action twisted on both sides.  This isomorphism maps $u$ to $\sigma(u)$ for all $u \in (\cU_\Gamma)^\sigma$.  Thus we have an isomorphism
  \begin{equation} \label{eq:twisting-isom}
    (\cU_\Gamma)^\sigma \otimes_{\cU(\g^\Gamma)} V^\sigma \cong \cU_\Gamma \otimes_{\cU(\g^\Gamma)} V,\quad u \otimes v_1 \mapsto \sigma(u) \otimes v_2,
  \end{equation}
  of left $\cU_\Gamma$-modules.

  Now, $\sigma(\g^0_{D_1}) = \g^0_{D_2}$, where $\g^0_{D_i}$ denotes the centralizer $C_\g(\h_{D_i})$ for $i=1,2$.  It follows that $\sigma \big( \cU((\g^0_{D_1} \otimes A)^\Gamma) \big) = \cU((\g^0_{D_2} \otimes A)^\Gamma)$.  Furthermore, the isomorphism~\eqref{eq:twisting-isom} implies that $\sigma \big( \Ann_{\cU((\g^0_{D_1} \otimes A)^\Gamma)} (1 \otimes v_1) \big) = \Ann_{\cU((\g^0_{D_2} \otimes A)^\Gamma)} (1 \otimes v_2)$, completing the proof.
\end{proof}

Each weight space $W^\Gamma(\lambda)_\mu$, $\mu \in \Lambda_\Gamma$, is a right $\bA_\Gamma^\lambda$-submodule.  In particular, $W^\Gamma(\lambda)_\lambda$ is a right $\bA_\Gamma^\lambda$-module and
\[
  W^\Gamma(\lambda)_\lambda \cong \bA_\Gamma^\lambda \quad \text{as right $\bA_\Gamma^\lambda$-modules}.
\]

\begin{defin}[Twisted Weyl functor $\bW^\Gamma_\lambda$] \label{def:twisted-Weyl-functor}
  For $\lambda \in \Lambda^+_\Gamma$, we define the right exact \emph{(twisted) Weyl functor} (cf.\ \cite[\S3.4]{CFK10}, \cite[\S4.1]{FMS11})
  \[
    \bW^\Gamma_\lambda \colon \bA_\Gamma^\lambda\text{-mod} \to \cI^\Gamma_{\le\lambda},\quad \bW^\Gamma_\lambda M = W^\Gamma(\lambda) \otimes_{\bA_\Gamma^\lambda} M,\quad \bW^\Gamma_\lambda f = 1 \otimes f,
  \]
  for $M \in \Ob \bA_\Gamma^\lambda$-mod, $f \in \Mor \bA_\Gamma^\lambda$-mod. That $\bW^\Gamma_{\lambda} M \in \Ob \cI^\Gamma_{\le\lambda}$ for all $M \in \Ob \bA_\Gamma^\lambda$-mod follows from the fact that $W^\Gamma(\lambda) \in \Ob \cI^\Gamma_{\le\lambda}$.  \details{In particular, if $uw_\lambda^\Gamma \otimes m\in \mathbf W_\lambda^\Gamma M$ for some $u \in \cU_\Gamma$ and $m\in M$, then $\cU(\g^\Gamma)(uw_\lambda^\Gamma \otimes m)\subseteq (\cU(\g^\Gamma)uw_\lambda^\Gamma)\otimes m$, which is a direct sum of irreducible finite-dimensional $\g^\Gamma$-modules.}
\end{defin}

\begin{lem} \label{lem:Ann-kills-V_lambda}
  For $\lambda \in \Lambda_\Gamma^+$ and $V \in \Ob \cI^\Gamma_{\le\lambda}$, we have $(\Ann_{\cU^0_\Gamma}(w_\lambda^\Gamma)) V_\lambda = 0$.
\end{lem}

\begin{proof}
  Suppose $v' \in V_\lambda$ and let $V' = \cU(\g^\Gamma)v'$ be the $\g^\Gamma$-submodule of $V$ generated by $v'$.  Since $V\in\Ob\cI^\Gamma_{\le\lambda}$, it follows that $V'$ is finite-dimensional.  Since $v'$ is a highest weight vector and $V$ is a sum of irreducible finite-dimensional $\g^\Gamma$-modules, we see that $V' \cong V^\Gamma(\lambda)$.  Identifying $V'$ with $V^\Gamma(\lambda)$, we therefore have a linear map
  \[
    \varphi \colon\cU_\Gamma \otimes_{\cU(\g^\Gamma)} V^\Gamma(\lambda) \to V,\ u \otimes v \mapsto uv,\quad u \in \cU_\Gamma,\ v \in V^\Gamma(\lambda),
  \]
  and
  \[
    \cU_\Gamma(\g^+\otimes A)^\Gamma\otimes v_\lambda^\Gamma \subseteq \ker \varphi.
  \]
  Thus $\varphi$ descends to a homomorphism $\bar{\varphi} \colon W^\Gamma(\lambda)\to V$ mapping $w_\lambda^\Gamma$ to $v'$.  Hence if $u \in \Ann_{\cU^0_\Gamma}(w_\lambda^\Gamma)$, we see that $uv'= \bar{\varphi}(u w_\lambda^\Gamma)=\bar{\varphi}(0)=0$.
\end{proof}

Note that
\[
  \bW^\Gamma_\lambda \bA_\Gamma^\lambda \cong W^\Gamma(\lambda) \quad \text{as $(\g \otimes A)^\Gamma$-modules},
\]
and
\[
  (\bW^\Gamma_\lambda M)_\mu = W^\Gamma(\lambda)_\mu \otimes_{\bA_\Gamma^\lambda} M \quad \text{as vector spaces},
\]
for $\lambda \in \Lambda_\Gamma^+$, $\mu \in \Lambda_\Gamma$ and $M \in \Ob \bA_\Gamma^\lambda$-mod.

\begin{defin}[Twisted restriction functor $\bR^\Gamma_\lambda$]\label{def:R-Gamma-lambda}
  By Lemma~\ref{lem:Ann-kills-V_lambda}, the left action of $\cU_\Gamma^0$ on $V \in \Ob \cI^\Gamma_{\le\lambda}$ induces a left action of $\bA_\Gamma^\lambda$ on $V_\lambda$.  We denote the resulting $\bA_\Gamma^\lambda$-module by $\bR^\Gamma_\lambda V$.  For $\psi \in \Hom_{\cI^\Gamma_{\le\lambda}}(V,V')$, the restriction $\bR^\Gamma_\lambda \psi \colon V_\lambda \to V'_\lambda$ is a morphism of $\bA_\Gamma^\lambda$-modules.  These maps define the \emph{(twisted) restriction functor} $\bR^\Gamma_\lambda \colon \cI^\Gamma_{\le\lambda} \to \bA_\Gamma^\lambda$-mod.  The functor $\bR^\Gamma_\lambda$ is exact since restriction to a weight space is exact.
\end{defin}

The following is a generalization of results that are known in the untwisted case (see \cite[\S3.6 and Prop.~5]{CFK10}).

\begin{prop} \label{prop:right-adjoint}
  The functors $\bW^\Gamma_\lambda$ and $\bR^\Gamma_\lambda$ have the following properties:
  \begin{enumerate}
    \item \label{prop-item:RW=1} $\bR^\Gamma_\lambda \bW^\Gamma_\lambda \cong \Id_{\bA_\Gamma^\lambda\textup{-mod}}$ (isomorphism of functors);

    \item $\bW^\Gamma_\lambda$ is left adjoint to $\bR^\Gamma_\lambda$;

    \item \label{lem-item:W-projectives} the functor $\bW^\Gamma_\lambda$ maps projective objects to projective objects.
  \end{enumerate}
\end{prop}

\begin{proof}
  \begin{asparaenum}
    \item For $M \in \Ob \bA_\Gamma^\lambda$-mod, we have the following isomorphisms of left $\bA_\Gamma^\lambda$-modules:
      \[
        \bR^\Gamma_\lambda \bW^\Gamma_\lambda M = (\bW^\Gamma_\lambda M)_\lambda = W^\Gamma(\lambda)_\lambda \otimes_{\bA_\Gamma^\lambda} M \cong \bA_\Gamma^\lambda \otimes_{\bA_\Gamma^\lambda} M \cong M.
      \]

    \item We must define natural transformations
      \[
        \epsilon \colon \bW^\Gamma_\lambda \bR^\Gamma_\lambda\Rightarrow \Id_{\cI^\Gamma_{\le\lambda}},\quad \eta \colon \Id_{\bA_\Gamma^\lambda\text{-mod}}\Rightarrow \bR^\Gamma_\lambda\bW^\Gamma_\lambda,
      \]
      such that, for each $M\in\Ob\bA_\Gamma^\lambda$-mod and $V\in\Ob\cI^\Gamma_{\le\lambda}$, we have the equality of morphisms
      \begin{equation}\label{eq:adjoint}
        \Id_{\bW^\Gamma_\lambda M}= \epsilon_{\bW^\Gamma_\lambda M}\circ\bW^\Gamma_\lambda(\eta_M), \quad \Id_{\bR^\Gamma_\lambda V}=\bR^\Gamma_\lambda(\epsilon_V)
        \circ\eta_{\bR^\Gamma_\lambda V}.
      \end{equation}
      For $M\in\Ob\bA_\Gamma^\lambda$-mod, define
      \[
        \eta_M \colon M\to \bR^\Gamma_\lambda\bW^\Gamma_\lambda M, \quad m\mapsto w^\Gamma_\lambda\otimes m.
      \]
      It is straightforward to verify that $\eta_M$ is natural in $M$ and thus the collection $\{\eta_M\ |\ M\in\Ob\bA_\Gamma^\lambda\text{-mod}\}$ does indeed define a natural transformation $\eta \colon \Id_{\bA_\Gamma^\lambda\text{-mod}}\Rightarrow \bR^\Gamma_\lambda\bW^\Gamma_\lambda$.
      \details{To see that $\eta_M$ is natural in $M$, suppose $f:M\to N$ is a morphism of $\bA^\lambda_\Gamma$-modules. Then, for $m\in M$, we have
      \[
        \bR^\Gamma_\lambda \bW^\Gamma_\lambda(f)(\eta_M(m))=\bR^\Gamma_\lambda\bW^\Gamma_\lambda(f)(w^\Gamma_\lambda\otimes m)=(1\otimes f)_\lambda(w^\Gamma_\lambda\otimes m)=w^\Gamma_\lambda\otimes f(m) = \eta_N(f(m)),
      \]
      proving the commutativity of the following diagram:
      \[
        \entrymodifiers={+!!<0pt,\fontdimen22\textfont2>}\xymatrixcolsep{5pc}\xymatrix{ M\ar[r]^{f} \ar[d]_{\eta_M}&N\ar[d]^{\eta_N} \\
        \bR^\Gamma_\lambda\bW^\Gamma_\lambda M\ar[r]_{\bR^\Gamma_\lambda\bW^\Gamma_\lambda(f)}&
        \bR^\Gamma_\lambda\bW^\Gamma_\lambda N}
      \]
      Hence the collection $\{\eta_M\ |\ M\in\Ob\bA_\Gamma^\lambda\text{-mod}\}$ does indeed define a natural transformation $\eta \colon \Id_{\bA_\Gamma^\lambda\text{-mod}}\Rightarrow \bR^\Gamma_\lambda\bW^\Gamma_\lambda$.}

      For $V \in \Ob \cI^\Gamma_{\le \lambda}$, define $\epsilon_V:\bW^\Gamma_\lambda\bR^\Gamma_\lambda V\to V$ as follows. First, we regard $W^\Gamma(\lambda)\otimes_{\kk} \bR^\Gamma_\lambda V$ as a left $(\g\otimes A)^\Gamma$-module via the action of $(\g\otimes A)^\Gamma$ on $W^\Gamma(\lambda)$. Then it follows by Proposition~\ref{weyl:gen-rel} that the assignment $\epsilon_1 \colon W^\Gamma(\lambda)\otimes_{\kk} \bR^\Gamma_\lambda V\to V$ given by $u w^\Gamma_\lambda\otimes v\mapsto uv,\ u\in \cU_\Gamma$, is a well-defined map of left $(\g\otimes A)^\Gamma$-modules. To see that this map factors through to a map $\epsilon_V:\bW^\Gamma_\lambda \bR^\Gamma_\lambda V\to V$ we observe that, for $a \in \bA^\lambda_\Gamma$ and $u\in \cU_\Gamma$, we have
      \[
        \epsilon_1(u w^\Gamma_\lambda a \otimes v) = \epsilon_1(u a w^\Gamma_\lambda \otimes v) = u av = \epsilon_1(u w^\Gamma_\lambda \otimes a v).
      \]
      Again, it is straightforward to check that $\epsilon_V$ is natural in $V$.
      \details{Thus $\epsilon_V$ is given by extending by linearity the assignment $uw^\Gamma_\lambda\otimes v\mapsto uv$, where $u\in \cU_\Gamma$ and $v\in V_\lambda$.  For $f \colon V\to U$ in $\Mor\cI^\Gamma_{\le\lambda}$, we have
      \[
        \epsilon_U(\bW^\Gamma_\lambda \bR^\Gamma_\lambda (f)(uw^\Gamma_\lambda \otimes v)) = \epsilon_U(uw^\Gamma_\lambda\otimes f(v))=uf(v) = f(uv) = f(\epsilon_V(uw^\Gamma_\lambda\otimes v)).
      \]
      Thus the diagram
      \[
        \entrymodifiers={+!!<0pt,\fontdimen22\textfont2>}\xymatrixcolsep{5pc}\xymatrix{ \bW^\Gamma_\lambda\bR^\Gamma_\lambda V\ar[r]^{\bR^\Gamma_\lambda\bW^\Gamma_\lambda(f)} \ar[d]_{\epsilon_V}&\bW^\Gamma_\lambda\bR^\Gamma_\lambda U\ar[d]^{\epsilon_U} \\
        V\ar[r]_{f}&U}
      \]
      commutes and so $\epsilon$ is natural in $V$.}

      Finally, we check the equalities in \eqref{eq:adjoint}. For $M \in \Ob \bA_\Gamma^\lambda$-mod and $m\in M$, we have
      \[
        (\epsilon_{\bW^\Gamma_\lambda M}\circ \bW^\Gamma_\lambda(\eta_M))(uw^\Gamma_\lambda\otimes m)=\epsilon_{\bW^\Gamma_\lambda M}(uw^\Gamma_\lambda\otimes w^\Gamma_\lambda\otimes m)=uw^\Gamma_\lambda\otimes m,
      \]
      and for $V\in \Ob \cI^\Gamma_{\le\lambda}$, $m \in \bR^\Gamma_\lambda V$, we have
      \[
        (\bR^\Gamma_\lambda(\epsilon_V) \circ \eta_{\bR^\Gamma_\lambda V})(m) = \bR^\Gamma_\lambda(\epsilon_V)(w^\Gamma_\lambda\otimes m)=m,
      \]
      which establishes \eqref{eq:adjoint}.

    \item This follows from the fact that $\bW^\Gamma_\lambda$ is left adjoint to a right exact functor (see, for example, \cite[Prop.~II.10.2]{HS97}). \qedhere
  \end{asparaenum}
\end{proof}

\begin{cor} \label{cor:global-Weyl-module-free}
  For $\lambda \in \Lambda_\Gamma^+$, the global Weyl module $W^\Gamma(\lambda)$ is a projective object in $\cI^\Gamma_{\le \lambda}$.
\end{cor}

\begin{proof}
  This follows from Proposition~\ref{prop:right-adjoint}\eqref{lem-item:W-projectives} and the fact that $\bW^\Gamma_\lambda \bA^\lambda_\Gamma = W^\Gamma(\lambda)$.
\end{proof}

By Proposition~\ref{prop:right-adjoint}\eqref{prop-item:RW=1}, we have
\begin{equation} \label{eq:W=WRW}
  \bW^\Gamma_\lambda M\cong\bW^\Gamma_\lambda \bR^\Gamma_\lambda \bW^\Gamma_\lambda M \quad \text{for all } M \in \Ob \bA^\lambda_\Gamma\text{-mod}.
\end{equation}
The following theorem gives a homological characterization of this property.  In the untwisted case (i.e.\ when $\Gamma$ is trivial), it was proved in \cite[Th.~1]{CFK10}.

\begin{theo} \label{theo:WR=1-hom-char}
  Let $V\in\Ob\cI^\Gamma_{\leq\lambda}$. Then $V\cong \bW^\Gamma_\lambda\bR^\Gamma_\lambda V$ if and only if, for each $U\in\Ob\cI^\Gamma_{\leq\lambda}$ with $U_\lambda=0$, we have
  \[
    \Hom_{\cI^\Gamma_{\leq\lambda}}(V,U)=0,\quad\Ext^1_{\cI^\Gamma_{\leq\lambda}}(V,U)=0.
  \]
\end{theo}

\begin{proof}
  First, let $V\in\Ob\cI^\Gamma_{\leq\lambda}$ with $V\cong\bW^\Gamma_\lambda\bR^\Gamma_\lambda V$. Suppose that there is a homomorphism of $(\g\otimes A)^\Gamma$-modules $\varphi \colon V\to U$, where $U\in\Ob\cI^\Gamma_{\leq\lambda}$ with $U_\lambda=0$. By hypothesis $V$ is generated as a $\cU_\Gamma$-module by $\bR^\Gamma_\lambda V$, but on the other hand $\bR^\Gamma_\lambda \varphi=0$. Hence, $\varphi=0$.

  To establish the second condition, let $P$ be a projective object of $\bA_\Gamma^\lambda$-mod equipped with a surjective homomorphism $\pi \colon P\to \bR^\Gamma_\lambda V$. Applying the right exact functor $\bW^\Gamma_\lambda$ yields a surjective homomorphism of $(\g\otimes A)^\Gamma$-modules
  \[
    1\otimes\pi:\bW^\Gamma_\lambda P\to \bW^\Gamma_\lambda \bR^\Gamma_\lambda V\cong V,
  \]
  with $\bW^\Gamma_\lambda P$ a projective module by Proposition~\ref{prop:right-adjoint}\eqref{lem-item:W-projectives}. Take $K=\ker(1\otimes\pi)$ to obtain a short exact sequence
  \begin{equation}\label{weyl-seq}
    0\to K\to \bW^\Gamma_\lambda P\to V\to 0.
  \end{equation}
  Now, $K$ is generated by $K_\lambda$ as a $\cU_\Gamma$-module, being a homomorphic image of $\bW^\Gamma_\lambda(\ker\pi)$, which is itself  generated by its $\lambda$ weight space $W^\Gamma(\lambda)_\lambda\otimes_{\bA_\Gamma^\lambda} (\ker\pi)$.  Thus $\Hom_{\cI^\Gamma_{\leq\lambda}}(K,U)=0$, and it follows from the long exact sequence obtained by applying the functor $\Hom_{\cI^\Gamma_{\leq\lambda}}(-,U)$ to the sequence \eqref{weyl-seq} that $\Ext^1_{\cI^\Gamma_{\leq\lambda}}(V, U)=0$.
  \details{We get the long exact sequence:
  \begin{multline*}
    0 \to \Hom_{\cI^\Gamma_{\leq\lambda}}(V,U) \to \Hom_{\cI^\Gamma_{\leq\lambda}}(\bW^\Gamma_\lambda P,U) \to \Hom_{\cI^\Gamma_{\leq\lambda}} (K,U) \to \Ext^1_{\cI^\Gamma_{\leq\lambda}}(V,U) \\
    \to \Ext^1_{\cI^\Gamma_{\leq\lambda}}(\bW^\Gamma_\lambda P) \to \Ext^1_{\cI^\Gamma_{\leq\lambda}} (K,U) \to \cdots,
  \end{multline*}
  which becomes (using the fact that $\Ext^1_{\cI^\Gamma_{\leq\lambda}}(\bW^\Gamma_\lambda P,U)=0$ since $\bW^\Gamma_\lambda P$ is projective)
  \[
    0 \to 0 \to \Hom_{\cI^\Gamma_{\leq\lambda}}(\bW^\Gamma_\lambda P,U) \to 0 \to \Ext^1_{\cI^\Gamma_{\leq\lambda}}(V,U) \to \Ext^1_{\cI^\Gamma_{\leq\lambda}}(\bW^\Gamma_\lambda P,U) \to \Ext^1_{\cI^\Gamma_{\leq\lambda}} (K,U) \to \cdots.
  \]
  }

  Conversely, let $V\in \Ob \cI^\Gamma_{\le \lambda}$ satisfy the given vanishing conditions on $\Hom$ and $\Ext^1$.  Set $V'=\cU_\Gamma V_\lambda$ and observe that $V/V'\in\Ob\cI^\Gamma_\lambda$ with $(V/V')_\lambda=0$.  Thus our hypothesis implies that $\Hom_{\cI^\Gamma_\lambda}(V,V/V')=0$, so that $V=V'$.  We immediately see that the map $\epsilon_V \colon \bW^\Gamma_\lambda \bR^\Gamma_\lambda V \to V$ defined in the proof of Proposition~\ref{prop:right-adjoint} is surjective.  Let $U=\ker(\epsilon_V)$, so that $U_\lambda=0$. Consider the short exact sequence
  \[
    0\to U\to \bW^\Gamma_\lambda\bR^\Gamma_\lambda V \xrightarrow{\epsilon_V} V\to 0.
  \]
  The long exact sequence obtained by applying $\Hom_{\cI^\Gamma_{\le\lambda}}(-,U)$ now gives $\Hom_{\cI^\Gamma_{\le\lambda}}(U,U) = 0$ and hence $U=0$.
  \details{We get the long exact sequence:
  \[
    0 \to \Hom_{\cI^\Gamma_{\leq\lambda}}(V,U) \to \Hom_{\cI^\Gamma_{\leq\lambda}}(\bW^\Gamma_\lambda\bR^\Gamma_\lambda V,U) \to \Hom_{\cI^\Gamma_{\leq\lambda}}(U,U) \to \Ext^1_{\cI^\Gamma_{\leq\lambda}}(V,U) \to \cdots.
  \]
  which becomes (using the fact that $\Hom_{\cI^\Gamma_{\leq\lambda}}(\bW^\Gamma_\lambda\bR^\Gamma_\lambda V,U)=0$ since $\bW^\Gamma_\lambda\bR^\Gamma_\lambda V$ is generated by its $\lambda$ weight space, but $U_\lambda=0$)
  \[
    0 \to 0 \to 0 \to \Hom_{\cI^\Gamma_{\leq\lambda}}(U,U) \to 0 \to \cdots.
  \]
  }
  Thus $\epsilon_V$ is an isomorphism, which completes the proof.
\end{proof}

%%%%%%%%%%%%%%%%%%%%%%%%%%%%%%%%%%%%%%%%%%%%%%%%%%%%%%%%%%%%%%%%%%%
%
\section{Properties of global Weyl modules} \label{sec:properties-global}
%
%%%%%%%%%%%%%%%%%%%%%%%%%%%%%%%%%%%%%%%%%%%%%%%%%%%%%%%%%%%%%%%%%%%

This section marks the beginning of the second part of the current paper.  For the remainder of the paper, we assume that $\Gamma$ is a finite abelian group acting freely on $X_\rat$ and acting on $\g$ by diagram automorphisms.  (Note that $\Gamma$ is a finite abelian group and acts freely on $X_\rat$ in the case of (twisted) loop and multiloop algebras.)  Under this additional assumption, we shall deduce in this section some further properties of global Weyl modules and of the algebra $\bA^\Gamma_\lambda$.  In particular, we will see that both are finitely generated.  This generalizes the results of \cite{FMS11} (which considers the twisted loop algebra), but is new even in the case of twisted multiloop algebras.

Fix a triangular decomposition $\g = \n_- \oplus \h \oplus \n_+$.  Then, since $\Gamma$ acts by diagram automorphisms, we have an induced triangular decomposition $\g^\Gamma = \n_-^\Gamma \oplus \h^\Gamma \oplus \n_+^\Gamma$.  We let $R^+$ (respectively, $\Pi$) and $R^+_\Gamma$\label{def:roots-simple-roots} (respectively, $\Pi_\Gamma$) denote the sets of positive (respectively, simple) roots of $\g$ and $\g^\Gamma$, respectively.

\begin{lem}[{\cite[Ch.~V, \S1, no.~9, Th.~2]{Bou85b}}] \label{lem:AGamma-fg}
  The algebra $A^\Gamma$ is finitely generated (as an algebra) and $A_\xi$ is finitely generated as an $A^\Gamma$-module for all $\xi \in \Xi$.
\end{lem}

\begin{rem} \label{rem:cyclic-assumption}
  Lemma~\ref{lem:AGamma-fg} allows us to make a simplifying assumption as follows.  Let $\Gamma'$ be the subgroup of $\Gamma$ acting trivially on $\g$ (equivalently, fixing the Dynkin diagram of $\g$).  Then $(\g \otimes A)^\Gamma \cong (\g \otimes A^{\Gamma'})^{\Gamma/\Gamma'}$.  By Lemma~\ref{lem:AGamma-fg}, $A^{\Gamma'}$ is finitely generated.  Thus, replacing $A$ by $A^{\Gamma'}$ and $\Gamma$ by $\Gamma/\Gamma'$, we may assume without loss of generality that $\Gamma$ acts faithfully on $\g$ (equivalently, on the Dynkin diagram of $\g$).  Since $\Gamma$ is abelian, this implies that $\Gamma$ is either trivial or is a cyclic group of order two or three.
\end{rem}

\begin{assumption}
  For the remainder of the paper, we will assume that $\Gamma$ is a cyclic group generated by an automorphism $\sigma$ of the Dynkin diagram of $\g$.
\end{assumption}

\begin{lem} \label{lem:A-xi-generators}
  Let $\xi \in \Xi \setminus \{0\}$.  Then there exists a finite subset of $A_\xi$ that generates $A$ as an algebra.
\end{lem}

\begin{proof}
  Since $A$ is finitely generated, it admits a finite generating set $\{a_1,\dots,a_n\}$.  Taking homogeneous components of the elements in this set, we may assume that each $a_i$ is homogeneous (i.e.\ belongs to $A_\tau$ for some $\tau \in \Xi$).  Let $m = |\Gamma|$.  By \cite[Lem.~4.4]{NS11}, we have $(A_\xi)^m = A_{m \xi} = A_0$.  Thus, we can write $1 = \sum_{\ell=1}^k g_{\ell,1} g_{\ell,2} \dotsm g_{\ell,m}$ for some $g_{\ell,s} \in A_\xi$, $\ell=1,\dotsc,k$, $s = 1,\dotsc,m$.  Let $S = \{a_1,\dots,a_n\} \cup \{g_{\ell,s}\ |\ \ell=1,\dotsc,k,\ s=1,\dotsc,m\}$.

  Now, suppose $a_i \in A_\tau$ for $\tau \ne \xi$.  Since $\Gamma$ (hence $\Xi$) is a cyclic group of order two or three (see Remark~\ref{rem:cyclic-assumption}), $\xi$ generates $\Xi$.  Thus we have $-\tau = (r-1)\xi$ for some $1 \le r \le m$.  So $\tau + r \xi = \xi$.  Now, $a_i = a_i 1 = a_i \sum_{\ell=1}^k g_{\ell,1} g_{\ell,2} \dotsm g_{\ell,m}$.  For $\ell=1,\dotsc,k$, let $a_{i,\ell} = a_i g_{\ell,1} \dotsm g_{\ell,r} \in A_\xi$ and set $S' = (S \setminus \{a_i\}) \cup \{a_{i,1},\dotsc,a_{i,k}\}$.  Then $S'$ generates $A$ and, compared with $S$, has one fewer element lying outside $A_\xi$.  The result then follows by induction.
\end{proof}

\begin{defin}[$\delta_\bi$]
  For $\mathbf{i}\in I_\Gamma$, define
  \[
    \delta_\bi \defeq
    \begin{cases}
      1, & \text{if $\g$ is of type $A_{2n}$}, \\
      1, &\text{if $\g$ is not of type $A_{2n}$ but $\alpha_{\mathbf i}$ is a short root (i.e.\ $|\bi| > 1$)}, \\
      |\Gamma|, &\text{if $\g$ is not of type $A_{2n}$ and $\alpha_{\mathbf i}$ is a long root (i.e.\ $|\bi|=1$)},
    \end{cases}
  \]
  where we have identified $I_\Gamma$ with the set of $\Gamma$-orbits in $I$ and $|\bi|$ denotes the size of the orbit $\bi$.
\end{defin}

\begin{defin}[$\Gamma_i$, $\overline{e_i \otimes a}$, $\overline{f_i \otimes a}$, $\overline{h_i \otimes a}$] \label{def:overline}
  Let $\{e_i,f_i,h_i\}_{i \in I}$ be a set of Chevalley generators of $\g$.  For $i \in I$, let $\Gamma_i = \{\gamma \in \Gamma\ |\ \gamma i = i\}$ be the isotropy subgroup of $i$.  Then, for $a \in A^{\Gamma_i}$, let
  \[ \ts
    \overline{e_i \otimes a} \defeq \sum_{\gamma \in \Gamma/\Gamma_i} \gamma(e_i \otimes a) = \sum_{\gamma \in \Gamma/\Gamma_i} e_{\gamma i} \otimes \gamma(a) \in (\n^+ \otimes A)^\Gamma.
  \]
  We define $\overline{f_i \otimes a}$ and $\overline{h_i \otimes a}$ similarly.  Note that replacing $i$ by another element in the same $\Gamma$-orbit in $I$ only changes the above elements by a scalar multiple.
\end{defin}

The elements $\overline{h_i \otimes a}$, $i \in I$, $a \in A^{\Gamma_i}$, span $(\h \otimes A)^\Gamma$.  Now fix $\xi \in \Xi \setminus \{0\}$.  Since $(A_\xi)^{|\Gamma|}=A^\Gamma$ (by \cite[Lem.~4.4]{NS11}), we see that the elements $\overline{h_i \otimes a^{\delta_\bi}}$, $\bi \in I_\Gamma$, $i \in \bi$, $a \in A_\xi$, also span $(\h \otimes A)^\Gamma$.

For $b$ an element of any associative algebra $B$, we denote by $b^{(r)}$ the divided power $b^r/r!$.

\begin{lem}\label{lem:twisted-garland}
  Suppose that $\ell\in\mathbb N$, $a\in A_\xi$ for $\xi\in\Xi \setminus\{0\}$, and $\bi \in I_\Gamma$.  Let $i \in \bi$.  Then there exist $q(i,a)_s \in \cU((\h \otimes A)^\Gamma)$, $s=1,\dotsc,\ell$, such that the following statements hold:
  \begin{enumerate}
    \item We have $q(i,a)_0=1$ and $q(i,a)_s \in \cU((\h \otimes \kk[a]_s)^\Gamma)$, where
        \[ \ts
          \kk[a]_s \defeq \{\sum_{i=0}^s c_i a^{\delta_\bi i}\ |\ c_1,\dotsc,c_s \in \kk\} \subseteq A.
        \]

    \item If $\g$ is not of type $A_{2n}$, or if $\g$ is of type $A_{2n}$ and $\alpha_\bi$ is a long root, then
        \begin{equation}\label{lem-eq:a-lambda-fg-eq2} \ts
          (\overline{e_i\otimes a^{\delta_\bi}})^{(\ell)}(\overline{ f_i\otimes 1})^{(\ell+1)}+ (-1)^{\ell+1} \sum_{s=0}^{\ell} (\overline{f_i \otimes a^{\delta_\bi(\ell  - s)}})q(i,a)_s \in \mathcal U_\Gamma(\n^+\otimes A)^\Gamma.
        \end{equation}

    \item If $\g$ is of type $A_{2n}$ and $\alpha_\bi\in R^+_\Gamma$ is the simple short root of $\g^\Gamma$, then
        \begin{equation} \label{lem:eq:a-lambda-fd-A2n-short} \ts
          (\sqrt{2}\ \overline{e_i\otimes 1})^{(2\ell+1)} (y_i\otimes a)^{(\ell+1)}+\sum_{s=0}^{\ell} (\sqrt{2}\ \overline{f_i \otimes a^{\ell  - s+1}})q(i,a)_s \in \mathcal U_\Gamma(\n^+\otimes A)^\Gamma,
        \end{equation}
        where $y_i=-[f_i,f_{i+1}]\in\g_1$ and the subscript $1$ here denotes the nonzero element of $\Xi$.
\end{enumerate}
\end{lem}

\begin{proof}
  The homomorphism of $\kk$-algebras $\kk[t]\to A$ given by $t \mapsto a$ extends to a homomorphism of Lie algebras $\g\otimes\kk[t]\to\g\otimes A$. Applying this latter map to the formulas found in \cite[Lem.~3.3]{CFS08} proves the lemma.  Note that \cite[Lem.~3.3(iii)(a)]{CFS08} is incorrect. To obtain the correct statement, one should replace $t^{mk+\epsilon}$ in the power series $\mathbf{x}_{n}^-(u)$ in that reference by $t^{mk+\epsilon+1}$.
\end{proof}

\begin{theo} \label{theo:A-fg}
  For all $\lambda \in \Lambda^+_\Gamma$, $\mathbf{A}_{\Gamma}^{\lambda}$ is a finitely generated algebra.
\end{theo}

\begin{proof}
  If the action of $\Gamma$ on $\g$ is trivial, then $(\g \otimes A)^\Gamma \cong \g \otimes A^\Gamma$ and the theorem follows from \cite[Th.~2(i)]{CFK10}.  Thus, we assume that $\Gamma$ acts nontrivially on $\g$.  Fix $\xi \in \Xi \setminus \{0\}$, let $\{a_1, \ldots, a_N \} \subseteq A_\xi$ be generators of $A$ as an algebra (see Lemma~\ref{lem:A-xi-generators}) and fix a generator $a_k$ with $1\le k\le N$.

  First, suppose either that $\g$ is not of type $A_{2n}$ and $\alpha_\bi$ is any simple root of $\g^\Gamma$, or that $\g$ is of type $A_{2n}$ and $\alpha_\bi$ is a long root of $\g^\Gamma$ (i.e.\ $\bi \ne \{n,n+1\}$ when we identify $I_\Gamma$ with $\Gamma$-orbits in $I$).  For any homogeneous $a \in A^{\Gamma_i}$, multiplying both sides of~\eqref{lem-eq:a-lambda-fg-eq2} by $\overline{ e_i \otimes a}$ on the left gives the following identity in $\mathcal U_\Gamma$:
  \[ \ts
    (\overline{e_i \otimes a})(\overline{e_i \otimes a_k^{\delta_\bi}})^{( \ell )}(\overline{f_i\otimes 1})^{(\ell+1)} + (-1)^{\ell+1}  \sum_{s=0}^{\ell} (\overline{h_i \otimes a a_k^{\delta_\bi(\ell - s)}}) q(i,a_k)_s   \in  \mathcal U_\Gamma (\n^+ \otimes A)^\Gamma.
  \]
  Now suppose that $\ell\ge \lambda(h_\bi)$. In this case, we have $(\overline{f_i \otimes 1})^{\ell+1}w_\lambda^\Gamma = 0,$ and thus
  \[ \ts
    (\overline{h_i \otimes a a_k^{\delta_\bi\ell}})w_\lambda^\Gamma = -\left(\sum_{s=1}^{\ell} (\overline{h_i \otimes a a_k^{\delta_\bi(\ell - s)}})q(i,a_k)_s \right)w_\lambda^\Gamma.
  \]
  Iterating this argument, it follows that for all $s_1,\ldots,s_n\in\mathbb N$,
  \begin{equation}\label{a-lambda-fg-H-long}
    (\overline{h_i \otimes a_1^{\delta_\bi s_1}\cdots a_N^{\delta_\bi s_N}})w_\lambda^\Gamma = H(i,s_1,\ldots,s_N)w_\lambda^\Gamma,
  \end{equation}
  where $H(i,s_1,\ldots,s_N)$ is a linear combination of finite products of elements of $\mathcal U((\h\otimes A)^\Gamma$) of the form $(\overline{h_i\otimes a_1^{\delta_\bi\ell_1}\cdots a_N^{\delta_\bi\ell_n}})$ with $0\le \ell_1,\ldots,\ell_N < \lambda(h_\bi)$.  Therefore, the images of the vectors
  \[
    \{ (\overline{h_i \otimes a_1^{\delta_\bi\ell_1} \cdots a_N^{\delta_\bi\ell_N}})\ :\  0 \leq \ell_1, \ldots, \ell_N < \lambda(h_\bi) \}
  \]
  generate $\bA_\Gamma^\lambda$.

  Now let $\g$ be of type $A_{2n}$ and suppose that $\bi$ corresponds to the orbit $\{n,n+1\}$, that is, $\alpha_\bi$ is the simple short root of $\g^\Gamma$.
  Since the generators $a_1, \ldots, a_N$ lie in $A_1$ (where 1 denotes the nontrivial character of $\Gamma$), it follows that $y_i \otimes a_k \in (\g \otimes A)^\Gamma$.  Let $a \in A$.  Multiplying both sides of~\eqref{lem:eq:a-lambda-fd-A2n-short} by $\sqrt{2}\ \overline{ e_i \otimes a}$ on the left gives
  \[ \ts
    2^{\ell+1}(\overline{e_i \otimes a})(\overline{e_i \otimes 1})^{(2 \ell + 1)} (y_i \otimes a_k)^{(\ell+1)} + \sum_{s=0}^{\ell} 2(\overline{h_i \otimes a a_k^{\ell - s+1}})q(i,a_k)_s   \in  \mathcal U_{\Gamma}(\n^+ \otimes A)^\Gamma.
  \]
  Now, we claim that for all $\ell  \ge r_i \defeq \frac{1}{2} \lambda(h_\bi)$, we have $(y_i \otimes a_k)^{\ell+1}w_\lambda^\Gamma = 0$, which implies that
  \[ \ts
    (\overline{h_i \otimes a a_k^{\ell+1}})w_\lambda^\Gamma =-  \sum_{s=1}^{\ell} (\overline{h_i \otimes a a_k^{\ell - s+1}})q(n,a_k)_s w_\lambda^\Gamma.
  \]
  Iterating this, we arrive again at \eqref{a-lambda-fg-H-long} (with the upper bound $\lambda(h_\bi)$ on $\ell_j$, $1 \le j \le N$, replaced by $r_i+1$), and the result follows.

  It remains to prove the claim. A straightforward calculation shows that $[h,y_i]=-2\alpha_{\bi}(h)y_i$ for all $h\in\h$.
  \details{We check on vectors $h_{\mathbf j}$. If $\mathbf j=\{n,n+1\}$ we compute that $2[h_n+h_{n+1},[f_{i+1},f_i]] =-4y_i.$ If $\mathbf j=\{n-1, n+2\}$ we get $2y_i$, and otherwise zero. This is exactly $-2\alpha_\bi(h_{\mathbf j})y_i.$}
  Thus, $(y_i \otimes a_k)^{r_i+1}w_\lambda^\Gamma$ has weight $\lambda-(\lambda(h_\bi)+2)\alpha_\bi$, and so it suffices to show that $\lambda-(\lambda(h_\bi)+2)\alpha_\bi$ is not a weight of $W^\Gamma(\lambda)$.  Since $W^\Gamma(\lambda)$ is a direct sum of irreducible finite-dimensional $\g^\Gamma$-modules, its weights (which all lie below $\lambda$) are invariant under the action of the Weyl group of $\g^\Gamma$.  But $s_\bi(\lambda-(\lambda(h_\bi)+2)\alpha_\bi)=\lambda+2\alpha_\bi$ does not lie below $\lambda$, concluding the proof.
\end{proof}

\begin{lem}\label{lem:garland}
  Suppose $r \in \N$, $a \in A^\Gamma$, $\alpha \in R_\Gamma^+$, and $\{x_\alpha^-, x_\alpha^+, h_\alpha \}$ is an $\mathfrak{sl}_2$-triple in $\g^\Gamma$ corresponding to $\alpha$.  Then
  \[ \ts
    (x_\alpha^+ \otimes a)^r(x_\alpha^- \otimes 1)^{r+1} -  \sum_{s = 0}^r (x_\alpha^- \otimes a^{r-s})p(a, \alpha)_s \in \cU(\g^\Gamma \otimes A^\Gamma)(\n_+^\Gamma \otimes A^\Gamma),
  \]
  for some $p(a, \alpha)_s \in \cU(\h^\Gamma \otimes A^\Gamma) \subseteq \cU((\h \otimes A)^\Gamma)$ with $p(a,\alpha)_0=1$.
\end{lem}

\begin{proof}
  This follows from~\cite[Lem.~5]{CFK10} (see also~\cite[Lem.~7.5]{Gar78}), where we replace the $\g$ and $A$ there by $\g^\Gamma$ and $A^\Gamma$ respectively.
\end{proof}

\begin{theo}\label{theo:finite-generated}
  For all $\lambda \in \Lambda^+_\Gamma$, the global Weyl module $W^\Gamma(\lambda)$ is a finitely generated right $\bA^\lambda_\Gamma$-module.
\end{theo}

\begin{proof}
  If the action of $\Gamma$ on $\g$ is trivial, then $(\g \otimes A)^\Gamma \cong \g \otimes A^\Gamma$ and the theorem follows from \cite[Th.~2(i)]{CFK10}.  Thus we assume that $\Gamma$ acts nontrivially on $\g$.

  Let $\{a_1, \ldots, a_s\}$ be a finite set of generators of $A^\Gamma$ (see Lemma~\ref{lem:AGamma-fg}).  By Proposition~\ref{weyl:gen-rel}, we have $(\n_+\otimes A)^\Gamma w_\lambda^\Gamma = 0$, where $w_\lambda^\Gamma$ is the usual generator of $W^\Gamma(\lambda)$.  We also have $\cU((\h \otimes A)^\Gamma) w_\lambda^\Gamma = w_\lambda^\Gamma \bA^\lambda_\Gamma$. Then the PBW theorem implies that
  \[
    W^\Gamma(\lambda) = \cU((\n^- \otimes A)^\Gamma) w_\lambda^\Gamma \bA^\lambda_\Gamma.
  \]
  For a positive root $\alpha \in R^+_\Gamma$ of $\g^\Gamma$, let $\{x_\alpha^+, x_\alpha^-, h_\alpha\} \subseteq \g^\Gamma$ be a corresponding $\mathfrak{sl}_2$-triple.

  \medskip

  \noindent \emph{Claim 1:} For all $\alpha \in R_\Gamma^+$, we have
  \[
    (x_{\alpha}^- \otimes A^\Gamma) w_\lambda^\Gamma \subseteq \Span\{ (x_{\alpha}^- \otimes a_1^{\ell_1} \cdots a_s^{\ell_s}) w_\lambda^\Gamma\bA^\lambda_\Gamma \ | \ 0 \leq \ell_1, \ldots, \ell_s< \lambda(h_\alpha)\}.
  \]

  \medskip

  \noindent \emph{Proof of Claim 1:} Let $r \geq \lambda(h_{\alpha})$ and $1 \le i \le s$.  By Lemma~\ref{lem:garland} we have
  \[ \ts
    \sum_{s = 0}^r (x_\alpha^- \otimes a_i^{r-s})p(a_i, \alpha)_s w_\lambda^\Gamma = (x_\alpha^+ \otimes a_i)^r(x_\alpha^- \otimes 1)^{r+1} w_\lambda^\Gamma = 0.
  \]
  Thus
  \[
    (x_\alpha^- \otimes a_i^r) w_\lambda^\Gamma \in \Span \{ (x_{\alpha}^- \otimes a_i^s) w_\lambda^\Gamma \bA^\lambda_\Gamma \ | \ 0 \leq s < r \}.
  \]
  We then have, by induction, that
  \begin{eqnarray} \label{eq:fin-1}
    (x_\alpha^- \otimes a_i^r) w_\lambda^\Gamma \in \Span \{ (x_{\alpha}^- \otimes a_i^\ell) w_\lambda^\Gamma\bA^\lambda_\Gamma \ | \ 0 \leq \ell < \lambda(h_\alpha)\} \quad \text{for all } r \ge \lambda(h_\alpha).
  \end{eqnarray}
  Now, for $1 \le i,j \le s$ and $m_i, m_j \in \N$, we have
  \begin{align} \label{eq:fin-2}
    (h_\alpha \otimes a_j^{m_j})(x_\alpha^- \otimes a_i^{m_i}) w_\lambda^\Gamma &=  \big( -2 x_\alpha^- \otimes a_j^{m_j} a_i^{m_i} + (x_\alpha^- \otimes a_i^{m_i})(h_\alpha \otimes a_j^{m_j}) \big) w_\lambda^\Gamma \\
    &\in -2(x_\alpha^- \otimes a_j^{m_j} a_i^{m_i})w_\lambda^\Gamma + (x_\alpha^- \otimes a_i^{m_i}) w_\lambda^\Gamma \bA^\lambda_\Gamma \nonumber
  \end{align}
  and so
  \begin{align*}
    (x_\alpha^- \otimes a_j^{m_j} a_i^{m_i})w_\lambda^\Gamma
    &\in \ts \frac{-1}{2} (h_\alpha \otimes a_j^{m_j})(x_\alpha^- \otimes a_i^{m_i}) w_\lambda^\Gamma + (x_\alpha^- \otimes a_i^{m_i}) w_\lambda^\Gamma \bA^\lambda_\Gamma \\
    &\in \ts \Span \{ (h_\alpha \otimes a_j^{m_j}) (x_{\alpha}^- \otimes a_i^\ell) w_\lambda^\Gamma\bA^\lambda_\Gamma, (x_{\alpha}^- \otimes a_i^\ell) w_\lambda^\Gamma\bA^\lambda_\Gamma \ | \ 0 \leq \ell < \lambda(h_\alpha)\}
  \end{align*}
  by~\eqref{eq:fin-1}.  Then, using~\eqref{eq:fin-2} with $m_i=\ell$, we have
  \begin{equation} \label{eq:fin-3}
    (x_\alpha^- \otimes a_j^{m_j} a_i^{m_i})w_\lambda^\Gamma \in \Span \{(x_\alpha^- \otimes a_j^{m_j} a_i^\ell)w_\lambda^\Gamma \bA^\lambda_\Gamma, (x_\alpha^- \otimes a_i^\ell)w_\lambda^\Gamma \bA^\lambda_\Gamma\ |\ 0 \leq \ell < \lambda(h_\alpha)\}.
  \end{equation}
  Replacing $j,i,m_j,m_i$ by $i,j,\ell,m_j$ (respectively) in~\eqref{eq:fin-3}, we have
  \[
    (x_\alpha^- \otimes a_i^\ell a_j^{m_j})w_\lambda^\Gamma \in \Span \{(x_\alpha^- \otimes a_i^\ell a_j^{\ell_j}) w_\lambda^\Gamma \bA^\lambda_\Gamma, (x_\alpha^- \otimes a_j^{\ell_j})w_\lambda^\Gamma \bA^\lambda_\Gamma \ |\ 0 \le \ell_j < h(h_\alpha)\}.
  \]
  Thus
  \[
    (x_\alpha^- \otimes a_j^{m_j} a_i^{m_i})w_\lambda^\Gamma \in \Span \{(x_\alpha^- \otimes a_j^{\ell_j} a_i^{\ell_i})w_\lambda^\Gamma \bA^\lambda_\Gamma\ |\ 0 \leq \ell_i, \ell_j < \lambda(h_\alpha)\}.
  \]

  Repeating the above argument gives
  \begin{equation} \label{eq:fin-5}
    (x_\alpha^- \otimes a_1^{m_1} \cdots a_s^{m_s}) w_\lambda^\Gamma \in \Span\{ (x_{\alpha}^- \otimes a_1^{\ell_1} \cdots a_s^{\ell_s}) w_\lambda^\Gamma\bA^\lambda_\Gamma \ | \ 0 \leq \ell_i  < \lambda(h_\alpha)\ \forall\ i\}.
  \end{equation}
  Since $\{a_1^{m_1} \cdots a_s^{m_s} \ |\ m_i \geq 0\}$ is a spanning set of $A^\Gamma$, Claim 1 follows.

  \medskip

  \noindent \emph{Claim 2:} As a right $\bA^\lambda_\Gamma$-module, $(\n' \otimes A)^\Gamma w_\lambda^\Gamma$ is finitely generated, where $\n' \defeq \bigoplus_{\alpha \in \Pi} \g_{-\alpha}$.

  \medskip

  \noindent \emph{Proof of Claim 2:} Let $m$ be the order of the generator $\sigma$ of $\Gamma$.  This generator $\sigma$ induces a permutation of the simple roots of $\g$. For a simple root $\beta \in \Pi$ of $\g$, we denote by $\{y_\beta^+, y_\beta^-, h_\beta \} \subseteq \g$ a corresponding $\mathfrak{sl}_2$-triple. A basis of $(\n')^\Gamma$ is given by the set
  \[ \ts
    \cB_0 \defeq \{ y_\beta^- \ |\ \beta \in \Pi,\ \sigma(\beta) = \beta \} \cup \left\{ \sum_{j= 0}^{m-1} y_{\sigma^j(\beta)}^- \ |\ \beta \in \Pi,\ \sigma(\beta) \neq  \beta \right\}.
  \]
  If $m=2$, we have a basis $\cB_0 \sqcup \cB_1$ of $\n'$ where
  \[
    \cB_1 \defeq \{ y_\beta^- - y_{\sigma(\beta)}^- \ |\ \beta \in \Pi,\ \sigma(\beta) \neq \beta \}
  \]
  consists of $\sigma$-eigenvectors of eigenvalue $-1$.  If $m=3$ we have a basis $\cB_0 \sqcup \cB_1 \sqcup \cB_2$ of $\n'$ where
  \begin{gather*}
    \cB_1 \defeq \{y_\beta^- + \eta y_{\sigma(\beta)}^- + \eta^2 y_{\sigma^2(\beta)}^- \ |\ \beta \in \Pi,\ \sigma(\beta) \ne \beta \},\\
    \cB_2 \defeq \{y_\beta^- + \eta^2 y_{\sigma(\beta)}^- + \eta y_{\sigma^2(\beta)}^- \ |\ \beta \in \Pi,\ \sigma(\beta) \ne \beta \}.
  \end{gather*}
  Here $\cB_i$, $i=1,2$, consists of $\sigma$-eigenvectors of eigenvalue $\eta^i$, where $\eta$ is a primitive third root of unity.

  For each $\alpha \in \Pi_\Gamma$, after multiplying by a scalar if necessary, we have
  $x_\alpha^- = y_{\beta_\alpha}^-$ (or $x_\alpha^- = \sqrt{\kappa_\alpha} \sum_{j= 0}^{m-1} y_{\sigma^j(\beta_\alpha)}^-$) and $h_\alpha = h_{\beta_\alpha}$ (respectively, $h_\alpha = \kappa_\alpha \sum_{j= 0}^{m-1} h_{\sigma^j(\beta_\alpha)}$) for some $\beta_\alpha \in R^+$, where $\kappa_\alpha=2$ if $\g$ is of type $A_{2n}$ and $\alpha$ is the simple short root of $\g^\Gamma$, otherwise $\kappa_\alpha=1$.  In fact, the $\Gamma$-orbit of $\beta_\alpha$ is uniquely determined by the condition $\beta_\alpha|_{\h^\Gamma} = \alpha$.

  For the remainder of the proof, we restrict our attention to the case $m=2$.  The case $m=3$ is similar and will be omitted.  We have
  \begin{equation} \label{eq:n'A-decomp}
    (\n' \otimes A)^\Gamma = (\n'_0 \otimes A_0) \oplus (\n'_1 \otimes A_1),
  \end{equation}
  where the subscript 1 denotes the nontrivial character of $\Gamma$.  Furthermore, $\cB_1$ is a basis of $\n'_1$.  By Lemma~\ref{lem:AGamma-fg}, we know that $A_1$ is a finitely generated $A^\Gamma$-module.  Let $\{b_1, \ldots, b_k\}$ be a finite set of generators of this module.  Now choose $\alpha \in \Pi_\Gamma$ such that $\sigma(\beta_\alpha) \ne \beta_\alpha$ and set $\beta = \beta_\alpha$.  Then
  \[
    \{ (y_\beta^- - y_{\sigma(\beta)}^-) \otimes a_1^{m_1} \cdots a_s^{m_s} b_i \ |\ m_j \geq 0,\ 1 \leq i \leq k \} \subseteq \n'_1 \otimes A_1 \subseteq (\n' \otimes A)^\Gamma.
  \]
  Furthermore, $(h_\beta - h_{\sigma(\beta)}) \otimes b_i \in (\h \otimes A)^\Gamma$, and so $((h_\beta - h_{\sigma(\beta)}) \otimes b_i) w_\lambda^\Gamma \in w_\lambda^\Gamma \bA^\lambda_\Gamma$.  Now,
  \begin{equation} \label{eq:fin-comm-rel}
    [ h_\beta - h_{\sigma(\beta)}, x_\alpha^- ] = [h_\beta - h_{\sigma(\beta)}, y_\beta^- + y_{\sigma(\beta)}^-] = -(\kappa_\alpha+1)(y_\beta^- - y_{\sigma(\beta)}^-).
  \end{equation}
  This implies, for all $m_j \geq 0$, that
  \begin{align*}
    -(\kappa_\alpha&+1)((y_\beta^- - y_{\sigma(\beta)}^-) \otimes a_1^{m_1} \dotsb a_s^{m_s} b_i)w_\lambda^\Gamma \\
    &= ((h_\beta - h_{\sigma(\beta)}) \otimes b_i)(x_\alpha^- \otimes a_1^{m_1} \cdots a_s^{m_s}) w_\lambda^\Gamma - (x_\alpha^- \otimes a_1^{m_1} \cdots a_s^{m_s})((h_\beta - h_{\sigma(\beta)}) \otimes b_i) w_\lambda^\Gamma \\
    &\in ((h_\beta - h_{\sigma(\beta)}) \otimes b_i)(x_\alpha^- \otimes a_1^{m_1} \cdots a_s^{m_s}) w_\lambda^\Gamma - (x_\alpha^- \otimes a_1^{m_1} \cdots a_s^{m_s}) w_\lambda^\Gamma \bA_\Gamma^\lambda \\
    &\subseteq (h_\beta - h_{\sigma(\beta)}) \otimes b_i)  \Span\{ (x_{\alpha}^- \otimes a_1^{\ell_1} \cdots a_s^{\ell_s}) w_\lambda^\Gamma\bA^\lambda_\Gamma \ | \ 0 \leq \ell_i  < \lambda(h_\alpha)\} \\
    &\qquad \qquad \qquad \qquad \qquad +  \Span\{ (x_{\alpha}^- \otimes a_1^{\ell_1} \cdots a_s^{\ell_s}) w_\lambda^\Gamma\bA^\lambda_\Gamma \ | \ 0 \leq \ell_i  < \lambda(h_\alpha)\ \forall\ i\} \quad \text{(by~\eqref{eq:fin-5})} \\
    &\subseteq \Span \{ ((y_\beta^- - y_{\sigma(\beta)}^-) \otimes a_1^{\ell_1} \cdots a_s^{\ell_s}b_i) w_\lambda^\Gamma\bA^\lambda_\Gamma , (x_{\alpha}^- \otimes a_1^{\ell_1} \cdots a_s^{\ell_s}) w_\lambda^\Gamma\bA^\lambda_\Gamma\ | \ 0 \leq \ell_i  < \lambda(h_\alpha)\ \forall\ i\},
  \end{align*}
  where the last containment follows from~\eqref{eq:fin-comm-rel}.  Since $\kappa_\alpha+1 \ne 0$ and $A_1$ is spanned by elements of the form $a_1^{m_1} \cdots a_s^{m_s}b_i$, we see that $((y_\beta^- - y_{\sigma(\beta)}^-) \otimes A_1) w_\lambda^\Gamma$ is contained in
  \[
    \Span \{ ((y_\beta^- - y_{\sigma(\beta)}^-) \otimes a_1^{\ell_1} \cdots a_s^{\ell_s}b_i) w_\lambda^\Gamma\bA^\lambda_\Gamma , (x_{\alpha}^- \otimes a_1^{\ell_1} \cdots a_s^{\ell_s}) w_\lambda^\Gamma\bA^\lambda_\Gamma\ | \ 0 \leq \ell_j  < \lambda(h_\alpha),\ 1 \le i \le k\}.
  \]
  Claim 2 now follows from~\eqref{eq:n'A-decomp} and the above arguments.

  \medskip

  \noindent \emph{Completion of the proof of Theorem~\ref{theo:finite-generated}:} We continue to assume that $m=2$, the case $m=3$ being similar.  Define
  \begin{gather*}
    \mathcal{D}_0 \defeq \{ z \otimes a_1^{\ell_1} \cdots a_s^{\ell_s} \ | \ z \in \cB_0,\ 0 \leq \ell_j < \lambda(h_\alpha)\ \forall\ j\} \quad \text{and} \\
    \mathcal{D}_1 \defeq \{ z \otimes a_1^{\ell_1} \cdots a_s^{\ell_s} b_i\ | \ z \in \cB_1,\ 0 \leq \ell_j < \lambda(h_\alpha)\ \forall\ j,\ 1 \le i \le k\}.
  \end{gather*}
  Let $\mathcal{D} = \mathcal{D}_0 \cup \mathcal{D}_1$.  We claim that $\cU((\n' \otimes A)^\Gamma)_n w_\lambda^\Gamma \subseteq \sum_{\ell=0}^n \mathcal{D}^\ell w_\lambda^\Gamma \bA^\lambda_\Gamma$ for all $n \in \N_+$.  The result is true for $n=1$ by the above.  Assume that it is true for some $n \ge 1$. Let $u \in (\n' \otimes A)^\Gamma$ and $\tilde u \in \cU((\n' \otimes A)^\Gamma)_n$.  Then, by assumption, we have $\tilde u w_\lambda^\Gamma \in u' w_\lambda^\Gamma \bA^\lambda_\Gamma$ for some $u' \in \sum_{\ell=0}^n \mathcal{D}^\ell$.  Then we have
  \begin{multline*} \ts
    u \tilde u w_\lambda^\Gamma \in uu' w_\lambda^\Gamma \bA^\lambda_\Gamma = ([u,u']w_\lambda^\Gamma + u'uw_\lambda^\Gamma)\bA^\lambda_\Gamma \\ \ts
    \subseteq \cU((\n' \otimes A)^\Gamma)_n w_\lambda^\Gamma \bA^\lambda_\Gamma + u' (\mathcal{D} w_\lambda^\Gamma) \bA^\lambda_\Gamma \subseteq \sum_{\ell =1}^{n+1} \mathcal{D}^\ell w_\lambda^\Gamma \bA^\lambda_\Gamma.
  \end{multline*}
  Thus our claim holds by induction.

  Now,
  \begin{align*}
    [(\n' \otimes A)^\Gamma, (\n' \otimes A)^\Gamma] &= [\n'_0 \otimes A_0 + \n'_1 \otimes A_1, \n'_0 \otimes A_0 + \n'_1 \otimes A_1] \\
    &= \big(([\n'_0,\n'_0] + [\n'_1,\n'_1]) \otimes A_0\big) \oplus \big([\n'_0,\n'_1] \otimes A_1\big) \\
    &= ([\n',\n'] \otimes A)^\Gamma,
  \end{align*}
  where, in the second equality, we have used that $A_1^2=A_0$ by~\cite[Lem.~4.4]{NS11}.  Since $\n'$ generates $\n_-$, an easy inductive argument then shows that $\cU((\n_- \otimes A)^\Gamma)_1 \subseteq \cU((\n' \otimes A)^\Gamma)_N$ for some $N \in \N_+$.  Thus, for $n \in \N_+$, we have
  \[ \ts
    \cU((\n_- \otimes A)^\Gamma)_n w_\lambda^\Gamma \subseteq \cU((\n' \otimes A))_{Nn} w_\lambda^\Gamma \subseteq \sum_{\ell=0}^{Nn} \mathcal{D}^\ell w_\lambda^\Gamma \bA^\lambda_\Gamma.
  \]

  Now, all $\g^\Gamma$-weights of $\n_- \otimes A$ are nonzero (this follows from the fact that $\mathcal{D}$ is a basis of $\n' \otimes A$ and all of its elements have nonzero $\g^\Gamma$-weight) and the set of $\g^\Gamma$-weights occurring in $W^\Gamma(\lambda)$ is finite.  Thus, there exists an $M \in \N$ such that $\cU((\n_- \otimes A)^\Gamma)_n w_\lambda^\Gamma \bA^\lambda_\Gamma = W^\Gamma(\lambda)$ for all $n \ge M$.  Therefore $\sum_{\ell=1}^{NM} \mathcal{D}^\ell w_\lambda^\Gamma \bA^\lambda_\Gamma = W^\Gamma(\lambda)$ and so $W^\Gamma(\lambda)$ is finitely generated as an $\bA^\lambda_\Gamma$-module.
\end{proof}

Theorems~\ref{theo:A-fg} and~\ref{theo:finite-generated} are generalizations of~\cite[Th.~2(i)]{CFK10}, which gives the result in the untwisted setting (i.e.\ when $\Gamma$ is trivial).

We have the following immediate corollary.

\begin{cor} \label{cor:WM-fin-dim}
  If $M$ is a finite-dimensional $\bA^\lambda_\Gamma$-module, then $\bW_\lambda^\Gamma M$ is finite-dimensional.
\end{cor}

It is straightforward to show that if $V \in \Ob \cI^\Gamma$ and, for some $\lambda \in \Lambda_\Gamma^+$, we have $\dim V_\lambda = 1$, $\wt V \subseteq \lambda - Q_\Gamma^+$ and $\cU((\g \otimes A)^\Gamma) V_\lambda = V$, then $V$ has a unique irreducible quotient.

The following is a generalization of~\cite[Prop.~8]{CFK10} and a refinement of Theorem~\ref{theo:WR=1-hom-char}.

\begin{cor} \label{cor:local-weyl-module-char-fd}
  Let $V\in\Ob\cI^\Gamma_{\le \lambda}$ with $\dim V_\lambda < \infty$. Then $V\cong \bW^\Gamma_\lambda\bR^\Gamma_\lambda V$ if and only if
  \begin{equation} \label{eq:lWm-hom-char-fd}
    \Hom_{\cI^\Gamma_{\le \lambda}}(V,U)=0 \quad \text{and} \quad \Ext^1_{\cI^\Gamma_{\le \lambda}}(V,U)=0
  \end{equation}
  for all irreducible finite-dimensional $U\in\Ob\cI^\Gamma_{\leq\tau}$ with $U_\lambda=0$.
\end{cor}

\begin{proof}
  The forward implication holds by Theorem~\ref{theo:WR=1-hom-char}.  To prove the reverse implication, assume $V \in \Ob \cI^\Gamma_{\le \lambda}$ satisfies \eqref{eq:lWm-hom-char-fd} for all irreducible finite-dimensional $U \in \Ob \cI^\Gamma_{\le \lambda}$ with $U_\lambda=0$. As in the proof of Theorem~\ref{theo:WR=1-hom-char}, we see that $V = \cU((\g \otimes A)^\Gamma) V_\lambda$.  The map $\epsilon_V$ of the proof of Proposition~\ref{prop:right-adjoint} is surjective since $V$ is generated by its highest weight space.  Thus we have a short exact sequence
  \begin{equation} \label{eq:K-seq}
    0 \to K \to \bW^\Gamma_\lambda V_\lambda \xrightarrow{\epsilon_V} V \to 0.
  \end{equation}
  By Corollary~\ref{cor:WM-fin-dim}, $\dim \bW^\Gamma_\lambda V_\lambda < \infty$ and so $\dim K < \infty$ and $K_\lambda = 0$. If $K \ne 0$, then there exists some irreducible finite-dimensional $U \in \Ob \cI^\Gamma_{\le \lambda}$ with $U_\lambda=0$ such that $\Hom_{(\g \otimes A)^\Gamma} (K,U) \ne 0$.  A straightforward argument using the long exact sequence obtained from~\eqref{eq:K-seq} by applying the contravariant left exact functor $\Hom_{(\g \otimes A)^\Gamma}(-,U)$ then yields a contradiction.
  \details{Applying the contravariant left exact functor $\Hom_{(\g \otimes A)^\Gamma}(-,U)$ to the short exact sequence~\eqref{eq:K-seq} yields the long exact sequence
  \begin{multline*}
    0 \to \Hom_{(\g \otimes A)^\Gamma} (V,U) \to \Hom_{(\g \otimes A)^\Gamma} (\bW^\Gamma_\lambda V_\lambda,U) \\
    \to \Hom_{(\g \otimes A)^\Gamma} (K,U) \to \Ext^1_{(\g \otimes A)^\Gamma} (V,U) \to \cdots.
  \end{multline*}
  Since $\Hom_{(\g \otimes A)^\Gamma} (\bW^\Gamma_\lambda V_\lambda,U)=0$ by assumption, we see that $\Hom_{(\g \otimes A)^\Gamma} (K,U)$ injects into $\Ext^1_{(\g \otimes A)^\Gamma} (V,U)$.  But this contradicts the assumption on $V$, which implies that the latter space is zero.}  Thus $K=0$ and hence $V \cong \bW^\Gamma_\lambda V_\lambda = \bW^\Gamma_\lambda \bR^\Gamma_\lambda V$.
\end{proof}

\begin{defin}[Map $\bV^\Gamma_\lambda$] \label{def:V-Gamma-lambda}
  Since $\bA^\lambda_\Gamma$ is a finitely generated commutative algebra, any irreducible finite-dimensional $\bA^\lambda_\Gamma$-module $M$ has dimension one.  For such an $\bA^\lambda_\Gamma$-module $M$, we let $\bV^\Gamma_\lambda M$ denote the unique irreducible quotient of $\bW_\lambda^\Gamma M$, which is finite-dimensional by Corollary~\ref{cor:WM-fin-dim}.  This defines a map $\bV^\Gamma_\lambda$ from the set of irreducible $\bA^\lambda_\Gamma$-modules to the set of irreducible finite-dimensional $(\g \otimes A)^\Gamma$-modules.
\end{defin}

Recall the definition of $V(\psi)$ and $V^\Gamma(\psi)$ from Definition~\ref{def:EMA-irreducibles}.

\begin{defin}[Modules $M(\psi)$ and $M^\Gamma(\psi)$] \label{def:M-psi}
  For $\psi \in \cE_\llift$ (respectively, $\psi \in \cE^\Gamma_\lambda$), define $M(\psi) \defeq \bR_\llift V(\psi)$ (respectively, $M^\Gamma(\psi) \defeq \bR_\lambda^\Gamma V^\Gamma(\psi)$).
\end{defin}

\begin{prop} \label{prop:irred-wt-space-modules}
  \begin{enumerate}
    \item \label{prop-item:gWm-zero} The global Weyl module $W^\Gamma(\lambda)$ is the zero module, and hence the algebra $\bA^\lambda_\Gamma$ is the zero algebra, if $\lambda \in \Lambda_\Gamma^+$ is not the restriction of some element of $\Lambda^+$.
    \item \label{prop-item:irred-functor} For all $\psi \in \cE^\Gamma_\lambda$, we have $V^\Gamma(\psi) \cong \bV_\lambda^\Gamma M^\Gamma(\psi)$.
    \item \label{prop-item:VR-bijection} For all $\lambda \in \Lambda^+_\Gamma$, the maps $\bV^\Gamma_\lambda$ and $\bR^\Gamma_\lambda$ induce mutually inverse bijections between the set of irreducible finite-dimensional $(\g \otimes A)^\Gamma$-modules whose highest weight as a $\g^\Gamma$-module is $\lambda$ and the set of irreducible finite-dimensional $\bA^\lambda_\Gamma$-modules.
    \item The map $\psi \mapsto [M^\Gamma(\psi)]$ is a bijection from $\cE^\Gamma_\lambda$ to the set of isomorphism classes of irreducible finite-dimensional $\bA^\lambda_\Gamma$-modules.
  \end{enumerate}
\end{prop}

\begin{proof}
  \begin{asparaenum}
    \item Suppose $\lambda \in \Lambda_\Gamma^+$.  The irreducible (one-dimensional) $\bA^\lambda_\Gamma$-modules are of the form $M=\bA^\lambda_\Gamma/\sm$ for some maximal ideal $\sm$ of $\bA^\lambda_\Gamma$.  By Nakayama's Lemma, $\bW^\Gamma_\lambda M = W^\Gamma(\lambda)/\sm W^\Gamma(\lambda)$ is zero if and only if the global Weyl module $W^\Gamma(\lambda)$ is zero.  By Corollary~\ref{cor:WM-fin-dim}, $\bW^\Gamma_\lambda M$ is finite-dimensional.  If it is nonzero, it has some nonzero irreducible finite-dimensional quotient $V$ whose highest $\g^\Gamma$ weight is also $\lambda$.  By \cite[Th.~5.5]{NSS12}, $V$ is a tensor product of evaluation representations (corresponding to representations of $\g$).  Thus, its highest weight must be a restriction of a weight of $\g$.

    \item As in the proof of Proposition~\ref{prop:right-adjoint}, we have a (nonzero) surjective map
      \[
        \bW^\Gamma_\lambda M^\Gamma(\psi) = \bW^\Gamma_\lambda \bR^\Gamma_\lambda V^\Gamma(\psi) \twoheadrightarrow V^\Gamma(\psi).
      \]
      Thus $V^\Gamma(\psi)$ must be isomorphic to the unique irreducible quotient $\bV_\lambda^\Gamma M^\Gamma(\psi)$ of $\bW^\Gamma_\lambda M^\Gamma(\psi)$.

    \item Let $\lambda \in \Lambda^+_\Gamma$.  By \cite[Th.~5.5]{NSS12}, every irreducible finite-dimensional $(\g \otimes A)^\Gamma$-module with highest weight $\lambda$ is of the form $V^\Gamma(\psi)$ for some $\psi \in \cE^\Gamma_\lambda$.  Thus, by part~\eqref{prop-item:irred-functor}, we have that $\bV^\Gamma_\lambda \bR^\Gamma_\lambda$ is the identity map on the set of such modules.  Now, for an irreducible $\bA^\lambda_\Gamma$-module $M$, we have that
        \[
          \bR^\Gamma_\lambda \bV^\Gamma_\lambda M = (\bV^\Gamma_\lambda M)_\lambda = (\bW^\Gamma_\lambda M)_\lambda = \bR^\Gamma_\lambda \bW^\Gamma_\lambda M = M.
        \]

     \item By \cite[Th.~5.5]{NSS12}, the map $\psi \mapsto V^\Gamma(\psi)$ is bijection from $\cE^\Gamma_\lambda$ to the set of irreducible finite-dimensional $(\g \otimes A)^\Gamma$-modules with highest weight $\lambda$.  Then the result follows from part~\eqref{prop-item:VR-bijection}. \qedhere
  \end{asparaenum}
\end{proof}

\begin{rem} \label{rem:restricted-weights}
  The condition in Proposition~\ref{prop:irred-wt-space-modules}\eqref{prop-item:gWm-zero} that $\lambda$ be the restriction of a $\g$ weight is only relevant in the case where $\g$ is of type $A_{2n}$ (and $\Gamma$ acts nontrivially on $\g$).  In this case, the restriction condition amounts to requiring that when $\lambda$ is written as a sum of fundamental weights, its coefficient for the fundamental weight corresponding to the short root be even.  If $\g$ is not of type $A_{2n}$, then the restriction map $\Lambda^+ \to \Lambda_\Gamma^+$ is surjective (see Lemma~\ref{lem:height-wt}\eqref{lem-item:root-restriction}).
\end{rem}

The \hyperlink{proof:projective-free}{proof} of the following theorem will be given at the end of Section~\ref{sec:local-Weyl-modules}.  Part~\eqref{theo-item:loop-global-Weyl-module-free} was first proved in \cite[Th.~6.5]{FMS11}.  However, we will provide some details omitted there.

\begin{theo} \label{theo:global-Weyl-module-projective}
  Suppose $\kk = \C$ and assume that $\Gamma$ is abelian, acting freely on $X_\rat$ and acting on $\g$ by diagram automorphisms.
  \begin{asparaenum}
    \item \label{theo-item:fund-weight-proj-property} If $A$ is the coordinate algebra of a smooth complex algebraic variety and $\lambda$ is a fundamental weight of $\g^\Gamma$, then the global Weyl module $W^\Gamma(\lambda)$ is a projective $\bA^\lambda_\Gamma$-module.  If, in addition, $\bA^\lambda_\Gamma$ is a generalized Laurent polynomial ring $\C[t_1^{\pm 1},\dotsc,t_n^{\pm 1},s_1,\dotsc,s_m]$, $n,m \in \N$, then the global Weyl module $W^\Gamma(\lambda)$ is a free $\bA^\lambda_\Gamma$-module.

    \item \label{theo-item:loop-global-Weyl-module-free} If $A=\C[t^{\pm 1}]$, then $W^\Gamma(\lambda)$ is a free $\bA^\lambda_\Gamma$-module for all $\lambda \in \Lambda_\Gamma^+$, and its rank is equal to the dimension of any local Weyl module.
  \end{asparaenum}
\end{theo}

\begin{rem}
  The condition that $\bA^\lambda_\Gamma$ is a generalized Laurent polynomial ring can be verified in specific cases using the explicit realization of $\bA^\lambda_\Gamma$ given in Section~\ref{sec:A-lambda-Gamma} (see Theorem~\ref{theo:Alambda-isom}).
\end{rem}

%%%%%%%%%%%%%%%%%%%%%%%%%%%%%%%%%%%%%%%%%%%%%%%%%%%%%%%%%%%%%%%%%%%
%
\section{Twisting functors} \label{sec:twisting}
%
%%%%%%%%%%%%%%%%%%%%%%%%%%%%%%%%%%%%%%%%%%%%%%%%%%%%%%%%%%%%%%%%%%%

In this section we recall the twisting functors introduced in \cite{FKKS12} and prove some facts related to them that will be used in the sequel.  We continue to assume that $\Gamma$ is cyclic, acts freely on $X_\rat$, and acts faithfully on $\g$ by diagram automorphisms (see Remark~\ref{rem:cyclic-assumption}).

We define the \emph{support} of an ideal $J$ of $A$ to be
\[\label{def:supp-ideal}
  \Supp J \defeq \{\sm \in \maxSpec A\ |\ J \subseteq \sm\} \cong \maxSpec (A/J).
\]
Note that the support of an ideal is often defined to be the set of prime (rather than maximal) ideals containing it.  So our definition is more restrictive.  When we refer to the \emph{codimension} of an ideal of an algebra, we mean its codimension as a $\kk$-vector space (and not, for instance, some geometric codimension).

\begin{lem} \label{lem:ideal-form}
  All ideals of $(\g \otimes A)^\Gamma$ are of the form $(\g \otimes J)^\Gamma = \bigoplus_{\xi \in \Xi} \g_\xi \otimes J_{-\xi}$, where $J=\bigoplus_{\xi \in \Xi} J_\xi$ is a $\Gamma$-invariant ideal of $A$.
\end{lem}

\begin{proof}
  This is proved in \cite[Prop.~7.1]{Sav12} in the more general setting of Lie superalgebras.
\end{proof}

\begin{defin}[Support]\label{def:support-module}
  It follows from Lemma~\ref{lem:ideal-form} that the annihilator of any $(\g \otimes A)^\Gamma$-module $V$ is of the form $(\g \otimes J)^\Gamma$ for a unique $\Gamma$-invariant ideal $J$ of $A$.  We denote this ideal $J$ by $\Ann_A^\Gamma V$.  Thus
  \[
    \Ann_A^\Gamma V \defeq \langle f \in A\ |\ uV=0 \text{ for all } u \in (\g \otimes A)^\Gamma \cap (\g \otimes f) \rangle.
  \]
  We define the \emph{support} of $V$ to be
  \[
    \Supp_A^\Gamma V \defeq \Supp \Ann_A^\Gamma V.
  \]
  When the group $\Gamma$ is trivial, we will often omit the superscript $\Gamma$.
\end{defin}

Let $X_*$\label{def:X-star} denote the set of finite subsets of $X_\rat$ that do not contain two points in the same $\Gamma$-orbit.

\begin{defin}[Categories $\cF_\bx$ and $\cF_\bx^\Gamma$]\label{def:F-x}
  For $\bx \in X_*$, let $\cF_\bx$ denote the full subcategory of the category of $(\g \otimes A)$-modules whose objects are finite-dimensional $(\g \otimes A)$-modules $V$ with $\Supp_A V \subseteq \bx$.  Similarly, let $\cF_\bx^\Gamma$ be the full subcategory of the category of $(\g \otimes A)^\Gamma$-modules whose objects are finite-dimensional $(\g \otimes A)^\Gamma$-modules $V$ with $\Supp_A^\Gamma V \subseteq \Gamma \cdot \bx$.
\end{defin}

If $V$ is a finitely supported $(\g \otimes A)$-module and $V^\Gamma$ denotes the corresponding $(\g \otimes A)^\Gamma$-module obtained by restriction, then it is clear that $\Supp_A^\Gamma V^\Gamma = \Gamma \cdot \Supp_A V$.

\begin{defin}[{Twisting functors $\bT$ and $\bT_\bx$ (\cite[Def.~2.8]{FKKS12})}]\label{def:twisting-functor}
  We have a natural \emph{twisting functor} $\bT$ from the category of $(\g \otimes A)$-modules to the category of $(\g \otimes A)^\Gamma$-modules, defined by restriction.  For any $\bx \in X_*$, we have the induced functor $\bT_\bx \colon \cF_\bx \to \cF_\bx^\Gamma$.
\end{defin}

\begin{prop}[{\cite[Th.~2.10]{FKKS12}}] \label{prop:twisting}
  For $\bx \in X_*$, the functor $\bT_\bx \colon \cF_\bx \to \cF_\bx^\Gamma$ is an isomorphism of categories.  Furthermore, for $\psi \in \cE^\Gamma$ with $\bx \in (\Supp \psi)_\Gamma$, we have $\bT_\bx(V(\psi_\bx)) = V^\Gamma(\psi)$.
\end{prop}

\begin{proof}
  This follows immediately from~\cite[Th.~2.10]{FKKS12} after the straightforward verification that $(\psi_\bx)^\Gamma = \psi$ in the notation of that theorem.
\end{proof}

Let $\omega_i$ be the fundamental weight of $\g$ corresponding to $i \in I$.  So we have $\Lambda^+ = \sum_{i \in I} \N \omega_i$.  Recall that the set of nodes $I_\Gamma$ of the Dynkin diagram of $\g^\Gamma$ can be naturally identified with the set of $\Gamma$-orbits in $I$ (and we will equate the two in what follows).  For $\bi \in I_\Gamma$, we define
\begin{equation} \label{eq:ebi-fbi-hbi} \ts
  e_\bi \defeq \sqrt{\kappa_\bi} \sum_{i \in \bi} e_i,\quad f_\bi \defeq \sqrt{\kappa_\bi} \sum_{i \in \bi} f_i,\quad h_\bi \defeq \kappa_\bi \sum_{i \in \bi} h_i,
\end{equation}
where $\kappa_\bi = 2$\label{def:kappa-bi} if $\g$ is of type $A_{2n}$, $\Gamma$ acts nontrivially on $\g$ and $\bi$ corresponds to the short root of $\g^\Gamma$ (which is of type $B_n$).  Otherwise, $\kappa_\bi=1$.  Then $\{e_\bi, f_\bi, h_\bi\}$ is an $\mathfrak{sl}_2$-triple for each $\bi \in I_\Gamma$ and these triples generate $\g^\Gamma$.  We refer the reader to \cite[\S8.3]{Kac90} for details.

We let $\alpha_\bi$ and $\omega_\bi$ denote the simple root and fundamental weight, respectively, of $\g^\Gamma$ corresponding to $\bi \in I_\Gamma$.  Thus
\begin{equation} \label{def:Lambda-Gamma}\ts
  \Lambda_\Gamma = \bigoplus_{\bi \in I_\Gamma} \Z \omega_\bi,\quad \Lambda_\Gamma^+ = \bigoplus_{\bi \in I_\Gamma} \N \omega_\bi, \quad Q_\Gamma = \bigoplus_{\bi \in I_\Gamma} \Z \alpha_\bi,\quad Q_\Gamma^+ = \bigoplus_{\bi \in I_\Gamma} \N \alpha_\bi
\end{equation}
are the integral weight lattice, dominant integral weight lattice, root lattice, and positive root lattice of $\g^\Gamma$ respectively.

We conclude this section with a lemma collecting some technical results that will be used in the sequel.

\begin{lem} \label{lem:height-wt}
  Suppose that $\Gamma$ acts nontrivially on $\g$ by diagram automorphisms.
  \begin{enumerate}
    \item \label{lem-item:root-restriction} For all $i \in I$, we have $\alpha_i|_{\h^\Gamma} = \alpha_{\Gamma i}$ and $\omega_i|_{\h^\Gamma} = \kappa_{\Gamma i} \omega_{\Gamma i}$.

%    \item For all $\lambda \in \Lambda$, we have $\hei_\Gamma (\lambda|_{\h^\Gamma}) = \hei \lambda$.

    \item \label{lem-item:height-wt} For all $\psi \in \cE^\Gamma$ and $\bx \in (\Supp \psi)_\Gamma$, we have $\hei_\Gamma \psi = \hei \psi_\bx$.

    \item \label{lem-item:category-I-restriction} Let $\llift \in \Lambda^+$ and set $\lambda \defeq \llift|_{\h^\Gamma}$.  Then, for $V \in \Ob \cI_{\le \llift}$, we have $\bT(V) \in \Ob \cI^\Gamma_{\le \lambda}$ and $V_\llift = \bT(V)_{\lambda}$ as vector spaces.
  \end{enumerate}
\end{lem}

\begin{proof}
  \begin{asparaenum}
    \item This is a straightforward computation and will be omitted.

    \item Suppose $\psi \in \cE^\Gamma$ and choose $\bx \in (\Supp \psi)_\Gamma$.  Then
      \[
        \hei_\Gamma \psi = \hei_\Gamma \left((\wt \psi_\bx)|_{\h^\Gamma}\right) = \hei \wt \psi_\bx = \hei \psi_\bx,
      \]
      where the second equality follows from part~\eqref{lem-item:root-restriction}.

    \item Let $V \in \Ob \cI_{\le \llift}$.  Then the $\g$-weights of $V$ lie in $\llift - Q^+$, where $Q^+$ is the positive root lattice of $\g$.  By part~\eqref{lem-item:root-restriction}, the $\g^\Gamma$-weights of $\bT(V)$ lie in $\lambda - Q^+_\Gamma$ and so $\bT(V) \in \Ob \cI^\Gamma_{\le \lambda}$.  The second part of the statement follows easily. \qedhere
  \end{asparaenum}
\end{proof}

%%%%%%%%%%%%%%%%%%%%%%%%%%%%%%%%%%%%%%%%%%%%%%%%%%%%%%%%%%%%%%%%%%%
%
\section{Local Weyl modules} \label{sec:local-Weyl-modules}
%
%%%%%%%%%%%%%%%%%%%%%%%%%%%%%%%%%%%%%%%%%%%%%%%%%%%%%%%%%%%%%%%%%%%

In this section we define local Weyl modules and prove some of their important properties.  We continue to assume that $\Gamma$ is cyclic, acts freely on $X_\rat$ and acts faithfully on $\g$ by diagram automorphisms.  We will also assume that the action of $\Gamma$ on $\g$ is nontrivial, since the case of trivial action has been covered in \cite{CFK10}.  Recall the definition of the $M(\psi)$ and $M^\Gamma(\psi)$ from Definition~\ref{def:M-psi}.

\begin{defin}[Local Weyl modules $W(\psi)$ and $W^\Gamma(\psi)$]\label{def:local-weyl}
  Let $\psi \in \cE$ (respectively, $\psi \in \cE^\Gamma$) and set $\llift = \wt \psi$ (respectively, $\lambda = \wt_\Gamma \psi$).  The corresponding \emph{untwisted} (respectively, \emph{twisted}) \emph{local Weyl module} is $W(\psi) \defeq \bW_\llift M(\psi)$ (respectively, $W^\Gamma(\psi) \defeq \bW_\lambda^\Gamma M^\Gamma(\psi)$).
\end{defin}

\begin{lem} \label{lem:ideal-gen-by-zero-component}
  If $J = \bigoplus_{\xi \in \Xi} J_\xi$ is a $\Gamma$-invariant ideal of $A$, then $A_\tau J_\xi = J_{\tau + \xi}$ for all $\tau,\xi \in \Xi$.
\end{lem}

\begin{proof}
  Fix $\tau, \xi \in \Xi$.  Since $J$ is an ideal, we have $A_\tau J_\xi \subseteq J_{\tau + \xi}$.  Now choose $a \in J_{\tau + \xi}$.  By \cite[Lem.~4.4]{NS11}, we have $A_\tau A_{-\tau} = A_0$.  Thus $J_{\tau + \xi} = A_0 J_{\tau + \xi} = A_\tau A_{-\tau} J_{\tau + \xi} \subseteq A_\tau J_\xi$.
\end{proof}

\begin{lem} \label{lem:EMA-ideal}
  Suppose that $J_0$ is an ideal of $A_0 = A^\Gamma$.  Then the ideal of $(\g \otimes A)^\Gamma$ generated by $\g^\Gamma \otimes J_0$ is $(\g \otimes J)^\Gamma$, where $J= \bigoplus_{\xi \in \Xi} A_\xi J_0$ is the ideal of $A$ generated by $J_0$.
\end{lem}

\begin{proof}
  By Lemma~\ref{lem:ideal-form}, the ideal of $(\g \otimes A)^\Gamma$ generated by $\g^\Gamma \otimes J_0$ is of the form $(\g \otimes J')^\Gamma = \bigoplus_{\xi \in \Xi} \g_\xi \otimes J'_{-\xi}$ for some $\Gamma$-invariant ideal $J'$ of $A$.  Clearly we have $J_0 \subseteq J'_0$ and so $J \subseteq J'$.  By Lemma~\ref{lem:ideal-gen-by-zero-component}, $(\g \otimes J)^\Gamma$ is an ideal of $(\g \otimes A)^\Gamma$ containing $\g^\Gamma \otimes J_0$, and so we must have $J' \subseteq J$.
\end{proof}

For $\psi \in \cE$, define
\begin{equation} \label{def:J-psi} \ts
  J(\psi) \defeq \prod_{\sm \in \Supp \psi} \sm.
\end{equation}
If $\psi \in \cE^\Gamma$, then $J(\psi)$ is clearly $\Gamma$-invariant.  Recall the definition of $Y_\Gamma$ for a $\Gamma$-invariant subset $Y \subseteq X_\rat$ given in Definition~\ref{def:wt-ht}.

\begin{prop} \label{prop:local-Weyl-module-annihilated}
  For $\psi \in \cE^\Gamma$, the ideal $(\g \otimes J(\psi)^k)^\Gamma$ annihilates the local Weyl module $W^\Gamma(\psi)$ for some positive integer $k$.  In particular,  $W^\Gamma(\psi) \in \cF^\Gamma_\bx$ for $\bx \in (\Supp \psi)_\Gamma$.
\end{prop}

\begin{proof}
  Let $\lambda = \wt_\Gamma \psi$, fix a nonzero element $m \in M^\Gamma(\psi)$ and let $J = J(\psi)$.

  First suppose that $\g$ is not of type $A_{2n}$.  Let $\theta$ be the highest root of $\g$ and let $\{e_\theta, f_\theta, h_\theta\}$ be a corresponding $\mathfrak{sl}_2$-triple.  Then $\{e_\theta, f_\theta, h_\theta\} \subseteq \g^\Gamma$ and this set forms an $\mathfrak{sl}_2$-triple corresponding to $\theta_\Gamma \defeq \theta|_{\h^\Gamma}$, which is the highest root of $\g^\Gamma$.
  \details{Using Lemma~\ref{lem:height-wt}\eqref{lem-item:root-restriction} and the explicit description of the highest roots given, for instance, in \cite[\S12.2, Table~2]{Hum72}, we see that $\theta|_{\h^\Gamma}$ is the highest root of $\g^\Gamma$ and that $\theta$ is the \emph{only} root of $\g$ that restricts to the highest root of $\g^\Gamma$.  Thus $\Gamma$ must act trivially on $\g_\theta$ (similarly, on $\g_{-\theta}$) and the result follows.}
  One can see from Proposition~\ref{weyl:gen-rel} (see also \cite[Prop.~4]{CFK10}) that there is a map of $(\g^\Gamma \otimes A^\Gamma)$-modules from the untwisted global Weyl module for $\g^\Gamma \otimes A^\Gamma$ to the twisted global Weyl module for $(\g \otimes A)^\Gamma$ that maps $w_\lambda$ to $w_\lambda^\Gamma$.  Thus, applying \cite[Prop.~9]{CFK10} (see also \cite[Prop.~4.1]{CFS08}) with $\g^\Gamma$ in place of $\g$ and $A^\Gamma$ in place of $A$, we have
  \begin{equation}\label{f-theta-annihilate}
    \left(f_\theta \otimes a^{\lambda(h_{\theta})}\right)(w_\lambda^\Gamma \otimes m) = 0,\quad a \in J^\Gamma.
  \end{equation}
  Since $[f_\theta, \n^-]=0$ and $\cU\left((\n^- \otimes A)^\Gamma\right) (w_\lambda^\Gamma \otimes m) = W^\Gamma(\psi)$,  we have
  \[
    \left(f_\theta \otimes a^{\lambda(h_\theta)}\right) W^\Gamma(\psi)=0, \quad a \in J^\Gamma.
  \]
  Since $\g^\Gamma$ is simple and the annihilator of $W^\Gamma(\psi)$ in $\g^\Gamma \otimes A^\Gamma$ is an ideal, it follows that $\g^\Gamma \otimes a^{\lambda(h_\theta)}$ annihilates $W^\Gamma(\psi)$ for all $a \in J^\Gamma = J_0$.  By Lemma~\ref{lem:ideal-form}, the annihilator of $W^\Gamma(\psi)$ is of the form $(\g \otimes J')^\Gamma$ for some $\Gamma$-invariant ideal $J'$ of $A$.  By the above, we have $a^{\lambda(h_\theta)} \in J'_0$ for all $a \in J_0$.  By \cite[Th.~5]{AKL94}, this implies that $(J_0)^{\lambda(h_\theta)} \subseteq J'_0$.  Let $K$ be the ideal of $A$ generated by $(J_0)^{\lambda(h_\theta)}$.  For $\xi \in \Xi$, we have
  \[
    K_\xi = A_\xi (J_0)^{\lambda(h_\theta)} = J_\xi (J_0)^{{\lambda(h_\theta)}-1} = (J^{\lambda(h_\theta)})_\xi,
  \]
  where the second equality holds by Lemma~\ref{lem:ideal-gen-by-zero-component} and the third equality holds by \cite[Lem.~4.4]{NS11}.  Thus $K=J^{\lambda(h_\theta)}$.  By Lemma~\ref{lem:EMA-ideal}, the ideal of $(\g \otimes A)^\Gamma$ generated by $\g \otimes (J_0)^{\lambda(h_\theta)}$ is equal to $(\g \otimes K)^\Gamma$.  Thus we have $J^{\lambda(h_\theta)} = K \subseteq J'$.

  Now suppose that $\g$ is of type $A_{2n}$. Let $\beta_\Gamma$ denote the highest root of $\g^\Gamma$ (which is of type $B_n$) and $\beta_\Gamma^s$ the highest short root. We let $\{e_{\beta_\Gamma}, f_{\beta_\Gamma}, h_{\beta_\Gamma}\}$ be an $\mathfrak{sl}_2$-triple in $\g^\Gamma$ corresponding to $\beta_\Gamma$.  Since $\Gamma$ is of order two, $\g$ decomposes into $\g^\Gamma \oplus \g_1$. By \cite[Prop.~8.3d]{Kac90}, we know that $\g_1$ is a simple $\g^\Gamma$-module of highest weight $2\sum_{\bi\in I_\Gamma}\alpha_\bi$ (\cite[\S8.3, Table]{Kac90}), which is equal to $2\beta_\Gamma^s$ (see, for example, \cite[\S12.2, Table 2]{Hum72}).

  Recall that $ \{ e_{i}, f_i , h_i\}$ is an $\mathfrak{sl}_2$-triple in $\g$ corresponding to the simple root $\alpha_i \in R^+$, $i \in I$. Then, for $1 \leq i \leq n$,
  \begin{equation} \label{eq:simple-root-spaces}
    f_{\mathbf{1}} = \sqrt{\kappa_{\mathbf{1}}}(f_1+f_{2n}) \in \g^\Gamma \text{ and }f_i - f_{2n+1 - i} \in \g_{1}.
  \end{equation}
  Before completing the proof of the proposition, we develop some additional ideas specific to the $A_{2n}$ case.  Since $\g^\Gamma$ is of type $B_n$, $2\beta_\Gamma^s-\beta_\Gamma$ is the simple root $\alpha_{\mathbf{1}}$.
  On the other hand, it follows from Lemma~\ref{lem:height-wt}\eqref{lem-item:root-restriction} that the weight of $f_1-f_{2n}$ is $-\alpha_{\mathbf1}$.  Consider the vector
  \[
   [f_{\beta_\Gamma}, f_1-f_{2n}]\in(\g_1)_{-2\beta_\Gamma^s},
  \]
  which we claim is nonzero. Assuming the claim, we see that $ [f_{\beta_\Gamma}, f_1-f_{2n}]$ spans the lowest weight space of $\g_1$ as a $\g^\Gamma$-module, since simple finite-dimensional modules for $B_n$ are self-dual. Therefore
  \begin{equation}\label{fbg-generates}
    \left [\cU(\n_+^\Gamma), [f_{\beta_\Gamma}, f_1-f_{2n}]\right]=\g_1,
  \end{equation}
  that is, $[f_{\beta_\Gamma}, f_1-f_{2n}]$ generates $\g_1$ as an $\n_+^\Gamma$-module.

  To see that $[f_{\beta_\Gamma}, f_1-f_{2n}]\ne 0$, let $w\in (\g_1)_{-2\beta_\Gamma^s}$.  Since $[h_{\beta_\Gamma},w]=-2\beta_\Gamma^s(h_{\beta_\Gamma})w=-2w$, it follows that $[e_{\beta_\Gamma}, w]$ is a nonzero weight vector of weight $-\alpha_{\mathbf1}$.  Now this weight space has dimension 1, for the following reason.  First, a straightforward computation shows that
  \[
    s_{\mathbf{2}}s_{\mathbf{3}} \dotsm s_{\mathbf{n-1}}s_{\mathbf{n}}s_{\mathbf{n-1}} \dotsm s_{\mathbf{3}}s_{\mathbf{2}}s_{\mathbf{1}}(-\alpha_{\mathbf1}) = \alpha_{\mathbf{1}}+2\alpha_{\mathbf{2}}+\cdots+2\alpha_{\mathbf{n}} = \beta_\Gamma
  \]
  and so $-\alpha_{\mathbf1}$ is in the Weyl group orbit of $\beta_\Gamma$.  Now the dimension of the $\beta_\Gamma$ weight space of the Verma module of highest weight $2\beta_\Gamma^s$ is one, the number of ways of writing $2\beta_\Gamma^s - \beta_\Gamma = \alpha_{\mathbf{1}}$ as a sum of positive roots.  Thus $V(2\beta_\Gamma^s)_{\beta_\Gamma}$ has dimension at most 1.  On the other hand, since $\beta_\Gamma$ is a dominant weight lying below the highest weight of $V(2\beta_\Gamma^s)$, the dimension is also at least 1, and thus $\dim_\kk V(2\beta_\Gamma^s)_{\beta_\Gamma}=1$.  Because $\beta_\Gamma$ is Weyl conjugate to $-\alpha_\mathbf{1}$, it follows that $\dim V(2\beta_\Gamma^s)_{-\alpha_{\mathbf{1}}} = 1$ and so $[e_{\beta_\Gamma}, w]$ is a nonzero constant multiple of $f_1-f_{2n}$. Since $[f_{\beta_\Gamma}, [e_{\beta_\Gamma},w]]$ is a nonzero multiple of $w$, we see that $[f_{\beta_\Gamma}, f_1-f_{2n}]\ne 0$, which proves the claim and establishes \eqref{fbg-generates}.

  Now, the set of weights $\mu_\Gamma\in \Lambda_\Gamma$ with $-\beta_\Gamma>\mu_\Gamma>-2\beta_\Gamma^s$ is empty since the difference $2\beta_\Gamma^s-\beta_\Gamma$ is the simple root $\alpha_\mathbf{1}$, and so it follows that
  \begin{equation}\label{n-1-minus}
    [(\n_-)_1, f_{\beta_\Gamma}]=\Span_\kk  [f_{\beta_\Gamma}, f_1-f_{2n}].
  \end{equation}
  Finally, we observe that
  \begin{equation} \label{n-minus-rk}
    [\n_-,[f_{\beta_\Gamma}, f_1-f_{2n}]] = 0.
  \end{equation}
  To see this, note that $[\n_-^\Gamma, [f_{\beta_\Gamma}, f_1-f_{2n}]]=0$ since any vector in this space would be a vector in $\g_1$ of $\g^\Gamma$-weight strictly lower than the lowest weight. On the other hand, $[(\n_-)_1, [f_{\beta_\Gamma}, f_1-f_{2n}]]=0$ since any vector in this space would be a vector in $\g^\Gamma$ of $\g^\Gamma$-weight strictly lower than $-2\beta_\Gamma^s\le-\beta_\Gamma$.

  To complete the proof of the proposition in the case where $\g$ is of type $A_{2n}$, we will generalize the arguments used in \cite[Prop.~4.1]{CFS08}, where the proposition was proved for the twisted loop algebra. Recall that we have decompositions $\g = \g^\Gamma \oplus  \g_{1}$ and $J=J_0\oplus J_1$, and so again by \cite[Lem.~4.4]{NS11}, we have
  \[
    (\g\otimes J^r)^\Gamma=(\g^\Gamma \otimes (J_0)^r)\oplus (\g_1\otimes J_1(J_0)^{r-1}) \quad \text{for } r \ge 1.
  \]
  Thus it suffices to show that
  \begin{gather} \label{6-a2n-sts1}
    \left( \g^\Gamma\otimes (J_0)^k \right) W^\Gamma(\psi) = 0\quad\text{and} \\
    \label{6-a2n-sts2}
    \left( \g_1\otimes J_1(J_0)^k \right) W^\Gamma(\psi) = 0,
  \end{gather}
  for sufficiently large $k$.  As in \eqref{f-theta-annihilate} we obtain
  \[ \label{eq:theta-action2}
    \left(f_{\beta_\Gamma} \otimes a^{ \lambda(h_{\beta_\Gamma})}\right)(w_\lambda^\Gamma \otimes m) = 0,\quad a \in J_0.
  \]
  Using the fact that $f_{\beta_\Gamma}$ is a lowest weight vector for the adjoint representation of $\g^\Gamma$, it follows, once again using \cite[Th.~5]{AKL94}, that
  \begin{equation} \label{eq:ggamma-ideal}
    \left(\g^\Gamma \otimes (J_0)^{\lambda(h_{\beta_\Gamma})}\right) (w_\lambda^\Gamma \otimes m) =0.
  \end{equation}
  In particular,
  \begin{equation} \label{eq:alphan-ideal}
    \left(f_\mathbf{1} \otimes   (J_0)^{\lambda(h_{\beta_\Gamma})}\right)(w_\lambda^\Gamma \otimes m) = 0.
  \end{equation}
  Since $[h_1-h_{2n},f_{\mathbf{1}}] = -(2+\delta_{1,n})\sqrt{\kappa_{\mathbf 1}}(f_1-f_{2n}) \in \g_1$, we have
  \begin{align*}
    \left( (f_1 - f_{2n}) \otimes J_1 (J_0)^{\lambda(h_{\beta_\Gamma})} \right) (w_\lambda^\Gamma \otimes m)
    &= \left( [h_1-h_{2n}, f_{\mathbf{1}}] \otimes J_1 (J_0)^{\lambda(h_{\beta_\Gamma})} \right) (w_\lambda^\Gamma \otimes m) \\
    &= \left[ (h_1 - h_{2n}) \otimes J_1, f_{\mathbf{1}} \otimes (J_0)^{\lambda(h_{\beta_\Gamma})} \right] (w_\lambda^\Gamma \otimes m).
  \end{align*}
  Thus
  \begin{equation}  \label{eq:yn-action}
    \left( (f_1 - f_{2n}) \otimes J_1 (J_0)^{\lambda(h_{\beta_\Gamma})} \right) (w_\lambda^\Gamma \otimes m) = 0,
  \end{equation}
  by~\eqref{eq:alphan-ideal} and the fact that $(h_1 - h_{2n}) \otimes J_1 \subseteq (\h \otimes A)^\Gamma$ acts by scalar multiplication on the highest weight vector $w_\lambda^\Gamma \otimes m$.

  By~\eqref{eq:ggamma-ideal} and~\eqref{eq:yn-action}, we have
  \begin{multline}\label{eq:lw-ideal}
    \left( [f_{\beta_\Gamma}, f_1 - f_{2n}] \otimes  J_{1} (J_0)^{2\lambda(h_{\beta_\Gamma})} \right) (w_\lambda^\Gamma \otimes m) \\
    = \left[f_{\beta_\Gamma}\otimes (J_0)^{\lambda(h_{\beta_\Gamma})}, (f_1 - f_{2n}) \otimes J_{1}(J_0)^{\lambda(h_{\beta_\Gamma})}\right] (w_\lambda^\Gamma\otimes m)=0.
  \end{multline}
  By~\eqref{n-minus-rk}, $[\n_-, [f_{\beta_\Gamma}, f_1-f_{2n}]]=0$. Thus, again using the fact that $W^\Gamma(\psi) = \cU((\n^- \otimes A)^\Gamma) (w_\lambda^\Gamma \otimes m)$, we have that
  \[
    \left( \left[f_{\beta_\Gamma}, f_1 -f_{2n}\right]\otimes J_{1}(J_0)^{2\lambda(h_{\beta_\Gamma})}\right) W^\Gamma(\psi) = 0.
  \]
  By \eqref{fbg-generates}, $[f_{\beta_\Gamma}, f_1 -f_{2n}]$ generates $\g_{1}$ as an $\n_{+}^{\Gamma}$-module, which now establishes~\eqref{6-a2n-sts2}.

  Combining~\eqref{n-1-minus}, \eqref{n-minus-rk} and~\eqref{eq:ggamma-ideal}, we see that
  \[
    \left( f_{\beta_\Gamma} \otimes  (J_0)^{k} \right)) W^\Gamma(\psi) \subseteq \cU((\n^- \otimes A)^{\Gamma}) \left( [f_{\beta_\Gamma}, f_1 -f_{2n}] \otimes  J_{1}(J_0)^{k} \right) (w_\lambda^\Gamma \otimes m)=0
  \]
  for $k\ge 2\lambda(h_{\beta_\Gamma})$, where the last equality follows by \eqref{eq:lw-ideal}.  Again since $f_{\beta_\Gamma}$ is a lowest weight vector for the adjoint representation of $\g^\Gamma$, we have
  \[
    (\g^\Gamma \otimes (J_0)^{k}) W^\Gamma(\psi) = 0 \quad \text{for} \quad k \geq 2\lambda(h_{\beta_\Gamma}),
  \]
  which establishes~\eqref{6-a2n-sts1} and completes the proof.
\end{proof}

\begin{lem} \label{lem:lWm-characterization}
  A $(\g \otimes A)^\Gamma$-module $V$ is isomorphic to the local Weyl module $W^\Gamma(\psi)$ if and only if it satisfies the following three conditions:
  \begin{enumerate}
    \item \label{lem-item:lWm-cat} $V \in \Ob \cI^\Gamma_{\le \lambda}$, where $\lambda = \wt_\Gamma \psi$;

    \item \label{lem-item:lWm-top} $\bR^\Gamma_\lambda V \cong M^\Gamma(\psi)$;

    \item \label{lem-item:lWm-hom} $\Hom_{\cI^\Gamma_{\le \lambda}} (V,U)=0$ and $\Ext^1_{\cI^\Gamma_{\le \lambda}}(V,U)=0$ for all irreducible finite-dimensional $U \in \Ob \cI^\Gamma_{\le \lambda}$ with $U_\lambda=0$.
  \end{enumerate}
\end{lem}

\begin{proof}
  If $V$ satisfies the conditions in the lemma, then by Corollary~\ref{cor:local-weyl-module-char-fd} we have $V \cong \bW^\Gamma_\lambda \bR^\Gamma_\lambda V \cong \bW^\Gamma_\lambda M^\Gamma(\psi) = W^\Gamma(\psi)$.

  Conversely, the local Weyl module $W^\Gamma(\psi)$ satisfies~\eqref{lem-item:lWm-cat} by definition of the functor $\bW^\Gamma_\lambda$.  We have $\bR^\Gamma_\lambda W^\Gamma(\psi) = \bR^\Gamma_\lambda \bW^\Gamma_\lambda M^\Gamma(\psi) \cong M^\Gamma(\psi)$ by Lemma~\ref{prop:right-adjoint}\eqref{prop-item:RW=1}.  Then $W^\Gamma(\psi)$ satisfies~\eqref{lem-item:lWm-hom} by~\eqref{eq:W=WRW} and Theorem~\ref{theo:WR=1-hom-char} (or Corollary~\ref{cor:local-weyl-module-char-fd}).
\end{proof}

Fix $\psi \in \cE^\Gamma$ and $\bx \in (\Supp \psi)_\Gamma$.  Set $\llift = \wt \psi_\bx$ and $\lambda = \wt_\Gamma \psi = \llift|_{\h^\Gamma}$.  By \cite[Prop.~9]{CFK10}, there exists a positive integer $n_1$ such that $\g \otimes J(\psi_\bx)^{n_1}$ annihilates the untwisted local Weyl module $\bW_\llift M(\psi_\bx)$.  By Proposition~\ref{prop:local-Weyl-module-annihilated}, there exists a positive integer $n_2$ such that $(\g \otimes J(\psi)^{n_2})^\Gamma$ annihilates the twisted local Weyl module $\bW_{\lambda}^\Gamma M^\Gamma(\psi)$.  Let $n=\max(n_1,n_2)$.  We have a sequence of isomorphisms
\begin{multline*} \ts
  (\g \otimes A)^\Gamma/(\g \otimes J(\psi)^n)^\Gamma \cong (\g \otimes A/J(\psi)^n)^\Gamma \cong \left( \bigoplus_{x \in \Supp \psi} \g \otimes A/\sm_x^n \right)^\Gamma \\ \ts
  \cong \bigoplus_{x \in \bx} \g \otimes A/\sm_x^n \cong (\g \otimes A)/(\g \otimes J(\psi_\bx)^n).
\end{multline*}

\begin{prop} \label{prop:equivalence-of-lWm-defs}
  Suppose that $\Gamma$ is abelian, acts freely on $X_\rat$ and acts on $\g$ by diagram automorphisms. Then $W^\Gamma(\psi) = \bT (W(\psi_\bx))$ for all $\psi \in \cE^\Gamma$ and $\bx \in (\Supp \psi)_\Gamma$.  In other words, the twisted local Weyl module is obtained by restriction from an untwisted local Weyl module and thus isomorphic to the twisted local Weyl module as defined in \cite[Def.~3.7]{FKKS12}.
\end{prop}

\begin{proof}
  Fix $\psi \in \cE^\Gamma$ and $\bx \in (\Supp \psi)_\Gamma$.  By Proposition~\ref{prop:local-Weyl-module-annihilated}, $W^\Gamma(\psi) \in \cF_\bx^\Gamma$.  Let $\llift = \wt \psi_\bx$ and $\lambda = \wt_\Gamma \psi$, so that $\lambda = \llift|_{\h^\Gamma}$.  Then $W(\psi_\bx) \in \Ob \cI_{\le \llift}$ and so, by Lemma~\ref{lem:height-wt}\eqref{lem-item:category-I-restriction}, $\bT(W(\psi_\bx)) \in \Ob \cI^\Gamma_{\le \lambda}$.

  Now suppose $V^\Gamma(\varphi)$, $\varphi \in \cE^\Gamma$, is an irreducible finite-dimensional object of $\cI_{\le \lambda}^\Gamma$ with $V^\Gamma(\varphi)_\lambda=0$.  This implies that $\wt_\Gamma \varphi \in \lambda - Q^+_\Gamma = \wt_\Gamma \psi - Q^+_\Gamma$.  Thus, by Lemma~\ref{lem:height-wt}\eqref{lem-item:height-wt}, we have $\hei \varphi_\bx < \hei \psi_\bx$.  Enlarging $\bx$ if necessary, we may assume that $\Supp \varphi \subseteq \Gamma \cdot \bx$ (i.e.\ $V^\Gamma(\varphi) \in \cF_\bx^\Gamma$).  For $\ell=0,1$, we have
  \[
    \Ext^\ell_{\cF_\bx^\Gamma}(\bT_\bx(W(\psi_\bx)), V^\Gamma(\varphi)) = \Ext^\ell_{\cF_\bx^\Gamma}(\bT_\bx(W(\psi_\bx)), \bT_\bx(V(\varphi_\bx))) = \Ext^\ell_{\cF_\bx}(W(\psi_\bx), V(\varphi_\bx))=0,
  \]
  where the first two equalities follow from Proposition~\ref{prop:twisting} and the last equality follows from \cite[Th.~4.5]{FKKS12}.

  Now, the weight space $W(\psi_\bx)_\llift$ is isomorphic to the weight space $V(\psi_\bx)_\llift$ as a $(\h \otimes A)$-module.  Restricting the action to $(\h \otimes A)^\Gamma$ and using the fact that $\bT_\bx (V(\psi_\bx)) = V(\psi)$, we see that the weight space $(\bT_\bx(W(\psi_\bx)))_\lambda$ is isomorphic to the weight space $V(\psi)_\lambda$ as a $(\h \otimes A)^\Gamma$-module, and hence as a $\bA^\lambda_\Gamma$-module.  Therefore $\bR^\Gamma_\lambda \bT_\bx (W(\psi_\bx)) = M^\Gamma(\psi)$.  Thus, by Lemma~\ref{lem:lWm-characterization}, we have $W^\Gamma(\psi) = \bT_\bx(W(\psi_\bx))$, as desired.
\end{proof}

We are now in a position to prove Theorem~\ref{theo:global-Weyl-module-projective}.

\begin{proof}[Proof of Theorem~\ref{theo:global-Weyl-module-projective}] \hypertarget{proof:projective-free}
  Choose $\llift \in \Lambda^+$ such that $\llift|_{\h^\Gamma} = \lambda$ (see Prop.~\ref{prop:irred-wt-space-modules}\eqref{prop-item:gWm-zero}).  By Proposition~\ref{prop:equivalence-of-lWm-defs}, the local Weyl modules $W^\Gamma(\psi)$, $\psi \in \cE_\lambda^\Gamma$, for $(\g \otimes A)^\Gamma$ are restrictions of local Weyl modules $W(\psi')$, $\psi' \in \cE_\llift$, for the untwisted map algebra $\g \otimes A$.  In addition, we have that $\bA^\lambda_\Gamma$ is a finitely generated algebra and $W^\Gamma(\lambda)$ is a finitely generated $\bA^\lambda_\Gamma$-module by Theorems~\ref{theo:A-fg} and~\ref{theo:finite-generated}.
  \begin{asparaenum}
    \item Assume that $A$ is the coordinate algebra of a smooth complex algebraic variety and $\lambda$ is a fundamental weight of $\g^\Gamma$.  Then, by \cite[Cor.~8]{CFK10}, the dimensions of the local Weyl modules $W(\psi')$, $\psi' \in \cE_\llift$, for $\g \otimes A$ are all equal.  By the above, this implies that the dimensions of the local Weyl modules $W^\Gamma(\psi)$, $\psi \in \cE_\lambda^\Gamma$, are all equal.  Then the result follows from Corollary~\ref{cor:constant-localized-dim}.

    \item Assume $A = \C[t^{\pm 1}]$.  Let $\psi \in \cE_\lambda^\Gamma$ and write $\psi = \sum_{\ell=1}^m \psi_\ell$ where $\psi_\ell \in \cE_{\lambda_\ell}^\Gamma$ and $\lambda_\ell$ is a fundamental weight for $\ell=1,\dotsc,m$.  By part~\eqref{theo-item:fund-weight-proj-property} (and Corollary~\ref{cor:constant-localized-dim}), it suffices to show that \begin{equation} \label{eq:local-Weyl-tensor-property} \ts
          \dim_\C W^\Gamma(\psi) = \prod_{\ell=1}^m \dim_\C W^\Gamma(\psi_\ell).
        \end{equation}
        By \cite[Th.~2]{CP01} (or \cite[Prop.~3.9]{FKKS12}), we are reduced to the case where the support of $\psi$ is a single $\Gamma$-orbit.  Furthermore, since the twisted local Weyl modules are restrictions of untwisted ones (as explained above), it suffices to consider the untwisted case, that is, we can assume that $\Gamma$ is trivial.  Suppose $\Supp \psi = \{a\}$ for some $a \in \C^*$.  Then, by Proposition~\ref{prop:local-Weyl-module-annihilated}, $W(\psi)$ is annihilated by $\g \otimes (t-a)^N\C[t^{\pm 1}]$ for some $N \in \N$.   We have a commutative diagram
        \[
          \xymatrix{
            \g \otimes \C[t^{\pm 1}] \ar@{->>}[r] & \g \otimes \big(\C[t^{\pm 1}]/(t-a)^N\C[t^{\pm 1}]\big) \\
            \g \otimes \C[t] \ar@{^(->}[u] \ar@{->>}[r] & \g \otimes \big(\C[t]/(t-a)^N\C[t]\big) \ar[u]^{\cong}
          }
        \]
        where the vertical arrows are induced by inclusion and the horizontal arrows are the natural projections.  Since the local Weyl modules for the current algebra $\g \otimes \C[t]$ are annihilated by $\g \otimes (t-a)^N \C[t]$ (increasing $N$ if necessary), it follows that the pullback of the local Weyl module for the loop algebra $\g \otimes \C[t^{\pm 1}]$ is the local Weyl module for the current algebra (see Lemma~\ref{lem:lWm-characterization}).  Thus it suffices to prove~\eqref{eq:local-Weyl-tensor-property} for the current algebra.  Then, by considering the automorphism of $\g \otimes \C[t]$ determined by $u \otimes f(t) \mapsto u \otimes f(t-a)$ for $f(t) \in \C[t]$, we see that it suffices to prove~\eqref{eq:local-Weyl-tensor-property} when $a=0$ (i.e.\ $\psi$ is supported at the origin).  This was proved in \cite[Th.~5]{CP01} for $\g = \mathfrak{sl}_2$ (in fact, the statement there is for the loop algebra itself), in \cite[Th.~1.5.1]{CL06} for $\g = \mathfrak{sl}_n$, in \cite[Cor.~B]{FL07} (and conjectured in~\cite{CP01b}) for simple simply laced $\g$ and in~\cite[Cor.~A]{Nao12} for arbitrary simple $\g$.  We note that the freeness of the global Weyl module in the untwisted case also follows from results found in \cite{Nak01,BN04}. \qedhere
  \end{asparaenum}
\end{proof}

%%%%%%%%%%%%%%%%%%%%%%%%%%%%%%%%%%%%%%%%%%%%%%%%%%%%%%%%%%%%%%%%%%%
%
\section{The algebra \texorpdfstring{$\bA_\lambda^\Gamma$}{acting on the weight spaces}} \label{sec:A-lambda-Gamma}
%
%%%%%%%%%%%%%%%%%%%%%%%%%%%%%%%%%%%%%%%%%%%%%%%%%%%%%%%%%%%%%%%%%%%

In this section we give an alternative, and more explicit, characterization of the algebra $\bA_\lambda^\Gamma$.  We also relate it to the corresponding algebra in the untwisted setting.  We assume in this section that the Jacobson radical $\rad A$ of $A$ (which is equal to the nilradical of $A$ since $A$ is finitely generated) is zero and that $\Gamma$ is a nontrivial cyclic group acting faithfully on $\g$ by diagram automorphisms (see Remark~\ref{rem:cyclic-assumption}).  We also continue to assume that $\Gamma$ acts freely on $X_\rat$.  In addition, we assume in this section that $\g$ is not of type $A_{2n}$.  In other words, we assume that $\Gamma$ acts by \emph{admissible} diagram automorphisms (no two nodes of the Dynkin diagram are contained in the same $\Gamma$-orbit).  In this section, for a commutative associative unital algebra $B$, we will not distinguish between a point of $\maxSpec B$ and its corresponding maximal ideal.

Since $\g$ is not of type $A_{2n}$, the restriction map $\Lambda^+ \to \Lambda_\Gamma^+$ is surjective (see Lemma~\ref{lem:height-wt}\eqref{lem-item:root-restriction}).  Recall that we can naturally identify $I_\Gamma$ with the set $I/\Gamma$ of $\Gamma$-orbits on $I$.  Fix $\lambda = \sum_{\bi \in I_\Gamma} r_\bi \omega_\bi \in \Lambda^+_\Gamma$.  For $\bi \in I_\Gamma$ and $i \in \bi$, let $r_i = r_\bi$.  Recall that, for $i \in I$, we let $\Gamma_i = \{\gamma \in \Gamma\ |\ \gamma i = i\}$ denote the corresponding isotropy subgroup.  Let $J \subseteq I$ contain one point in each orbit of the $\Gamma$-action on $I$.  We will often identify $J$ with the set $\{1,\dots,|J|\}$ using the standard labeling of Dynkin diagrams found, for instance, in \cite[\S11.4]{Hum72}.   Define
\begin{equation} \label{def:bbA} \ts
  \bbA^\lambda_\Gamma \defeq \bigotimes_{j \in J} S^{r_j} (A^{\Gamma_j}).
\end{equation}
So we have a natural identification
\begin{equation} \label{eq:AA-maxSpec} \ts
  \maxSpec \bbA^\lambda_\Gamma \cong \prod_{j \in J} \big( (\maxSpec A^{\Gamma_j})^{r_j} \big)/S_{r_j}.
\end{equation}
From now on, we will identify the two sides of~\eqref{eq:AA-maxSpec}.

Recall that $\cE^\Gamma$ is the set of $\Gamma$-equivariant, finitely-supported maps from $X_\rat = \maxSpec A$ to $\Lambda^+$.  Let $\mm \in \maxSpec \bbA^\lambda_\Gamma$.  Then $\mm$ is of the form
\[
  \mm = ((\mm_{j,\ell})_{\ell=1}^{r_j})_{j \in J},
\]
where, for $j \in J$, the (unordered) tuple $(\mm_{j,\ell})_{\ell=1}^{r_j}$ is an element of $\big( (\maxSpec A^{\Gamma_j})^{r_j} \big)/S_{r_j}$.  Hence,  each $\mm_{j,\ell}$ can be identified with an element of the quotient $(\maxSpec A)/{\Gamma_j} \cong \maxSpec A^{\Gamma_j}$, that is, with a $\Gamma_j$-orbit in $\maxSpec A$.  Then define $\psi_\mm \in \cE^\Gamma$ by
\[ \ts
  \psi_\mm(\sm) = \sum_{j \in J,\ \gamma \in \Gamma/\Gamma_j} \sum_{\ell=1}^{r_j} \delta_{\sm,\gamma \mm_{j,\ell}} \omega_{\gamma j},\quad \sm \in \maxSpec A,
\]
where $\delta_{\sm,\sm'}$, $\sm \in \maxSpec A$, $\sm' \in \maxSpec A^{\Gamma_j}$, is equal to one if $\sm \mapsto \sm'$ under the quotient map $\maxSpec A \twoheadrightarrow \maxSpec A^{\Gamma_j}$ corresponding to the inclusion $A^{\Gamma_j} \hookrightarrow A$ (i.e., if $\sm' = \sm \cap A^{\Gamma_j}$), and is equal to zero otherwise.  It is readily verified that the map $\mm \mapsto \psi_\mm$ is a well-defined bijection of sets $\maxSpec \bbA^\lambda_\Gamma \to \cE^\Gamma_\lambda$.
\details{To see that $\psi_\mm(\sm)$ is $\Gamma$ equivariant, note that, for $\tau \in \Gamma$,
\begin{gather*} \ts
  \psi_\mm(\tau \sm) = \sum_{j \in J,\ \gamma \in \Gamma/\Gamma_j} \sum_{\ell=1}^{r_j} \delta_{\tau \sm, \gamma \mm_{j,\ell}} \omega_{\gamma j} = \frac{1}{|\Gamma_j|} \sum_{j \in J,\ \gamma \in \Gamma} \sum_{\ell=1}^{r_j} \delta_{\tau \sm, \gamma \mm_{j,\ell}} \omega_{\gamma j} \\ \ts
  = \frac{1}{|\Gamma_j|} \sum_{j \in J,\ \gamma \in \Gamma} \sum_{\ell=1}^{r_j} \delta_{\sm, \tau^{-1} \gamma \mm_{j,\ell}} \omega_{\gamma j} = \frac{1}{|\Gamma_j|} \sum_{j \in J,\ \nu \in \Gamma} \sum_{\ell=1}^{r_j} \delta_{\sm, \nu \mm_{j,\ell}} \omega_{\nu \tau j} \\ \ts
  = \tau \left(\frac{1}{|\Gamma_j|} \sum_{j \in J,\ \nu \in \Gamma} \sum_{\ell=1}^{r_j} \delta_{\sm, \nu \mm_{j,\ell}} \omega_{\nu j} \right) = \tau \psi_\mm(\sm),
\end{gather*}
where, in the fourth equality, we set $\nu = \tau^{-1} \gamma$.}

Recall that for $\sm \in \maxSpec A$, the quotient $A/\sm$ is canonically isomorphic to $\kk$.  We will identify the two in what follows.  For $\psi \in \cE^\Gamma$, choose $\sM \in (\Supp \psi)_\Gamma$ and consider the composition
\begin{equation} \label{eq:hev-def} \ts
  (\h \otimes A)^\Gamma \hookrightarrow \h \otimes A \twoheadrightarrow \bigoplus_{\sm \in \sM} \left( \h \otimes (A/\sm) \right) \cong \bigoplus_{\sm \in \sM} \h \xrightarrow{\sum_{\sm \in \sM} \psi(\sm)} \kk.
\end{equation}
One readily verifies that this map does not depend on the choice of $\sM$.  It induces a map
\[
  \hev_\psi^\Gamma \colon \cU((\h \otimes A)^\Gamma) \to \kk.
\]
We use the notation $\hev_\psi^\Gamma$ to distinguish this evaluation representation of $(\h \otimes A)^\Gamma$ from the evaluation map $\ev_\psi^\Gamma$.

\begin{lem} \label{lem:Ann-in-hev}
  For $\lambda \in \Lambda^+_\Gamma$, we have $\Ann_{\cU((\h \otimes A)^\Gamma)} w_\lambda^\Gamma \subseteq \bigcap_{\psi \in \cE^\Gamma_\lambda} \ker \hev_\psi^\Gamma$.
\end{lem}

\begin{proof}
  Let $u \in \Ann_{\cU((\h \otimes A)^\Gamma)} w_\lambda^\Gamma$ and $\psi \in \cE^\Gamma_\lambda$.  Since $V^\Gamma(\psi)$ is a quotient of $W^\Gamma(\psi)$, we have $u v^\Gamma_\lambda = 0$.  But, by the definition of $V^\Gamma(\psi)$, we have $u v^\Gamma_\lambda = \hev_\psi^\Gamma(u) v^\Gamma_\lambda$.  Thus $\hev_\psi^\Gamma(u)=0$.
\end{proof}

For a $\kk$-algebra $B$ and $m \in \N$, define
\[ \ts
  \sym_m(b) \defeq \sum_{\ell=1}^m 1^{\otimes (\ell-1)} \otimes b \otimes 1^{\otimes (m-\ell)} \in S^mB,\quad b \in B.
\]
Recall that if $B$ is finitely generated, then $S^m B$ is generated by elements of the form $\sym_m(b)$, $b \in B$ (see \cite[Lem.~4.56(ii)]{EGHLSVY} or note that, since $B$ is finitely generated, this follows from the case where $B$ is a polynomial algebra in finitely many variables, in which case the result can be found in \cite[Th.~1.2]{Dal99}, but goes back to \cite{Sch1852}).  Thus the algebra $\bbA^\lambda_\Gamma$ is generated by the classes of elements of the form
\[ \ts
  \sym^j_\lambda(a) \defeq 1^{\otimes(r_1 + \dotsb + r_{j-1})} \otimes \sym_{r_j}(a) \otimes 1^{\otimes (r_{j+1} + \dotsb + r_{|J|})},\quad j \in J ,\ a \in A^{\Gamma_j}.
\]

The Lie algebra $(\h \otimes A)^\Gamma$ is spanned by elements of the form
\begin{equation} \label{eq:hAG-form} \ts
  \overline{h_j \otimes a},\quad j \in J,\ a \in A^{\Gamma_j}
\end{equation}
(recall Definition~\ref{def:overline}).  Let $\tilde \tau_\lambda \colon \cU((\h \otimes A)^\Gamma) \twoheadrightarrow \bbA^\lambda_\Gamma$ be the surjective map determined by
\begin{equation} \label{def:tau-tilde} \ts
  \tilde \tau_\lambda \left( \overline{h_j \otimes a} \right) = \sym^j_\lambda (a),\quad j \in J,\ a \in A^{\Gamma_j}.
\end{equation}

\begin{lem} \label{lem:tau-eval-composition}
  We have $\hev_{\psi_\mm}^\Gamma = \ev_\mm \circ \tilde \tau_\lambda$ for all $\mm \in \maxSpec \bbA^\lambda_\Gamma$, where $\ev_\mm \colon \bbA^\lambda_\Gamma \twoheadrightarrow \bbA^\lambda_\Gamma/\mm \cong \kk$ is the canonical projection for $\mm \in \maxSpec \bbA^\lambda_\Gamma$.
\end{lem}

\begin{proof}
  It suffices to prove that the maps agree on elements of the form~\eqref{eq:hAG-form}.  Fix $k \in J$ and $a \in A^{\Gamma_k}$.  Choose $\mm = ((\mm_{j,\ell})_{\ell=1}^{r_j})_{j \in J} \in \maxSpec \bbA^\lambda_\Gamma$.  Then
  \[ \ts
    \ev_\mm \circ \tilde \tau_\lambda ( \overline{h_k \otimes a} ) = \ev_\mm (\sym^k_\lambda(a)) = \sum_{\ell=1}^{r_k} (a + \mm_{k,\ell}) \in \kk,
  \]
  where we have canonically identified $A^{\Gamma_k}/\mm_{k,\ell} \cong \kk$ for $\ell=1,\dotsc,r_k$.

  On the other hand, choose $\sM$ in~\eqref{eq:hev-def} to contain $\{\sm_{k,\ell}\}_{\ell=1}^{r_k}$, where $\sm_{k,\ell} \mapsto \mm_{k,\ell}$ under the quotient map $\maxSpec A \twoheadrightarrow \maxSpec A^{\Gamma_k}$.  Then
  \[ \ts
    \hev_{\psi_\mm}^\Gamma ( \overline{h_k \otimes a} ) = \sum_{\ell=1}^{r_k} (a + \sm_{k,\ell}) \in \kk.
  \]
  Since, for $\ell=1,\dotsc,r_k$, the inclusion $A^{\Gamma_k} \hookrightarrow A$ induces an isomorphism $A^{\Gamma_k}/\mm_{k,\ell} \cong A/\sm_{k,\ell}$ mapping $a + \mm_{k,\ell} \mapsto a + \sm_{k,\ell}$, the proof is complete.
\end{proof}

\begin{lem} \label{lem:tau-surjective}
  We have $\Ann_{\cU((\h \otimes A)^\Gamma)} w_\lambda^\Gamma \subseteq \ker \tilde \tau_\lambda$ and thus $\tilde \tau_\lambda$ induces a surjective algebra homomorphism $\tau_\lambda \colon \bA^\lambda_\Gamma \twoheadrightarrow \bbA^\lambda_\Gamma$.
\end{lem}

\begin{proof}
  We have
  \[ \ts
    \ker \tilde \tau_\lambda = \bigcap_{\mm \in \maxSpec \bbA^\lambda_\Gamma} \ker \ev_\mm \circ \tilde \tau_\lambda = \bigcap_{\psi \in \cE_\lambda^\Gamma} \ker \hev^\Gamma_\psi.
  \]
  Indeed, the first equality follows from the fact that $\rad \bbA^\lambda_\Gamma = 0$ (since $\rad A=0$) and the second follows from Lemma~\ref{lem:tau-eval-composition} and the fact that the map $\mm \mapsto \psi_\mm$ is a bijection $\maxSpec \bbA^\lambda_\Gamma \to \cE^\Gamma_\lambda$.   The statement of the lemma then follows from Lemma~\ref{lem:Ann-in-hev}.
\end{proof}

Recall that we have a $\Xi$-grading $A = \bigoplus_{\xi \in \Xi} A_\xi$ on $A$, where $\Xi$ is the character group of $\Gamma$.  Choosing a basis for each $A_\xi$ yields a basis $\cB$ of $A$.  We do this in such a way that our basis contains the element $1 \in A$.  Since $A$ is finitely generated, the basis $\cB$ is countable.  So we can write $\cB = \{a_r\ |\ r \in \N\}$, with $a_0=1$.  We say that $a_r \le a_{r'}$ if $r \le r'$.  Since each $A^{\Gamma_j}$, $j \in J$, is a sum of isotypic components $A_\xi$, $\xi \in \Xi$, we have that $\cB_j \defeq \cB \cap A^{\Gamma_j}$ is a basis of $A^{\Gamma_j}$ for $j \in J$.

\begin{lem} \label{lem:A-lambda-Gamma-spanning-set}
  The elements
  \[ \ts
    \prod_{j \in J} \prod_{s=1}^{m_j} \overline{h_j \otimes b_{j,s}}\, w^\Gamma_\lambda, \quad b_{j,s} \in \cB_j,\ a_0 < b_{j,1} \le \dots \le b_{j,m_j},\ m_j \le r_j \text{ for } j \in J,
  \]
  span $W^\Gamma(\lambda)_\lambda$.
\end{lem}

\begin{proof}
  It suffices to prove that for all $j \in J$ and $a_{p_1},\dotsc,a_{p_\ell} \in \cB_j$ with $1 \le p_1 \le \dots \le p_\ell$, $\ell \in \N$, we have
  \[ \ts
    \prod_{s=1}^\ell \overline{h_j \otimes a_{p_s}}\, w^\Gamma_\lambda \in \Span_\kk \left\{ \left. \prod_{q=1}^m \overline{h_j \otimes a_{t_q}}\, w^\Gamma_\lambda \ \right|\ 1 \le t_1 \le \dots \le t_m,\ m \le r_j \right\}.
  \]
  For $j \in J$ and $a,a' \in \cB_j$, we have
  \begin{gather*} \ts
    \left[ \overline{e_j \otimes a}, \overline{f_j \otimes a'} \right] = \overline{h_j \otimes aa'},\\ \ts
    \left[ \overline{h_j \otimes a}, \overline{e_j \otimes a'} \right] = 2\, \overline{e_j \otimes aa'},\quad \left[ \overline{h_j \otimes a}, \overline{f_j \otimes a'} \right] = - 2\, \overline{f_j \otimes aa'}.
  \end{gather*}
  Thus, for $\ell \ge r_j + 1$, we have
  \[ \ts
    0 = \left( \prod_{s=1}^\ell \overline{e_j \otimes a_{p_s}} \right) \left( \overline{f_j \otimes 1} \right)^\ell w^\Gamma_\lambda = \prod_{s=1}^\ell \overline{h_j \otimes a_{p_s}}\, w^\Gamma_\lambda + C w^\Gamma_\lambda,
  \]
  where $C$ is a $\kk$-linear combination of elements of the form $\prod_{s=1}^m \overline{h_j \otimes a_{p_{k_s}}}$ with $m < \ell$.  The lemma follows by induction on $\ell$.
\end{proof}

By Lemma~\ref{lem:A-lambda-Gamma-spanning-set}, we see that $\bA^\lambda_\Gamma$ is spanned by the image of the set
\begin{equation} \label{eq:A-lambda-Gamma-span-set} \ts
  \left\{ \left. \prod_{j \in J} \prod_{s=1}^{m_j} \overline{h_j \otimes b_{j,s}}\ \right|\ b_{j,s} \in \cB_j,\ a_0 < b_{j,1} \le \dots \le b_{j,m_j},\ m_j \le r_j \text{ for } j \in J \right\}.
\end{equation}

We now state the first main result of this section, which gives an explicit realization of the algebra $\bA^\lambda_\Gamma$.

\begin{theo} \label{theo:Alambda-isom}
  The map $\tau_\lambda \colon \bA^\lambda_\Gamma \to \bbA^\lambda_\Gamma$ is an isomorphism of algebras.
\end{theo}

\begin{proof}
  By Lemma~\ref{lem:tau-surjective}, it suffices to show that $\tau_\lambda$ is injective.  For this, it is enough to show that the images under $\tau_\lambda$ of the elements of the set~\eqref{eq:A-lambda-Gamma-span-set} are linearly independent (over $\kk$) in $\bbA^\lambda_\Gamma$. Now, for $b_{j,s} \in \cB_j$, $s=1,\dotsc,m_j$, $a_0 < b_{j,1} \le \dots \le b_{j,m_j}$, $m_j \le r_j$, $j \in J$, we have
  \begin{equation} \label{eq:tau-lambda-formula} \ts
    \tau_\lambda \left( \prod_{j \in J} \prod_{s=1}^{m_j} \overline{h_j \otimes b_{j,s}} \right) = \prod_{j \in J} \prod_{s=1}^{m_j} \sym^j_\lambda(b_{j,s}).
  \end{equation}
  Since the tensor product of linearly independent sets is linearly independent, it suffices to prove that, for a fixed $j \in J$, the elements
  \[ \ts
    \prod_{s=1}^m \sym_{r_j}(b_s),\quad b_s \in \cB_j,\ s=1,\dotsc,m,\ a_0 < b_1 \le \dots \le b_m,\ m \le r_j,
  \]
  are linearly independent elements of $S^{r_j}(A^{\Gamma_j})$.  Consider a linear combination of distinct elements of this set equal to zero:
  \begin{equation} \label{eq:lin-comb-syms} \ts
    \sum_{t=1}^N \left( c_t \prod_{s=1}^{m_t} \sym_{r_j} (b_{s,t}) \right) = 0
  \end{equation}
  for some $b_{s,t} \in \cB_j$, $t=1,\dotsc,N$, $s=1,\dotsc,m_t$, $a_0 < b_{1,t} \le \dotsb \le b_{m_t,t}$, $m_t \le r_j$ and $c_1,\dotsc,c_N \in \kk$.  Choose $\ell \in \{1,\dotsc,r_j\}$.  Let $A^{\Gamma_j}_+ \defeq \Span_\kk \{b \in \cB_j\ |\ b \ne 1\} \subsetneq A^{\Gamma_j}$.  Applying the projection $S^{r_j}(A^{\Gamma_j}) \twoheadrightarrow (A^{\Gamma_j}_+)^{\otimes \ell} \otimes 1^{\otimes(r_j-\ell)}$ to both sides of~\eqref{eq:lin-comb-syms} gives
  \[ \ts
    \sum_{1 \le t \le N,\ m_t=\ell} c_t \sum_{\varsigma \in S_{\ell}} \varsigma \left( b_{1,t} \otimes b_{2,t} \otimes \dotsb \otimes b_{\ell,t} \otimes 1^{\otimes(r_j-\ell)} \right) = 0,
  \]
  where we view $S_\ell$ as a subgroup of $S_n$ in the natural way (i.e., permuting the first $\ell$ elements). Since $b_{1,t},\dots,b_{\ell,t}$ are elements of a basis of $A$ (for any $t$), the elements
  \[ \ts
    \sum_{\varsigma \in S_{\ell}} \varsigma \left( b_{1,t} \otimes b_{2,t} \otimes \dotsb \otimes b_{\ell,t} \otimes 1^{\otimes(r_j-\ell)} \right)
  \]
  appearing above are linearly independent.  Thus $c_t=0$ for all $t=1,\dots,N$.
\end{proof}

In the remainder of this section, we provide an alternative description of $\bbA^\lambda_\Gamma$ in terms of coinvariants that does not depend on the choice $J$ of one element from each $\Gamma$-orbit in $I$.  For an algebra $B$ with the action (by automorphisms) of a finite group $\Upsilon$, we define $B_{(\Upsilon)}$ to be the ideal of $B$ generated by the set $\{b - \gamma b\ |\ b \in B,\ \gamma \in \Upsilon\}$ and let $B_\Upsilon \defeq B/B_{(\Upsilon)}$ be the algebra of coinvariants.  We hope this causes no confusion with the notation $\bbA^\lambda_\Gamma$, which is not, a priori, the algebra of coinvariants of $\bbA^\lambda$ (but see Theorem~\ref{theo:AlamGam-coinvariants}).  Note that if $\Upsilon$ is abelian, then $B_{(\Upsilon)}$ is the ideal of $B$ generated by $\bigoplus_{\xi \in \Xi,\, \xi \ne 0} B_\xi$, where $\Xi$ is the character group of $\Upsilon$.

\begin{lem} \label{lem:transitive-group-action-on-tensor-prod}
  Suppose that $\Upsilon$ is a finite group acting on a commutative unital $\kk$-algebra $B$ by algebra automorphisms and simply transitively on a finite set $Z$.  Consider the action of $\Upsilon$ on $B' \defeq \bigotimes_{z \in Z} B_z$, where $B_z=B$ for all $z \in Z$, determined by
  \[ \ts
    \gamma \left(\bigotimes_{z \in Z} b_z\right) = \bigotimes_{z \in Z} \gamma b_{\gamma^{-1}z},\quad \gamma \in \Upsilon,\ b_z \in B, z \in Z.
  \]
  Then $(B')_\Upsilon \cong B$.
\end{lem}

\begin{proof}
  Label the elements of $\Upsilon$ so that we have $\Upsilon = \{\tau_1,\dots,\tau_n\}$, with $\tau_1$ being the identity element of $\Upsilon$.  Choose $z_1 \in Z$, and set $z_i = \tau_i z_1$ for $i=1,\dotsc,n$.  So $Z=\{z_1,\dots,z_n\}$ and we have a natural action of $\Upsilon$ on the set $\{1,\dots,n\}$ by defining $\gamma i = j$ if $\gamma z_i = z_j$ (equivalently, if $\gamma \tau_i = \tau_j$) for $\gamma \in \Upsilon$ and $i \in \{1,\dots,n\}$.

  Consider the algebra homomorphism determined by
  \[ \ts
    \varpi \colon B' \to B,\quad \bigotimes_{i=1}^n b_i \mapsto \prod_{i=1}^n \tau_i^{-1} b_i,\quad b_i \in B_{z_i}.
  \]
  Clearly $\varpi$ is surjective and so it remains to show that $\ker \varpi = (B')_{(\Upsilon)}$.  For all $b' = \bigotimes_{i=1}^n b_i$ (recalling that elements of this form span $B'$), we have
  \[ \ts
    \varpi (\gamma b') = \varpi \left( \bigotimes_{i=1}^n \gamma b_{\gamma^{-1}i} \right) = \prod_{i=1}^n \tau_i^{-1} \gamma b_{\gamma^{-1}i} = \prod_{j=1}^n \tau_j^{-1} b_j = \varpi (b'),
  \]
  where, in the third equality, we have changed the index by setting $j = \gamma^{-1}i$. Thus $(B')_{(\Upsilon)} \subseteq \ker \varpi$.

  Now suppose $b' \in \ker \varpi$.  Then we can write
  \[ \ts
    b' = \sum_{j=1}^\ell ( b_{1,j} \otimes \dotsb \otimes b_{n,j} ),\quad b_{i,j} \in B_{z_i} \text{ for all } 1 \le i \le n,\ 1 \le j \le \ell.
  \]
  It is straightforward to verify that
  \[ \ts
    b_{1,j} \otimes \dotsb \otimes b_{n,j} \equiv \left( \prod_{i=1}^n \tau_i^{-1}b_{i,j} \right) \otimes 1 \otimes \dotsb \otimes 1 \mod B_{(\Upsilon)}.
  \]
  \details{
    For each $j=1,\dots,n$, the element $b_{1,j} \otimes \dotsb \otimes b_{n,j}$ is equal to the telescoping sum
    \begin{gather*} \ts
       \left( \prod_{i=1}^n \tau_i^{-1}b_{i,j} \right) \otimes 1 \otimes \dotsb \otimes 1 \\
       \ts + \left( \prod_{i=1}^{n-1} \tau_i^{-1} b_{i,j} \right) \otimes 1 \otimes \dotsb \otimes 1 \otimes b_{n,j} - \left( \prod_{i=1}^n \tau_i^{-1}b_{i,j} \right) \otimes 1 \otimes \dotsb \otimes 1 \\
       \ts + \left( \prod_{i=1}^{n-2} \tau_i^{-1} b_{i,j} \right) \otimes 1 \otimes \dotsb \otimes 1 \otimes b_{n-1,j} \otimes b_{n,j} - \left( \prod_{i=1}^{n-1} \tau_i^{-1} b_{i,j} \right) \otimes 1 \otimes \dotsb \otimes 1 \otimes b_{n,j} \\
       \ts \dotsb + b_{1,j} \otimes \dotsb \otimes b_{n,j} - (b_{1,j} \tau_2^{-2}b_{2,j}) \otimes 1 \otimes b_{3,j} \otimes \dotsb \otimes b_{n,j},
    \end{gather*}
    which, in turn, is equal to
    \begin{gather*} \ts
       \left( \prod_{i=1}^n \tau_i^{-1}b_{i,j} \right) \otimes 1 \otimes \dotsb \otimes 1 \\
       \ts + \Big(\left( \prod_{i=1}^{n-1} \tau_i^{-1} b_{i,j} \right) \otimes 1 \otimes \dotsb \otimes 1\Big) \Big(1 \otimes \dotsb \otimes 1 \otimes b_{n,j} - (\tau_n^{-1}b_{n,j})\otimes 1 \otimes \dotsb \otimes 1\Big) \\
       \ts + \Big(\left( \prod_{i=1}^{n-2} \tau_i^{-1} b_{i,j} \right) \otimes 1 \otimes \dotsb \otimes 1 \otimes b_{n,j} \Big) \Big(1 \otimes \dotsb \otimes 1 \otimes b_{n-1,j} \otimes 1 - ( \tau_{n-1}^{-1} b_{{n-1},j}) \otimes 1 \otimes \dotsb \otimes 1 \Big) \\
       \ts \dotsb + \Big( b_{1,j} \otimes 1 \otimes b_{3,j} \dotsb \otimes b_{n,j} \Big) \Big( 1 \otimes b_{2,j} \otimes 1 \otimes \dotsb \otimes 1 - \tau_2^{-2}b_{2,j} \otimes 1 \otimes \dotsb \otimes 1 \Big).
    \end{gather*}
    Each of the right-hand factors in the lines above (except the first line) are elements of the form $b'' - \tau b''$ for some $b'' \in B'$ and $\tau \in \Upsilon$.
  }
  Therefore,
  \[ \ts
    b' \equiv \left(\sum_{j=1}^\ell \prod_{i=1}^n \tau_i^{-1}b_{i,j} \right) \otimes 1 \otimes \dotsb \otimes 1 \equiv \varpi(b') \otimes 1 \otimes \dotsb \otimes 1 \equiv 0 \mod B_{(\Upsilon)}. \qedhere
  \]
\end{proof}

\begin{lem} \label{lem:coinvariant-reduction}
  Suppose that $\Upsilon$ is a finite cyclic group acting on a finitely generated commutative associative unital $\kk$-algebra $B$ by algebra automorphisms in such a way that the induced action of $\Upsilon$ on $\maxSpec B$ is free.  Fix a positive integer $n$ and consider the diagonal action of $\Upsilon$ on $B^{\otimes n}$.  This action commutes with the natural action of $S_n$ on $B^{\otimes n}$ and thus we have an induced action of $\Upsilon$ on $S^n B$.
  \begin{enumerate}
    \item If the order of $\Upsilon$ does not divide $n$, then $(S^n B)_\Upsilon = 0$.
    \item If $n = m |\Upsilon|$ for some positive integer $m$ and $(S^n B)_\Upsilon$ is reduced, then $(S^n B)_\Upsilon \cong S^m(B^\Upsilon)$.
  \end{enumerate}
\end{lem}

\begin{proof}
  \begin{asparaenum}
    \item \label{lem-item:coinvariant-reduction-zero} We have a bijection
      \[ \ts
        \maxSpec (S^n B)_\Upsilon \cong \left( \left( \prod_{i=1}^n \maxSpec B \right)/S_n \right)^\Upsilon.
      \]
      In other words, the maximal ideals of $(S^n B)_\Upsilon$ can be identified with $\Upsilon$-invariant unordered $n$-tuples of maximal ideals of $B$.  Therefore, they are unions of $\Upsilon$-orbits on the set $\maxSpec B$.  Since this action is free, $(S^n B)_\Upsilon$ has no maximal ideals if $n$ is not divisible by the order of $\Upsilon$.

    \item Let $\ell = |\Upsilon|$ and assume $n=m \ell$ for some positive integer $m$.  Recall that for any $\kk$-algebra $C$ with an $\Upsilon$-action, we have the induced grading $C=\bigoplus_{\xi \in \Xi}C_\xi$, where $\Xi$ is the character group of $\Upsilon$.  Let $\sigma$ be a generator of $\Upsilon$ and consider the map
      \[
        B^{\otimes \ell} \to B,\quad b_1 \otimes \dotsb \otimes b_\ell \mapsto b_1 \sigma(b_2) \dotsm \sigma^{\ell-1}(b_\ell),
      \]
      extended by linearity.  Let $\Phi'$ denote the restriction of this map to $S^\ell B$.  One easily checks that $\Phi'$ is $\Upsilon$-invariant and $\Phi'(S^\ell B) = B_0$.  Thus,
      \begin{equation} \label{eq:Phi'-vanish}
        \Phi'((S^\ell B)_\xi) = 0 \quad \text{for all } \xi \ne 0.
      \end{equation}
      We have the induced surjective map
      \[
        \Phi \colon (S^\ell B)^{\otimes m} \xrightarrow{(\Phi')^{\otimes m}} (B_0)^{\otimes m}.
      \]
      Restricting $\Phi$ to $S^n B$ gives a surjective map
      \[
        \varphi \colon S^n B \twoheadrightarrow S^m(B_0).
      \]
      Now,
      \begin{equation} \label{eq:Bn-decomp} \ts
        B^{\otimes n} \cong (B^{\otimes n})_0 \oplus B', \quad \text{where } B' = \bigoplus_{\xi \in \Xi,\, \xi \ne 0} (B^{\otimes n})_\xi.
      \end{equation}
      Each summand in the decomposition~\eqref{eq:Bn-decomp} is preserved by the action of $S_n$.  Thus,
      \[
        S^n B \cong ((B^{\otimes n})_0)^{S_n} \oplus (B')^{S_n}.
      \]
      Now,
      \[ \ts
        (B')^{S_n} \subseteq \bigoplus_{\xi_1 + \dotsb + \xi_m \ne 0} (S^\ell B)_{\xi_1} \otimes \dotsb \otimes (S^\ell B)_{\xi_m}.
      \]
      Thus it follows from~\eqref{eq:Phi'-vanish} that $\varphi((B')^{S_n})=0$.  So $\varphi$ vanishes on the ideal of $S^n B$ generated by $(B')^{S_n}$, which is precisely $(S^n B)_{(\Upsilon)}$.  Therefore, $\varphi$ induces a surjective map of algebras
      \[
        \bar \varphi\colon (S^n B)_\Upsilon \twoheadrightarrow S^m(B_0) = S^m(B^\Upsilon).
      \]
      Applying the functor $\Spec$ gives a morphism of schemes
      \[
        \Spec \bar \varphi \colon \Spec S^m(B^\Upsilon) \to \Spec (S^n B)_\Upsilon
      \]
      which is a closed immersion.  Now,
      \begin{gather*}
        \maxSpec S^m(B^\Upsilon) \cong (\maxSpec B^\Upsilon)^m/S_m \cong ((\maxSpec B)/\Upsilon)^m/S_m, \quad \text{and} \\
        \maxSpec (S^n B)_\Upsilon \cong (\maxSpec S^n B)^\Upsilon \cong ((\maxSpec B)^n/S_n)^\Upsilon.
      \end{gather*}
      The map $\Spec \bar \varphi$ induces a bijection between these two sets.  Namely, it maps the element of $\maxSpec S^m(B^\Upsilon)$ corresponding to an (unordered) $m$-tuple of $\Upsilon$-orbits on $\maxSpec B$ to the union (counting multiplicity) of these orbits, which is an $\Upsilon$-invariant $n$-tuple of $\maxSpec B$. In particular, $\Spec \bar \varphi$ is surjective on maximal ideals.  Thus $\ker \bar \varphi$ is included in the intersection of all the maximal ideals of $(S^n B)_\Upsilon$.  Now, since $B$ is finitely generated, so is $B^{\otimes n}$, hence so is the fixed point algebra $S^n B = (B^{\otimes n})^{S_n}$, and thus so is the quotient $(S^n B)_\Upsilon$.  Therefore, the intersection of all the maximal ideals of $(S^n B)_\Upsilon$ is equal to the nilradical of $(S^n B)_\Upsilon$, which is zero by our assumption that $(S^n B)_\Upsilon$ is reduced.   Thus, $\bar \varphi$ is injective and hence an isomorphism. \qedhere
\end{asparaenum}
\end{proof}

Define
\begin{equation} \ts
  \bbA^\lambda \defeq  \bigotimes_{i \in I} S^{r_i |\Gamma_i|} A.
\end{equation}
The diagonal action of $\Gamma$ on $A^{\otimes r_i|\Gamma_i|}$ induces an action on $S^{r_i |\Gamma_i|}A$ for each $i \in I$.  Then $\Gamma$ acts on $\bbA^\lambda$ via
\[ \ts
  \gamma \left( \bigotimes_{i \in I} a_i \right) = \bigotimes_{i \in I} \gamma a_{\gamma^{-1}i},\quad \gamma \in \Gamma,\ a_i \in S^{r_i|\Gamma_i|}A,\ i \in I.
\]

\begin{theo} \label{theo:AlamGam-coinvariants}
  If $(S^{r_i|\Gamma_i|}A)_{\Gamma_i}$ is reduced for all $i \in I$, then
  \[
    \bbA^\lambda_\Gamma \cong (\bbA^\lambda)_\Gamma.
  \]
\end{theo}

\begin{proof}
  For $i \in I$, let $B_i = S^{r_i |\Gamma_i|}A$.  Since the action of $\Gamma$ preserves each factor $\bigotimes_{i \in \bi} B_i$ in $\bbA^\lambda = \bigotimes_{\bi \in I_\Gamma} \bigotimes_{i \in \bi} B_i$, it suffices to prove the theorem for the case where $\lambda = r_\bi \omega_\bi$ for some $\bi \in I_\Gamma$.  Let $j \in J$ be the point in the $\Gamma$-orbit $\bi$ that we chose in our definition of $\bbA^\lambda_\Gamma$.  Since $\Gamma$ is commutative, we have $\Gamma_i = \Gamma_j$ for all $i \in \bi$. So we have $\bbA^\lambda_\Gamma = S^{r_j}(A^{\Gamma_j})$ and $\bbA^\lambda = \bigotimes_{i \in \bi} B_i$.

  By Lemma~\ref{lem:coinvariant-reduction}, we have
  \[
    \bbA^\lambda_\Gamma = S^{r_j}(A^{\Gamma_j}) \cong (B_j)_{\Gamma_j}.
  \]
  Consider the composition
  \begin{equation} \label{eq:double-surjection} \ts
    \bbA^\lambda = \bigotimes_{i \in \bi} B_i \stackrel{\varpi'}{\twoheadrightarrow} \left( \bigotimes_{i \in \bi} B_i \right)_{\Gamma_j} \cong \bigotimes_{i \in \bi} (B_i)_{\Gamma_j} \stackrel{\varpi}{\twoheadrightarrow} (B_j)_{\Gamma_j},
  \end{equation}
  where $\varpi'$ is the natural projection and the third map is the map $\varpi$ of the proof of Lemma~\ref{lem:transitive-group-action-on-tensor-prod}, with $\Upsilon \defeq \Gamma/\Gamma_j$ and $Z \defeq \bi$.  Since $\varpi$ and $\varpi'$ are both surjective, it suffices to show that the kernel of the above composition is $(\bbA^\lambda)_{(\Gamma)}$.  We know from the proof of Lemma~\ref{lem:transitive-group-action-on-tensor-prod} that the kernel of $\varpi$ is $\left( \bigotimes_{i \in \bi} (B_i)_{\Gamma_j} \right)_{(\Gamma/\Gamma_j)}$, which is isomorphic to $\left( \left( \bigotimes_{i \in \bi} B_i \right)_{\Gamma_j} \right)_{(\Gamma/\Gamma_j)}$ under the isomorphism in~\eqref{eq:double-surjection}. Thus it suffices to show that
  \begin{equation} \label{eq:varpi'-equality} \ts
    \left( \bigotimes_{i \in \bi} B_i \right)_{(\Gamma)} = (\varpi')^{-1}  \left( \left( \left( \bigotimes_{i \in \bi} B_i \right)_{\Gamma_j} \right)_{(\Gamma/\Gamma_j)} \right).
  \end{equation}
  But this follows from the fact that, for $b \in \bigotimes_{i \in \bi} B_i$ and $\gamma \in \Gamma$, we have $\varpi'(b-\gamma b) = \varpi'(b) - \bar \gamma \varpi'(b)$, where $\bar \gamma$ denotes the image of $\gamma$ in $\Gamma/\Gamma_j$.
\end{proof}

\begin{lem} \label{lem:symmetric-Laurent-polys}
  For $m \in \N$, we have
  \[
    \kk[t_1^{\pm 1},\dotsc,t_m^{\pm 1}]^{S_m} = \kk[e_1,\dotsc,e_m,e_m^{-1}],
  \]
  where $e_\ell$ denotes the $\ell$-th elementary symmetric polynomial for $\ell=1,\dotsc,m$.  Here the action of $S_m$ is by permutation of the variables $t_1,\dotsc,t_m$.
\end{lem}

\begin{proof}
  We clearly have $\kk[e_1,\dotsc,e_m,e_m^{-1}] \subseteq \kk[t_1^{\pm 1},\dotsc,t_m^{\pm 1}]^{S_m}$, so it remains to prove the reverse inclusion.  Suppose
  \[ \ts
    y = \sum_{i_1,\dotsc,i_m} c_{i_1,\dotsc,i_m} t_1^{i_1} \dotsb t_m^{i_m} \in \kk[t_1^{\pm 1},\dotsc,t_m^{\pm 1}]^{S_m}.
  \]
  Since all but a finite number of the $c_{i_1,\dotsc,i_m}$ are zero, we can find a positive integer $N$ such that $i_\ell + N \ge 0$ for all $\ell=1,\dotsc,m$ whenever $c_{i_1,\dotsc,i_m} \ne 0$.  Then
  \[ \ts
    y = (t_1 t_2 \dotsm t_m)^{-N} \sum_{i_1,\dotsc,i_m} c_{i_1,\dotsc,i_m} t_1^{i_1+N} \dotsb t_m^{i_m+N} \in e_m^{-N} \kk[t_1,\dots,t_m]^{S_m} \subseteq \kk[e_1,\dotsc,e_m,e_m^{-1}]. \qedhere
  \]
\end{proof}

\begin{cor}
  In the case of a twisted loop algebra, where $A = \kk[t,t^{-1}]$ and the generator $\sigma$ of $\Gamma$ acts on $X_\rat \cong \kk^\times$ by multiplication by a primitive $|\Gamma|$-th root of unity, then $\bbA^\lambda_\Gamma \cong (\bbA^\lambda)_\Gamma$.
\end{cor}

\begin{proof}
  In this case, for $i \in I$, we have, by Lemma~\ref{lem:symmetric-Laurent-polys},
  \[
    S^{r_i |\Gamma_i|}A \cong \kk [ e_1,\dotsc,e_{r_i |\Gamma_i|},e_{r_i |\Gamma_i|}^{-1} ] \subseteq \kk[t_1^{\pm 1},\dotsc,t_{r_i |\Gamma_i|}^{\pm 1}] \cong A^{\otimes r_i |\Gamma_i|}.
  \]
  Viewing the above isomorphisms as identifications, it is easily seen that $(S^{r_i |\Gamma_i|}A)_{(\Gamma_i)}$ is the subring of $S^{r_i |\Gamma_i|}A$ generated by the $e_\ell$ for $\ell$ not a multiple of $|\Gamma_i|$.  Thus
  \[
    (S^{r_i |\Gamma_i|}A)_{\Gamma_i} = \kk[e_{|\Gamma_i|},e_{2|\Gamma_i|},\dots,e_{r_i|\Gamma_i|},e_{r_i |\Gamma_i|}^{-1}]
  \]
  is reduced and the result follows from Theorem~\ref{theo:AlamGam-coinvariants}.
\end{proof}

%%%%%%%%%%%%%%%%%%%%%%%%%%%%%%%%%%%%%%%%%%%%%%%%%%%%%%%%%%%%%%%%%%%
%
\appendix
\section{Some results in commutative algebra} \label{sec:appendix}
%
%%%%%%%%%%%%%%%%%%%%%%%%%%%%%%%%%%%%%%%%%%%%%%%%%%%%%%%%%%%%%%%%%%%%

This appendix contains some technical results from commutative algebra that are used in the proof of Theorem~\ref{theo:global-Weyl-module-projective}.

\begin{defin}[Rank]
  If $R$ is an integral domain and $X$ is a finitely generated $R$-module, we define the \emph{rank} of $X$ to be $\rank X \defeq \dim_{S^{-1}R}(S^{-1}X)$, where $S=R\setminus\left\{0\right\}$.
\end{defin}

\begin{prop} \label{prop:app-local-rank}
  Suppose that $O$ is a Noetherian local domain with unique maximal ideal $\sm$. Let $Y$ be a finitely generated $O$-module. Then
  \[
    \dim_{O/\sm} Y/\sm Y\ge \rank Y,
  \]
  and if equality holds, then $Y$ is free over $O$ with rank equal to $\rank Y$.
\end{prop}

\begin{proof}
  Set $n=\dim_{O/\sm} Y/\sm Y$ and choose $y_1,\ldots,y_n\in Y$ whose images are a basis  of $Y/\sm Y$ over $O/\sm$. By Nakayama's Lemma, the elements $y_1,\ldots,y_n$ generate $Y$ over $O$. Then $y_1,\ldots,y_n$ span $S^{-1}Y$ over $S^{-1}O$, where $S=O\setminus\left\{0\right\}$, which proves the first part of the proposition. If equality holds, then $y_1,\ldots,y_n$ are linearly independent over $S^{-1}O$, and hence linearly independent over $O$; that is, $Y$ is free over $O$ with basis $y_1,\ldots,y_n$.
\end{proof}

\begin{cor}\label{cor:appendix}
  Let $R$ be a Noetherian integral domain, $X$ be a finitely generated $R$-module and $\sm$ be a maximal ideal of $R$. Then
  \[
    \dim_{R/\sm} X/\sm X\ge\rank X,
  \]
  and if equality holds, then $X_{\sm}$ is free as a module over $R_{\sm}$, with rank equal to $\rank X$.
\end{cor}

\begin{proof}
  The localization $R_\sm$ is a local ring with a unique maximal ideal $\mathsf{n}$.  Then $\rank X_\sm=\rank X$, and moreover
  \[
    \dim_{R/\sm} X/\sm X = \dim_{R_\sm/\mathsf{n}}  X_\sm/\mathsf{n} X_\sm,
  \]
  since a basis $x_1+\sm X,\ldots x_k+\sm X$ of $X/\sm X$ over $R/\sm$ passes to a basis $x_1+\mathsf n X_\sm,\ldots,x_k+\mathsf n X_\sm$ of $X_\sm/\mathsf n X_\sm$ over $R_\sm/\mathsf n$.
  The corollary then follows by applying Proposition~\ref{prop:app-local-rank} with $O=R_\sm$ and $Y=X_\sm$.
\end{proof}

\begin{lem}\label{lem:appendix}
  Let $R$ be an integral domain and $X$ be a finitely generated $R$-module. Then there is a nonempty open subset $U \subseteq \Spec R$ such that for each maximal ideal $\sm\in U$, we have $\dim_{R/\sm} X/\sm X=\rank X.$
\end{lem}

\begin{proof}
  Set $S=R\setminus\left\{0\right\}$ and choose a basis $\frac{x_1}{a_1},\dotsc,\frac{x_n}{a_n}$ of $S^{-1}X$ over $S^{-1}R$, where $n=\rank X$, $x_i\in X$ and $a_i\in S$ for $i=1,\ldots,n$. Rescaling gives a basis $x_1,\ldots,x_n$ of $S^{-1}X$ over $S^{-1}R$.

  Now choose a finite set $\{g_1,\ldots,g_N\}$ of generators of $X$ as an $R$-module.  For  $j=1,\ldots,N$, we can consider $g_j$ as an element of $S^{-1}X$ and thus write it in the form
  \[
    g_j=\textstyle\sum_{i=1}^n\frac{a_{ij}}{b_j}x_i \quad \text{for some } b_j \in S,\ a_{ij} \in R,\ 1 \le i \le n,
  \]
  where we found a common denominator for the coefficients of the $x_i$.  Thus,
  \[
    g_j\in\textstyle\sum_{i=1}^n R_{b_j}x_i,
  \]
  where $R_{b_j}$ denotes the localization of $R$ at the multiplicative set generated by $b_j$. Let $T$ denote the multiplicative subset of $R$ generated by $b_1,\ldots,b_N$. Thus, $X_T$ is generated by $x_1,\ldots,x_n$ over $T^{-1}R$.  Since $x_1,\ldots,x_n$ are linearly independent over $S^{-1}R$, they are also independent over $T^{-1}R$, and thus $X_T\cong R_{T}^{\oplus n}$.

  We claim it now follows that $U=\bigcap_{j=1}^N D_{b_j}$ satisfies the condition in the statement of the lemma, where $D_{b_j}=\left\{\mathsf{p}\in\Spec(R)\ |\ b_j \notin \mathsf{p} \right\}$ is the basic open set associated to $b_j$. To see this, it suffices to show that for each maximal ideal $\sm\in U$, the set $\cB \defeq \left\{ x_i+\sm X \right\}_{i=1}^n$ is a basis of $X/\sm X$ over $R/\sm$.

  We first show that $\cB$ is linearly independent. Suppose that
  \begin{equation} \label{appendix:lincomb} \ts
    0=\sum_{i=1}^n(r_i+\sm)(x_i+\sm X),\quad r_i \in R,\ x_i \in X,
  \end{equation}
  that is, $\sum_{i=1}^n r_i x_i = \sum_{j=1}^k m_j y_j$ for some $m_j \in \sm$ and $y_j\in X$.  Now, in $X_T$ we have, for $j=1,\dotsc,k$,
  \[ \ts
    y_j = \sum_{i=1}^n\frac{s_{i,j}}{t_{i,j}} x_i \quad \text{for some } t_{i,j}\in T,\ s_{i,j} \in R.
  \]
  Thus,
  \[ \ts
    \sum_{j=1}^k \sum_{i=1}^n\frac{s_{i,j} m}{t_{i,j}} x_i = \sum_{j=1}^k m_jy_j = \sum_{i=1}^n r_i x_i,
  \]
  and so, by the independence of the $x_i$ over $T^{-1}R$, we have $\sum_{j=1}^k \frac{s_{i,j} m}{t_{i,j}} = r_i$ for each $i$.  Therefore, $r_i \prod_{j=1}^k t_{i,j}\in\sm$. Since $\sm\in U$, we have $b_j\in R\setminus \sm$ for each $1\le j\le N$.  Thus $T \cap \sm = \varnothing$.  Hence $\prod_{j=1}^k t_{i,j} \notin\sm$, and so $r_i\in\sm$.  Thus the linear combination \eqref{appendix:lincomb} is trivial.

  Finally, we show that $\cB$ spans $X/\sm X$ over $R/\sm$.  Suppose $x+\sm X\in X/\sm X$. In $X_T$,
  again we write
  \[ \ts
    x = \sum_{i=1}^n\frac{s_i}{t_i} x_i \quad \text{for some } t_i \in T,\ s_i \in R.
  \]
  Setting $t=\prod_{i=1}^n t_i$ and $t_i'=\prod_{j\ne i} t_j$, we have
  \[ \ts
    tx = \sum_{i=1}^n\ s_it_i' x_i.
  \]
  Since $\sm$ is maximal, there exists $r\in R$ such that $1-rt\in\sm$, and we have
  \[
    trx=\textstyle\sum_{i=1}^n s_it_i'r x_i.
  \]
  Reducing modulo $\sm$, it follows that
  \[
    x+\sm X=(1+\sm)x+\sm X=(rt+\sm)x+\sm X =\textstyle\sum_{i=1}^n (s_it_i'r+\sm)(x_i+\sm X),
  \]
  which shows that $\cB$ is a spanning set and completes the proof.
\end{proof}

\begin{prop}\label{prop2:appendix}
  Let $R$ be a finitely generated algebra over a field such that $R$ is an integral domain and let $X$ be a finitely generated $R$-module. Suppose that there exists $n \in \N$ such that $\dim_{R/\sm} X/\sm X=n$ for all maximal ideals $\sm$ of $R$. Then $X_{\sm}$ is a free $R_{\sm}$-module of rank $n$ for all maximal ideals $\sm$ of $R$.
\end{prop}

\begin{proof}
  By Lemma~\ref{lem:appendix}, there is a nonempty open set $U \subseteq \Spec R$  with the property that, for each maximal ideal $\sm\in U$, we have $\dim_{R/\sm} X/\sm X=\rank X$.  Since $R$ is finitely generated, the maximal ideals of $R$ are dense in $\Spec (R)$, and so there is at least one maximal ideal in $U$.
  \details{Since $R$ is a finitely generated algebra over a field, every prime ideal of $R$ is an intersection of maximal ideals by \cite[Ch.~V, \S4, Th.~3]{Bou85b}. In particular each prime ideal is the intersection of maximal ideals containing it. It follows that for a prime ideal $\mathsf p$ contained in a basic open set $D_f$ for some $f\in R$, there is a maximal ideal containing $\mathsf p$ which lies in $D_f$. Therefore, the maximal ideals are dense in $\Spec(R)$.}
  Thus $\dim_{R/\sm} X/\sm X=\rank X$ for all maximal ideals $\sm$. Applying Corollary~\ref{cor:appendix} completes the proof.
\end{proof}

\begin{theo}[{\cite[Chap.~II, \S5, no.~2, Th.~1]{Bou85a}}] \label{theo-bourbaki:appendix}
  Let $R$ be a commutative ring and $P$ be an $R$-module. Then $P$ is a finitely generated projective module if and only if $P$ is a finitely presented module and, for every maximal ideal $\sm$ of $R$, $P_{\sm}$ is a free $R_{\sm}$-module.
\end{theo}

\begin{cor} \label{cor:constant-localized-dim}
  Suppose that $R$ is a finitely generated algebra over a field such that $R$ is an integral domain and suppose that $P$ is a finitely generated $R$-module.  Furthermore, assume that there exists $n \in \N$ such that $\dim_{R/\sm} P/\sm P=n$ for all maximal ideals $\sm$ of $R$.  Then $P$ is projective.  Moreover, if $R$ is a generalized Laurent polynomial ring $\kappa [t_1^{\pm 1},\dotsc,t_\ell^{\pm 1},s_1,\dotsc,s_m]$, $\ell,m \in \N$, over a field $\kappa$, then $P$ is free of rank $n$ over $R$.
\end{cor}

\begin{proof}
  The first part of the corollary follows immediately from Proposition~\ref{prop2:appendix} and Theorem~\ref{theo-bourbaki:appendix}.  The second part follows from the Quillen--Suslin Theorem (see, e.g., \cite[Cor. V.4.10]{Lam06}).
\end{proof}

%%%%%%%%%%%%%%%%%%%%%%%%%%%%%%%%%%%%%%%%%%%%%%%%%%%%%%%%%%%%%%%%%%%
% References
%%%%%%%%%%%%%%%%%%%%%%%%%%%%%%%%%%%%%%%%%%%%%%%%%%%%%%%%%%%%%%%%%%%

\bibliographystyle{alpha}
\bibliography{global-weyl-modules-biblist}

\end{document}